\documentclass[]{article}

\expandafter\let\csname equation*\endcsname\relax
\expandafter\let\csname endequation*\endcsname\relax

\usepackage{amsthm}
\usepackage{amsmath}
\usepackage{amssymb}
\usepackage{graphicx}
\usepackage{subfigure}
\usepackage{enumerate}
\usepackage{multirow}
\usepackage{authblk}

\usepackage{xcolor}
\usepackage[normalem]{ulem}

\usepackage{cancel}

\newcommand{\pmm}{+\!/\!-}

\newcommand{\RR}{\mathbb{R}}
\newcommand{\X}{\mathcal{X}}
\newcommand{\U}{\mathcal{U}}

\newcommand{\T}{\mathbb{T}_T}
\newcommand{\Tu}{\mathbb{T}}

\newcommand{\tTT}{\tilde{\mathcal{T}}}

\newcommand{\tO}{\tilde{\mathcal{O}}}

\newcommand{\tS}{\tilde{\Sigma}}

\newcommand{\tJ}{J_4}
\newcommand{\tz}{\tilde{z}}
\newcommand{\tL}{\tilde{\Lambda}}
\newcommand{\tC}{\tilde{\mathcal{C}}}
\newcommand{\tG}{\tilde{\Gamma}}
\newcommand{\tg}{\tilde{\gamma}^{\text{up}}}
\newcommand{\bt}{\bar{\zeta}}

\newcommand{\tk}{\tilde{k}}
\newcommand{\te}{\tilde{\delta}}
\newcommand{\tw}{\tilde{\omega}}

\newcommand{\impactwzz}[2]{\left(0,v^{#1}_{#2},x^{#1}_{#2},y^{#1}_{#2},s^{#1}_{#2} \right)}

\newcommand{\sgn}{\text{sgn}}
\newcommand{\tp}{\tilde{\phi}}
\newcommand{\condsC}{C.1--C.4}

\newcommand{\hlambda}{\hat{\lambda}}

\newcommand{\blambda}{\bar{\lambda}}
\newcommand{\bmu}{\bar{\mu}}

\newcommand{\sd}{\sigma^{\text{down}}}
\newcommand{\su}{\sigma^{\text{up}}}
\newcommand{\Su}{S^{\text{up}}}
\newcommand{\Sd}{S^{\text{down}}}

\newcommand{\F}{\mathcal{F}}

\newcommand{\ti}{\;\;\makebox[0pt]{$\top$}\makebox[0pt]{$\cap$}\;\;}

\newtheorem{lem}{Lemma}[section]
\newtheorem{remark}{Remark}[section]

\newtheorem{prop}{Proposition}[section]

\numberwithin{equation}{section}

\begin{document}

\unitlength=\textwidth

\title{The scattering map in two coupled piecewise-smooth systems, with
numerical application to rocking blocks}
\author[1]{A. Granados}
\author[2]{S.J. Hogan}
\author[3]{T.M. Seara}
\affil[1]{INRIA Paris-Rocquencourt}
\affil[2]{University of Bristol}
\affil[3]{Universitat Polit\`ecnica de Catalunya}

\maketitle

\begin{abstract}
We consider a non-autonomous dynamical system formed by coupling two
piecewise-smooth systems in $\RR^2$ through a non-autonomous periodic
perturbation. We study the dynamics around one of the heteroclinic orbits of one
of the piecewise-smooth systems.  In the unperturbed case, the system possesses
two $C^0$ normally hyperbolic invariant
manifolds of dimension two with a couple of three dimensional heteroclinic
manifolds between them. These heteroclinic manifolds are foliated by
heteroclinic connections between $C^0$ tori located at the same energy levels.
By means of the {\em impact map} we prove the persistence of these objects under
perturbation. In addition, we provide sufficient conditions of the existence of
transversal heteroclinic intersections through the existence of simple zeros of
Melnikov-like functions.  The heteroclinic manifolds allow us to define the {\em
scattering map}, which links asymptotic dynamics in the invariant manifolds
through heteroclinic connections. First order properties of this map provide
sufficient conditions for the asymptotic dynamics to be located in different
energy levels in the perturbed invariant manifolds. Hence we have an essential
tool for the construction of a heteroclinic skeleton which, when followed, can
lead to the existence of Arnol'd diffusion: trajectories that, on large time
scales, destabilize the system by further accumulating energy.  We validate all
the theoretical results with detailed numerical computations of a mechanical
system with impacts, formed by the linkage of two rocking blocks with a spring.
\end{abstract}


\section{Introduction}
This paper is concerned with the question of whether it is possible to observe
Arnol'd diffusion~\cite{Arn64} in systems governed by piecewise-smooth
differential equations, to which known results in the field can not be directly
applied. Arnol'd diffusion occurs when there is a large change in the action
variables in nearly integrable Hamiltonian systems. Systems governed by
piecewise-smooth differential equations  are widespread in engineering,
economics, electronics, ecology and biology; see~\cite{MakLam12} for a recent
comprehensive survey of the field. 

Action variables are conserved for integrable systems. When such systems
are perturbed, for example, by a periodic forcing, KAM theory
tells us that the value of these variables stays close to their conserved values
for {\em most} solutions. Subsequently Arnol'd~\cite{Arn64} gave an example of
an nearly integrable system for which there was large growth in the
action variables.

There has been a  lot of activity in the field of Arnol'd diffusion in recent
years and a large variety of results that have been obtained or announced. We
refer to  \cite{DelshamsGLS08b,Cheng08,Cheng10,Bernard10} for a detailed survey
of recent results.  Up to now, there are mainly two kind of methods used to
prove the existence of instabilities in Hamiltonian systems close to integrable;
variational methods
\cite{Berti02,BertiBB02,BertiBB03,Mather02,ChengY04,KalLev08,KalLev08b,BernardKZ11,KaloshinZ12}
and the so-called geometric methods
\cite{DelLlaSea00,DelLlaSea06,DelLlaSea08,GideaL06b,Treschev04,Treschev12,FejozGKR11},
both of which have been used to prove generic results or study concrete
examples.

The study of Arnol'd diffusion using geometric methods has been greatly
facilitated by the introduction~\cite{DelLlaSea00,DelLlaSea06,DelLlaSea08} of
the {\em scattering map} of a normally hyperbolic invariant manifold with
intersecting stable and unstable invariant manifolds along a homoclinic
manifold. This map finds the asymptotic orbit in the future, given an asymptotic
orbit in the past.  Perturbation theory of the scattering map~\cite{DelLlaSea08}
generalizes and extends several results obtained using Melnikov's
method~\cite{Mel63, GucHol83}.

For planar regular systems under non-autonomous periodic perturbations,
Melnikov's method is used to determine the persistence of periodic orbits and
homoclinic/heteroclinic connections by guaranteeing the existence of simple
zeros of the subharmonic Melnikov function and the Melnikov function,
respectively. The main idea is to consider a section normal to the unperturbed
vector field at some point on the unperturbed homoclinic/heteroclinic
connection. Then it is possible to measure the distance between the perturbed
manifolds, owing to the regularity properties of the stable and unstable
manifolds of hyperbolic critical points in smooth systems.

In~\cite{GraHogSea12} these classical results were rigorously extended to a
general class of piecewise-smooth differential equations, allowing for a general
periodic Hamiltonian perturbation, with no symmetry assumptions. For such
systems, the unperturbed system is defined in two domains, separated by a {\em
switching manifold} $\Sigma$, each possessing one hyperbolic critical point
either side of $\Sigma$. In this case, the vector normal to the unperturbed
vector field is not defined everywhere. By looking for the intersection between
the stable and unstable manifolds with the switching manifold, an asymptotic
formula for the distance between the manifolds was obtained. This turned out to
be a {\em modified} Melnikov function, whose zeros give rise to the existence of
heteroclinic connections for the perturbed system. The general results
in~\cite{GraHogSea12} were then applied to the case of the rocking
block~\cite{Hou63,Hog89} and excellent agreement was obtained with the results
of ~\cite{Hog89}.

Following these ideas, in this paper we study a system which consists of a
non-autonomous periodic perturbation of a piecewise-smooth integrable
Hamiltonian system in $\RR^4$.  The unperturbed system is given by the product
of two piecewise-smooth systems. We assume that one of them has two hyperbolic
critical points of saddle type with a pair of $C^0$ heteroclinic orbits between
them.  The other system behaves as a classical integrable system with a region
foliated by $C^0$ periodic orbits.  Therefore, the product system looks like a
classical {\emph{a priori}} unstable Hamiltonian system \cite{ChierchiaG94},
possessing two $C^0$ normally hyperbolic invariant manifolds of dimension two
with a couple of three dimensional $C^0$ heteroclinic manifolds.

The main difficulty in following the program of~\cite{DelLlaSea06} is that we
couple two piecewise-smooth systems,  each of which possesses its own switching
manifold.  Therefore, when considering the product system, we need to deal with
a piecewise-smooth system in $\RR^4$ with two 3-dimensional switching manifolds
that cross in a 2-dimensional one. Therefore the classical impact map associated
with one switching manifold will be piecewise-smooth in general. We overcome
this difficulty by restricting the impact map to suitable domains so that we can
apply classical results for normally hyperbolic invariant manifolds and their
persistence and obtain a scattering map between them with explicit asymptotic
formulae.

Note that, in this paper, we restrict our attention to the study of the
scattering map and we do not rigorously prove the existence of Arnol'd
diffusion.  Due to the continuous nature of the system considered in this paper,
the method of correctly aligned windows \cite{GideaL06b} seems to be very
suitable for application to our model for this purpose. In
fact, recent results in \cite{LlaveGS13}, which do not rely on the use of KAM
theory, appear to be capable of extension to piecewise-smooth systems in order to
achieve this goal.

Piecewise-smooth systems are found in a host of applications~\cite{MakLam12}. A
simple example is the rocking block model~\cite{Hou63}, which has wide
application in earthquake engineering and robotics. This piecewise-smooth system
has been shown to possess a vast array of solutions~\cite{Hog89}. The model has
been extended to include, for example, stacked rigid blocks~\cite{SpaRouPol01}
and multi--block structures~\cite{PenLouCam08}. Particular attention is paid to
the case of block overturning in the presence of an earthquake, as this has
consequences for safety in the nuclear industry~\cite{CK09} and for the
preservation of ancient statues~\cite{KouPapCot12}. Within the context of the
current paper, Arnol'd diffusion could be seen as one possible mechanism for
block overturning, when the perturbation (earthquake) of an apparently stable
system (two blocks coupled by a simple spring) leads to  overturning. An early
application of Melnikov theory to the rocking block problem~\cite{Kov10}
involved the calculation of the {\em stochastic} Melnikov criterion of
instability for a multidimensional rocking structure subjected to {\em random}
excitation.

Note that we are considering the class of piecewise-smooth differential
equations that involve {\em crossing}~\cite{MakLam12}, where the normal
components of the vector field either side of the switching manifold are in the
{\em same} sense. When these components are in the {\em opposite} sense, {\em
sliding} can occur~\cite{MakLam12}. The extension of the Melnikov method to this
case is still in its infancy~\cite{DuLi12}.

The paper is organised as follows.  In
section~\ref{sec:system_description_scattering} we present the system we will
consider and the main piecewise-smooth invariant geometrical objects that will
play a role in the process.  In section~\ref{sec:notation_properties_scattering}
we present the impact map associated with one switching manifold in the extended
phase space and its domains of regularity and provide an explicit expression for
it in the unperturbed case.  In section~\ref{sec:invariant_objects_persistence}
we study  some regular normally hyperbolic invariant manifolds for the impact
map which correspond to the piecewise-smooth ones for the flow in the extended
phase space.  We then apply classical perturbation theory to demonstrate the
persistence of the normally hyperbolic invariant manifolds and their stable and
unstable manifolds and deduce the persistence of the corresponding invariant
manifolds for the perturbed flow. This allows us to give explicit conditions for
the existence of transversal heteroclinic manifolds in the perturbed system in
terms of a modified Melnikov function and to derive explicit formulae for the
scattering map in section~\ref{sec:scattering_map}. In particular, we obtain
formulae for the change in the energy of the points related by the scattering
map and in the average energy along their orbits.  In section~\ref{sec:example}
we illustrate the theoretical results of section~\ref{sec:scattering_map} with
numerical computations for two coupled rocking blocks subjected to a small
periodic forcing.  We use the simple zeros of the Melnikov function to
numerically compute heteroclinic connections linking, forwards and backwards in
time, two trajectories at the invariant manifolds. These trajectories correspond
to one block performing small rocking oscillations while the other block rocks
about one of its heteroclinic orbits. During this large, fast, excursion, the
amplitude of the rocking block oscillations may lead to an increase or decrease
in its average energy. Using the first order analysis of the scattering map we
are able to approximately predict the magnitude of this change, which is in
excellent agreement with our numerical computations.

\section{System description}
\label{sec:system_description_scattering}
\subsection{Two uncoupled systems}
\label{sec:uncoupled_systems}
In this paper we consider a non-autonomous dynamical system formed by coupling
two piecewise-smooth systems in $\RR^2$ through a non-autonomous periodic
perturbation. We divide $\RR^2$ into two sets,
\begin{align*}
&S^+=\left\{ (q,p)\in\RR^2\,\vert\,q>0 \right\},\\
&S^-=\left\{ (q,p)\in\RR^2\,\vert\,q<0 \right\},
\end{align*}
separated by the switching manifold 
\begin{equation}
\Sigma=\Sigma^+\cup\Sigma^-\cup\{(0,0)\},
\label{eq:boundary_scattering}
\end{equation}
where
\begin{equation}
\begin{aligned}
&\Sigma^+=\left\{ (0,p)\in\RR^2\,\vert \, p>0 \right\},\\
&\Sigma^-=\left\{ (0,p)\in\RR^2\,\vert \, p<0 \right\}.
\end{aligned}
\label{eq:sigmas}
\end{equation}
We consider the piecewise-smooth systems defined in $\RR^2\backslash\Sigma$
\begin{equation}
\left( 
\begin{array}{c}
\dot{x}\\\dot{y}
\end{array}
 \right):=\X(x,y):=
\left\{ 
\begin{aligned}
&\X^+(x,y)&&\text{if }(x,y)\in S^+\\
&\X^-(x,y)&&\text{if }(x,y)\in S^-
\end{aligned}
 \right.
\label{eq:field1}
\end{equation}
\begin{equation}
\left( 
\begin{array}{c}
\dot{u}\\\dot{v}
\end{array}
 \right):=\U(u,v):=
\left\{ 
\begin{aligned}
&\U^+(u,v)&&\text{if }(u,v)\in S^+\\
&\U^-(u,v)&&\text{if }(u,v)\in S^-
\end{aligned}
 \right.
\label{eq:field2}
\end{equation}
with $\X^\pm(x,y),\,\U^\pm(u,v)\in C^\infty(\RR^2)$.\\

Let us assume that~(\ref{eq:field1}) and (\ref{eq:field2}) are Hamiltonian
systems associated, respectively, with $C^0(\RR^2)$ piecewise-smooth
Hamiltonians of the form
\begin{align}
X(x,y)&:=\frac{y^2}{2}+Y(x)\nonumber\\
&:=\left\{ 
\begin{aligned}
&X^+(x,y):=\frac{y^2}{2}+Y^+(x)&&\text{if }(x,y)\in
S^+
\\
&X^-(x,y):=\frac{y^2}{2}+Y^-(x)&&\text{if }(x,y)\in S^-
\end{aligned}
\right.
\label{eq:unperturbed_Hamiltonian1}
\end{align}
\begin{align}
U(u,v)&:=\frac{v^2}{2}+V(u)\nonumber\\
&:=\left\{ 
\begin{aligned}
&U^+(u,v):=\frac{v^2}{2}+V^+(u)&&\text{if }(u,v)\in
S^+
\\
&U^-(u,v):=\frac{v^2}{2}+V^-(u)&&\text{if }(u,v)\in S^- ,
\end{aligned}
\right.
\label{eq:unperturbed_Hamiltonian2}
\end{align}
with $Y^\pm,\,V^\pm\in C^{\infty}(\RR^2)$ satisfying
$Y^+(0)=Y^-(0)=0$ and $V^+(0)=V^-(0)=0$.  Then 
\begin{equation}
\begin{aligned}
\X^\pm&=J\nabla X^\pm\\
\U^\pm&=J\nabla U^\pm\\
\end{aligned}
\label{eq:Hamiltonian_relations}
\end{equation}
where $J$ is the symplectic matrix
\begin{equation*}
J=\left(
\begin{array}{cc}
0&1\\-1&0
\end{array}\right).
\end{equation*}

From the form of the Hamiltonians~(\ref{eq:unperturbed_Hamiltonian1})
and~(\ref{eq:unperturbed_Hamiltonian2}), it is natural to extend the
definition of the flows of $\X^+$ and  $\U^+$ to $S^+\cap\Sigma^+$ and of the flows of
$\X^-$ and $\U^-$ to $S^-\cap\Sigma^-$. Hence, the
Hamiltonian $X(x,y)$ in~(\ref{eq:unperturbed_Hamiltonian1}) is naturally extended to $\RR^2$ as
\begin{equation*}
X(x,y)=\left\{
\begin{aligned}
&X^+(x,y)&&\text{if }(x,y)\in S^+\cup\Sigma^+\cup\left\{ (0,0) \right\}\\
&X^-(x,y)&&\text{if }(x,y)\in S^-\cup\Sigma^-,
\end{aligned}\right.
\end{equation*}
and similarly for the Hamiltonian $U(u,v)$ in~(\ref{eq:unperturbed_Hamiltonian2}).
Note that the vector fields $\X^+$ and $\X^-$ are tangent to $\Sigma$ at $(0,0)$
(resp.  $\U^+$ and $\U^-$).

To define the flow associated with system~(\ref{eq:field1}), we proceed as usual
in piecewise-smooth systems. Given an initial condition $(x_0,y_0)\in S^{\pm}$,
we apply the flows $\phi_{\X^\pm}$ associated with the smooth systems $\X^\pm$
until the switching manifold $\Sigma$ is crossed at some point. Then, using this
point as the new initial condition we evolve with the flow in the new domain.
The flow associated with system~(\ref{eq:field2}) is defined in a similar way.
Note that, as no sliding along the switching manifold is possible, the
definition of the flows is straightforward. This allows us to consider the
flows 
\begin{equation}
\phi_{\X}(t;x_0,y_0)\text{ and } \phi_{\U}(t;,u_0,v_0)
\label{eq:unperturbed_flows}
\end{equation}
associated with systems~(\ref{eq:field1}) and~(\ref{eq:field2}), respectively,
that are $C^0$ functions piecewise-smooth in $t$ satisfying
\begin{align*}
\phi_{\X}(0;,x_0,y_0)&=(x_0,y_0)\\
\phi_{\U}(0;,u_0,v_0)&=(u_0,v_0).
\end{align*}

Let us assume that the following conditions are satisfied.
\begin{enumerate}[C.1]
\item System~(\ref{eq:field1}) possesses two hyperbolic critical points
$Q^+\in S^+$ and $Q^-\in S^-$ of saddle type belonging to the energy
level $X(x,y)=\bar{d}$.
\item The   energy level $X(x,y)=\bar{d}$ contains  two heteroclinic orbits given by $\gamma^{\text{up}}:= W^u(Q^-)=W^s(Q^+)$ and $\gamma^{\text{down}}:=W^u(Q^+)=W^s(Q^-)$.
\item The Hamiltonians $U^\pm$ in~(\ref{eq:unperturbed_Hamiltonian2}) satisfy
\begin{equation*}
(V^+)'(0)>0;\; (V^-)'(0)<0,
\end{equation*}
and so $(0,0)$ is an invisible quadratic tangency for both vector fields
$\U^\pm$ in~(\ref{eq:field2}). Following~\cite{Kuznetsov04,GuaSeaTei11}, we call the point
$(0,0)$ an invisible fold-fold.
\item System~(\ref{eq:field2}) possesses a continuum of (piecewise-smooth)
continuous periodic orbits surrounding the origin. These can be parameterized by the
Hamiltonian $U$ and have the form
\begin{equation}
\Lambda_c=\left\{ (u,v)\in\RR^2\,\vert\,U(u,v)=c \right\},\;0<c\le \bar{c}.
\label{eq:periodic_orbits_scattering}
\end{equation}
\end{enumerate}
The main purpose of this paper is to study the dynamics around one of the
heteroclinic orbits of system \eqref{eq:field1}. From now on, we focus on the upper one
\begin{equation*}
\gamma^{\text{up}}:=W^u(Q^-)\cap W^s(Q^+)=\left\{
(x,y)\in\RR^2\,\vert\,X(x,y)=\bar{d},\,y\ge 0 \right\}.
\end{equation*}
There we consider the following parameterization
\begin{equation}
\gamma^{\text{up}}=\left\{ \su(t),\,t\in\RR \right\}
\label{eq:unpert_het_orbit}
\end{equation}
where $\su(t)$ is the solution of system~(\ref{eq:field1}) satisfying
\begin{equation}\begin{aligned}
&\su(0)=(0,y_h)\in\Sigma\\
&\lim_{t\to\pm\infty}\su(t)=Q^\pm,
\end{aligned}
\label{eq:unpert_het_orbit_parameterization}
\end{equation}
where $(0,y_h)$, $y_h=\sqrt{\bar{d}}$, is given by
\begin{equation*}
(0,y_h)=W^u(Q^-)\cap\Sigma=W^s(Q^+)\cap\Sigma.
\end{equation*}

Before introducing the non-autonomous perturbation which will couple both
systems described above, we outline the invariant objects of the cross product
of both systems (see figure~\ref{fig:invariant_obj_unper_coupled}), which has a
(piecewise-smooth) Hamiltonian
\begin{equation}
H_0(u,v,x,y)=U(u,v)+X(x,y).
\end{equation}
\begin{figure}
\begin{center}
\includegraphics[width=1\textwidth]
{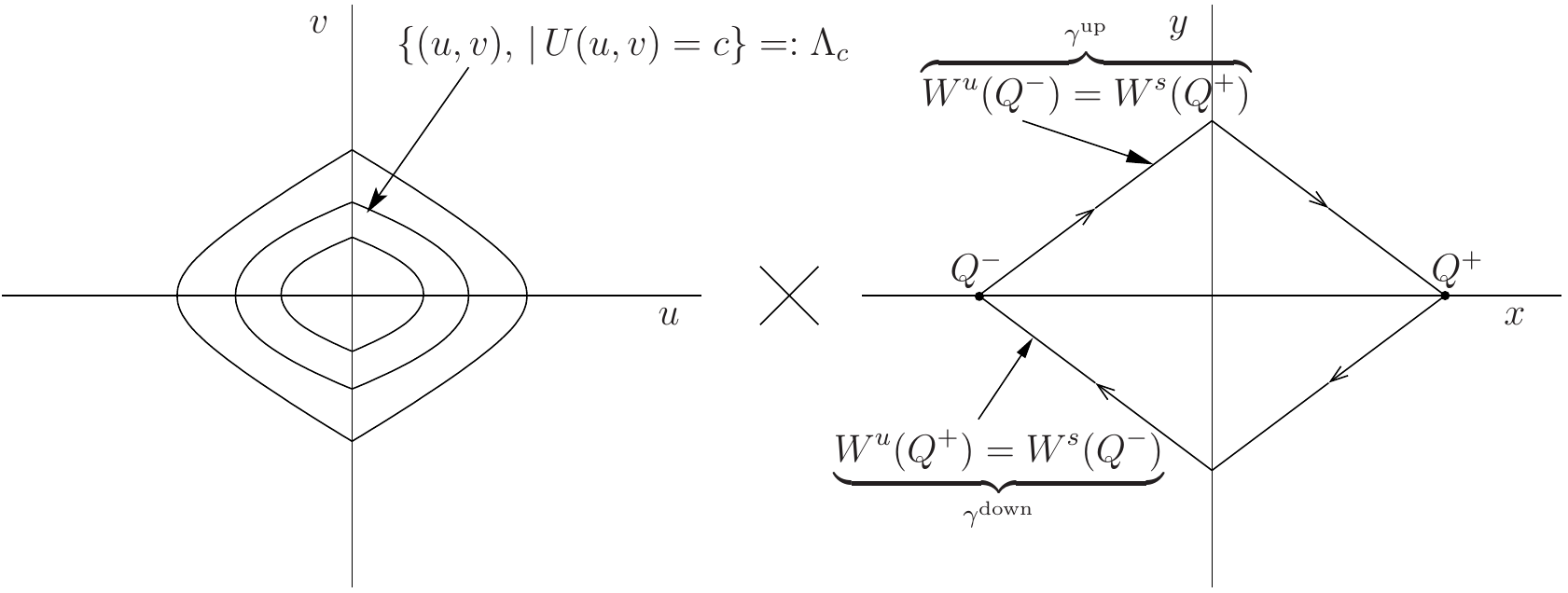}
\end{center}
\caption{Invariant objects for the unperturbed coupled system.}
\label{fig:invariant_obj_unper_coupled}
\end{figure}
Even if the periodic orbits $\Lambda_c\times Q^\pm$ are only continuous
manifolds, as $Q^\pm$ are hyperbolic critical points, they can be considered
hyperbolic periodic orbits. Moreover, their stable and unstable (non-regular)
manifolds, $W^{s,u}(\Lambda_c\times Q^\pm)$, are given by $\Lambda_c\times
W^{s,u}(Q^\pm)$. Furthermore, the stable/unstable manifold of each
periodic orbit $\Lambda_c\times Q^+$ coincides with the unstable/stable manifold
of the periodic orbit $\Lambda_c\times Q^-$, respectively, and hence there exist
(non-regular) heteroclinic manifolds connecting these periodic orbits.\\
Also of interest are the manifolds $\Lambda^\pm$ given by the cross
product of the critical points $Q^\pm$ with the union of all periodic orbits
\begin{equation*}
\begin{aligned}
\Lambda^+&=\bigcup
_{c\in[c_1,c_2]}\Lambda_c\times Q^+\\
&=\left\{ \left( u,v,Q^+\right)\,\vert\, U(u,v)=c,\, c_1\le c\le c_2,\,\right\}\\
\Lambda^-&=\bigcup_{c\in[c_1,c_2]}\Lambda_c\times Q^-\\
&=\left\{ \left( u,v,Q^-\right)\,\vert\,U(u,v)=c,\,c_1\le c\le c_2\right\},
\end{aligned}
\end{equation*}
for some $0<c_1,\,c_2<\bar{c}$. In figure~\ref{fig:norm_hyp_mani} we show these two manifolds schematically.\\
\begin{figure}
\begin{center}
\includegraphics[width=1\textwidth]
{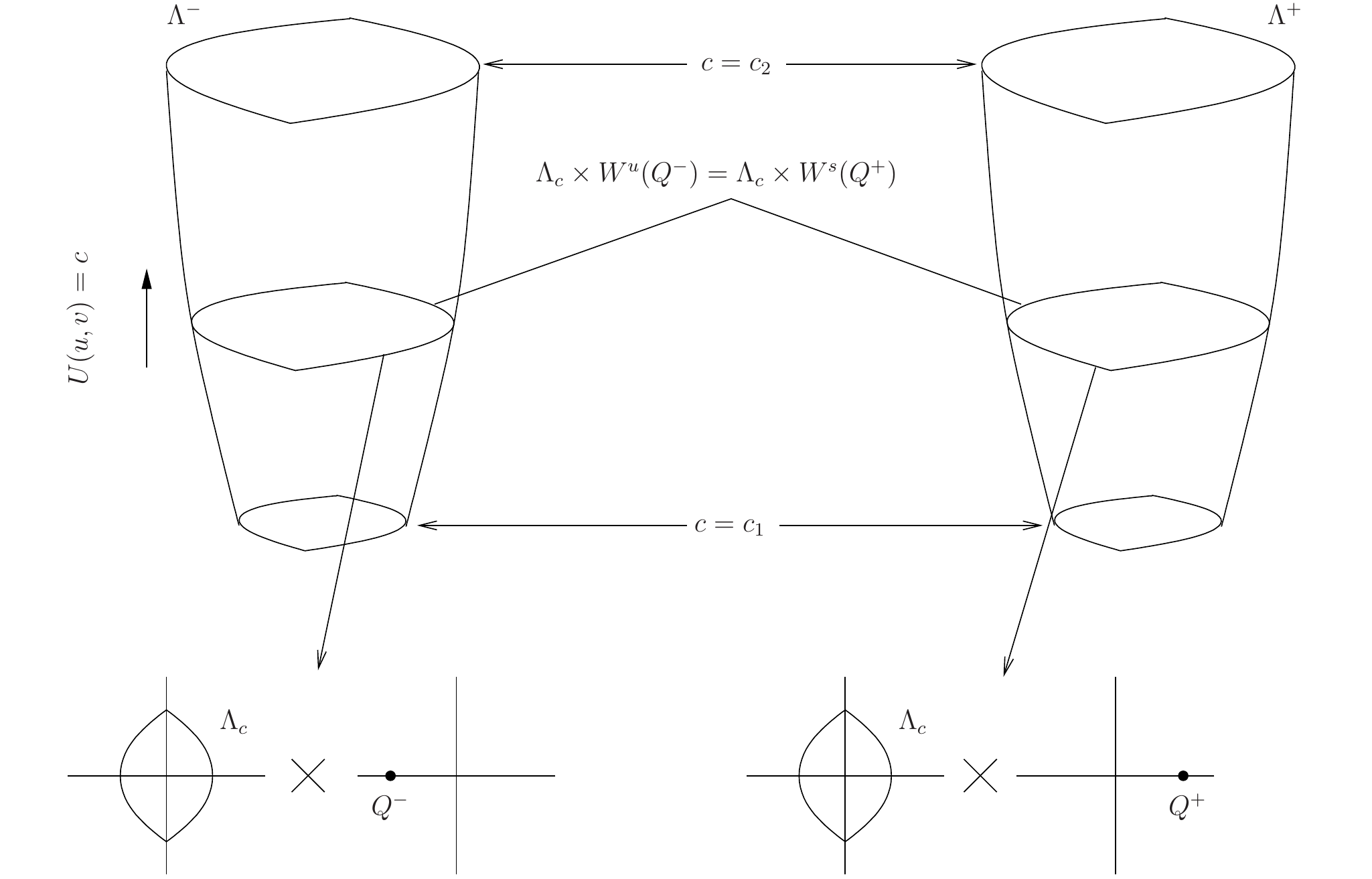}
\end{center}
\caption{Schematic representation of the manifolds $\Lambda^+$ and $\Lambda^-$.}
\label{fig:norm_hyp_mani}
\end{figure}

\subsection{The coupled system}\label{sec:coupled_system}
We now consider the system given by coupling
systems~(\ref{eq:field1}) and (\ref{eq:field2}) through a non-autonomous
$T$-periodic Hamiltonian perturbation $\varepsilon h(u,v,x,y,s)\in
C^\infty(\RR^5)$ satisfying 
\begin{equation*}
h(u,v,x,y,s)=h(u,v,x,y,s+T),\;\forall (u,v,x,y,s)\in\RR^5.
\end{equation*}
Therefore, the perturbed system is a non-autonomous $T$-periodic in time Hamiltonian system with Hamiltonian:
\begin{equation}
H_\varepsilon(\tz):=U(u,v)+X(x,y)+\varepsilon h(\tz),\;\varepsilon>0,
\label{eq:general_Hamiltonian_scattering}
\end{equation}
where $\tz=(z,s)=(u,v,x,y,s)$, $s\in\T$ and $\T=\RR\backslash T$.  To study the
dynamics of the corresponding Hamiltonian system we will work in the extended
state space $\RR^4\times\T$, adding the time $s$ as a state variable.  Note that
we retain $\T$, rather than the usual circle (modulus $1$), because $T$ is very
important in applications.\\
Recalling that the unperturbed systems~(\ref{eq:field1}) and~(\ref{eq:field2}) are
piecewise-smooth, 
the coupled system is defined in four
partitions of $\RR^4\times\T$ as follows
\newcommand{\Size}{\hspace{0.3\textwidth}}
\begin{equation}
\begin{aligned}
\left(
\begin{array}{c}
\dot{u}\\\dot{v}\\\dot{x}\\\dot{y}
\end{array}
\right)&=\left\{
\begin{aligned}
&\tJ\nabla \left( U^++X^+ +\varepsilon h\right)\left(\tz\right)\\
&\Size\text{if }\tz\in S^+\cup\Sigma^+\times S^+\cup\Sigma^+\times\T\\
&\tJ\nabla \left( U^++X^-+\varepsilon h\right)\left(\tz\right)\\
&\Size\text{if }\tz\in S^+\cup\Sigma^+\times S^-\cup\Sigma^-\times\T\\
&\tJ\nabla \left( U^-+X^-+\varepsilon h\right)\left(\tz\right)\\
&\Size\text{if }\tz\in S^-\cup\Sigma^-\times S^-\cup\Sigma^-\times\T\\
&\tJ\nabla \left( U^-+X^++\varepsilon h\right)\left(\tz\right)\\
&\Size\text{if }\tz\in S^-\cup\Sigma^-\times S^+\cup\Sigma^+\times\T
\end{aligned}
\right.\\
\dot{s}&=1,
\end{aligned}
\label{eq:general_field_scattering}
\end{equation}
where $\tz=(z,s)=(u,v,x,y,s)$, $s\in\T$ and
\begin{equation*}
\tJ=\left( \begin{array}{cccc}
0&1&0&0\\-1&0&0&0\\0&0&0&1\\0&0&-1&0
\end{array} \right).
\end{equation*}
These differential equations define four different autonomous flows
$\tp^{\pm\pm}(t;\tz_0;\varepsilon)$ in the extended phase space. Letting
$\varphi^{\pm\pm}(t;t_0,z_0;\varepsilon)$ denote the corresponding non-autonomous
flows such that $\varphi^{\pm\pm}(t_0;t_0,z_0;\varepsilon)=z_0$, we write
$\tp^{\pm\pm}(t;\tz_0;\varepsilon)$ satisfying
$\tp^{\pm\pm}(0;\tz_0;\varepsilon)=\tz_0$ as
\begin{equation*}
\tp^{\pm\pm}(t;\tz_0;\varepsilon)=\left( \phi^{\pm\pm}(t;\tz_0;\varepsilon),s_0+t \right),
\end{equation*}
where $\phi^{\pm\pm}(t;\tz_0;\varepsilon)$ are such that
\begin{equation*}
\varphi^{\pm\pm}(t;t_0,z_0;\varepsilon)=\phi^{\pm\pm}(t-t_0;\tz_0;\varepsilon).
\end{equation*}

Proceeding as we did for the systems $\U$ and $\X$, we can define the solution,
$\tp(t;\tz_0;\varepsilon)$, of the coupled
system~(\ref{eq:general_field_scattering}) satisfying
$\tp(0;\tz_0;\varepsilon)=\tz_0$ by properly concatenating the flows
$\tp^{+\pm}$ and $\tp^{-\pm}$ when the $4$-dimensional switching manifold $u=0$
is crossed, and $\tp^{\pm+}$ and $\tp^{\pm-}$ when $x=0$ is crossed. Following
this definition, we will omit from now on the indices $\pm$ and write just
$\tp$.  Note that $\tp$ is not differentiable at those times corresponding to
the crossings with the switching manifolds, although it is as smooth as the
flows $\tp^{\pm\pm}$ when restricted to the open domains given in the respective
branches.

Note that, for $\varepsilon=0$, all the invariant objects described
in~\S\ref{sec:uncoupled_systems} for the cross product of the
systems~(\ref{eq:field1}) and~(\ref{eq:field2}) become invariant objects of
system \eqref{eq:general_field_scattering} with one dimension more in the
extended phase space.  The study of these objects and their persistence after
adding the perturbation will be the goal of
section~\S\ref{sec:invariant_objects_persistence}.

\section{Some notation and properties}\label{sec:notation_properties_scattering}
\subsection{Impact map associated with $u=0$}
\label{sec:impact_map_scattering}
In $\RR^4\times\T$ let us define the sections
\begin{equation}
\tS=\Sigma\times\RR^2\times\T=\left\{ (0,v,x,y,s),\,(v,x,y,s)\in\RR^3\times\T \right\}\label{eq:section_Sigma_extended}
\end{equation}
and
\begin{align}
&\tS^+=\Sigma^+\times\RR^2\times\T=\left\{ (0,v,x,y,s)\in\tS\,\vert\,v>0 \right\}\\
&\tS^-=\Sigma^-\times\RR^2\times\T=\left\{ (0,v,x,y,s)\in\tS\,\vert\,v<0 \right\}.\label{eq:section_Sigmam_extended}
\end{align}
Note that $\tS$ is a switching manifold of
system~(\ref{eq:general_field_scattering}) in the extended phase space, and it will
play an important role in 
what follows.

We wish to define the {\it impact map} $P_\varepsilon$ associated with $\tS$,
that is, the Poincar\'{e} map from $\tS$ to itself (see
figure~\ref{fig:poincare_map_scattering}).  This map is as regular as the flows 
$\tp^{\pm\pm}$
restricted to some open domains, and this will allow us to apply classical
results from perturbation theory of smooth systems that will be useful in our
construction.
\begin{figure}
\begin{center}
\includegraphics[width=1\textwidth]
{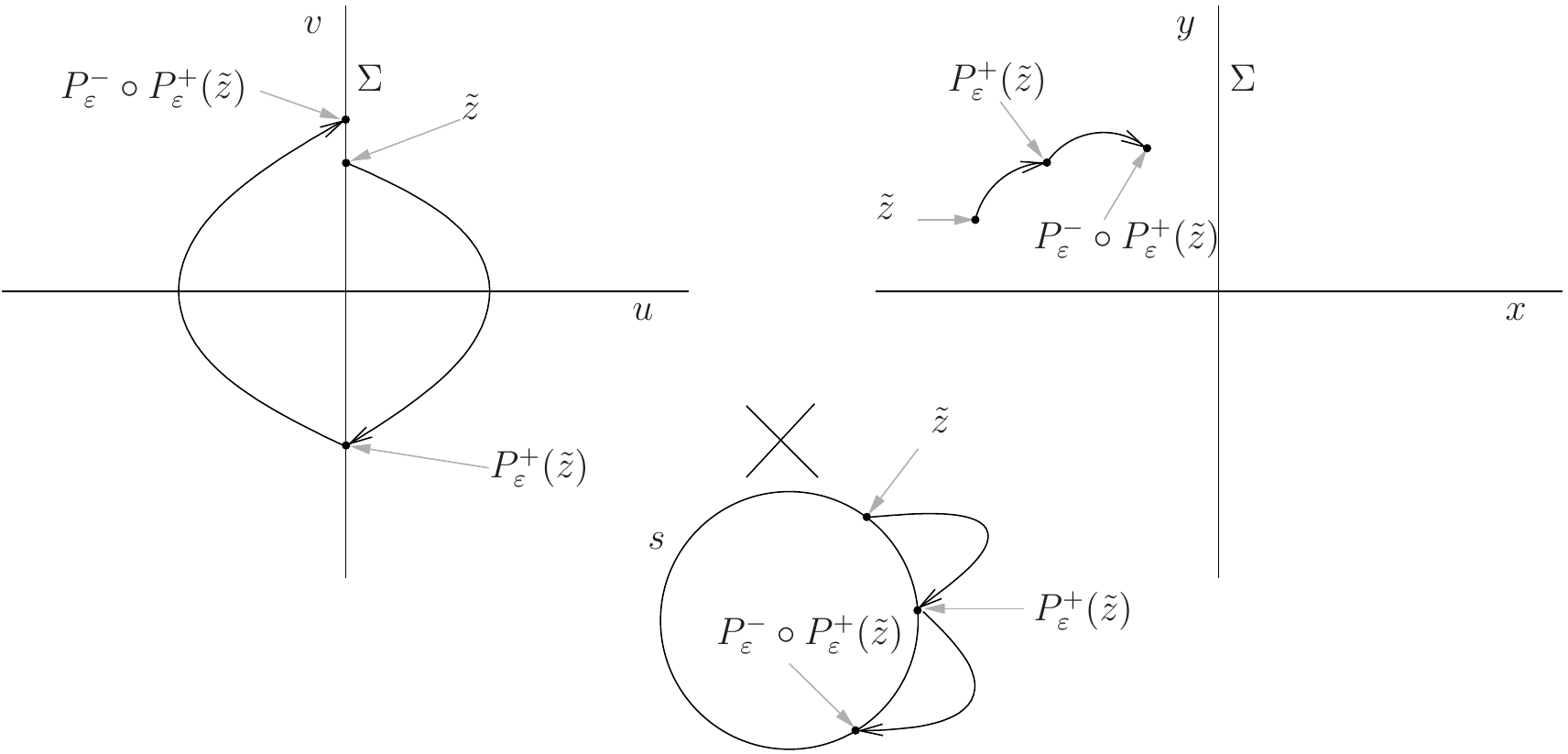}
\end{center}
\caption{Schematic representation of the maps $P_\varepsilon^-$,
$P_\varepsilon^+$.}
\label{fig:poincare_map_scattering}
\end{figure}
The impact map $P_\varepsilon$ is given by the composition of two intermediate maps,
\begin{align}
P^+_\varepsilon&:\tO_{P_\varepsilon^+}\subset\tS^+\longrightarrow\tS^-
\label{eq:impact_mapp}\\
P^-_\varepsilon&:\tO_{P_\varepsilon^-}\subset\tS^-\longrightarrow\tS^+
\label{eq:impact_mapm}
\end{align}
defined as
\begin{align*}
P_\varepsilon^+(\tz)&=\tp(t_{\tS^-};\tz;\varepsilon)\\
P_\varepsilon^-(\tz)&=\tp(t_{\tS^+};\tz;\varepsilon),
\end{align*}
where $t_{\tS^\pm}$ are the smallest values of $t>0$ such that
$\tp(t;\tz;\varepsilon)\in\tS^{\mp}$. The domains $\tO_{P_\varepsilon^\pm}$ where these maps
are smooth  are the open sets given by points in $\tS^\pm$ whose
trajectories first impact the switching manifold $\tS$ given by $u=0$ rather
than the switching manifold $x=0$. That is,
\begin{equation}
\tO_{P_\varepsilon^+}=\left\{ \tz\in
\tS^+,\vert\,\Pi_x\left(
\tp(t;\tz;\varepsilon)
\right)\ne0 \,\forall t\in[0,t_{\tS^-}] \right\}
\label{eq:open_set_Pp}
\end{equation}
and
\begin{equation*}
\tO_{P_\varepsilon^-}=\left\{ \tz\in
\tS^-,\vert\,\Pi_x\left(
\tp(t;\tz;\varepsilon)
\right)\ne0 \,\forall t\in[0,t_{\tS^+}] \right\}.
\end{equation*}
\begin{remark}
Due to the form of the Hamiltonian $X$ given
in~(\ref{eq:unperturbed_Hamiltonian1}), for $\varepsilon\ge0$ small enough the
flow crosses the switching manifold $x=0$ for increasing $x$ when $y>0$ and for 
decreasing $x$ for $y<0$. Hence, the points in $\tO_{P_\varepsilon^+}$ and
$\tO_{P_\varepsilon^-}$ can be arbitrarily close to $x=0$ when $xy\ge0$ (possibly containing some part of the
segment $x=0$) but not when $xy<0$. This implies that the sets
$\tO_{P_\varepsilon^\pm}$ consist of two connected components separated by the
switching manifold $x=0$, $\RR^2\times\Sigma\times\T$. How these sets
are separated from $x=0$ depends on the time required to reach the switching
manifold $u=0$.
\end{remark}
Let us consider an open set,
\begin{equation}
\tO_{P_\varepsilon}\subset \tO_{P_\varepsilon^+}\cup\tO_{P_\varepsilon^-}\subset\tS,
\label{eq:open_set}
\end{equation}
and define the Poincar\'e impact map
\begin{equation*}
P_\varepsilon:\tO_{P_\varepsilon}\subset\tS\longrightarrow\tS
\end{equation*}
as
\begin{equation*}
P_{\varepsilon}(0,v,x,y,s)=\left\{
\begin{aligned}
&P_\varepsilon^+\circ P_\varepsilon^-(0,v,x,y,s)&&\text{if }\left( 0,v,x,y,s
\right)\in\tO_{P_\varepsilon}\cap\tS^-\\
&P_\varepsilon^-\circ P_\varepsilon^+(0,v,x,y,s)&&\text{if }\left( 0,v,x,y,s
\right)\in\tO_{P_\varepsilon}\cap\tS^+
\end{aligned}\right.
\end{equation*}
To simplify notation, when considering points
$(0,v,x,y,s)\in\tS\subset\RR^4\times\T$, we introduce the new variable
$\tw=(v,x,y,s)$. Then points $\tz$ in $\RR^4\times\T$ will be written as
$\tz=(0,\tw)$. In addition we consider the set $\tO_{P_\varepsilon}$ in
$\RR^3\times\T$ and write the impact map
\begin{equation}
P_\varepsilon:\tO_{P_\varepsilon}\longrightarrow\RR^3\times\T
\label{eq:impact_map_scattering}
\end{equation}
as
\begin{equation}
P_{\varepsilon}(\tw)=\left\{
\begin{aligned}
&P_\varepsilon^+\circ P_\varepsilon^-(\tw)&&\text{if
}\tw\in\tO_{P_\varepsilon}\cap \left\{ (v,x,y,s)\in\RR^3\times\T,\,v<0 \right\}\\
&P_\varepsilon^-\circ P_\varepsilon^+(\tw)&&\text{if
}\tw\in\tO_{P_\varepsilon}\cap \left\{ (v,x,y,s)\in\RR^3\times\T,\,v>0 \right\}
\end{aligned}\right.
\label{eq:impact_map_explicit_composition}
\end{equation}
with
\begin{align*}
\tO_{P_\varepsilon}=\Big\{ &\tw=(v,x,y,s)\in
\left([-v_2,-v_1]\cup[v_1,v_2]\right)\times\RR^2\times\T\\
&\vert\,\Pi_x\left(
\tp(t;\left( 0,\tw \right);\varepsilon)
\right)\ne0 \,\forall  t\in\left[
0,\Pi_s\left( P_\varepsilon(\tw) \right) -s\right]\Big\}.
\end{align*}
Note that the map $P_\varepsilon$ is invertible in
$\tO_{P_\varepsilon^{-1}}:=P_\varepsilon(\tO_{P_\varepsilon})$ and hence we can consider
\begin{equation}
P_\varepsilon^{-1}:\tO_{P_\varepsilon^{-1}}\subset\RR^3\times\T\longrightarrow
\RR^3\times\T.
\label{eq:impact_map_inverse}
\end{equation}
\begin{remark}\label{rem:smoothness_of_impact_map}
Although the maps $P^+_\varepsilon$, $P^-_\varepsilon$, $P_{\varepsilon}$ can be
defined in a wider zone of the extended phase space, their restriction to the
domains $\tO_{P_\varepsilon^+}$, $\tO_{P_\varepsilon^-}$, $\tO_{P_\varepsilon}$
will be essential in our contructions.  The reason is that the maps
$P^+_\varepsilon$, $P^-_\varepsilon$, $P_{\varepsilon}$ are, in the domains
$\tO_{P_\varepsilon^+}$, $\tO_{P_\varepsilon^-}$, $\tO_{P_\varepsilon}$
respectively, as smooth as the flows $\tp^{\pm\pm}(t;\tz;\varepsilon)$
restricted to $S^\pm\times S^\pm\times\T$. Therefore, we can apply to them
classical results of smooth dynamical systems which need regularity assumptions.
\end{remark}

If $\varepsilon=0$, we can provide an explicit expression for the impact map as
follows. The flow $\tp(t;\tz_0;0)$ consists of the uncoupled flows $\phi_\U$ and
$\phi_\X$ described in~(\ref{eq:unperturbed_flows}) but extended by adding the
time $s$ as a state variable. From conditions~\condsC, the phase portrait of
system $\U$ is formed by the continuum of periodic orbits $\Lambda_c$ which, due
to the form of the Hamiltonian $U$, is symmetric with respect to $v=0$.  Hence
the maps $P_0^\pm$ can be written as
\begin{equation*}
P_0^\pm(\tz)=\left( 0,-v,\phi_\X(\alpha^\pm(v);x,y),s+\alpha^\pm(v) \right),
\end{equation*}
where
\begin{equation}
\alpha^\pm(v)=2\mspace{-30mu}\mathop{\int}_0^{(V^\pm)^{-1}(c)}
\mspace{-30mu}\frac{1}{\sqrt{2(c-V^\pm(x))}}dx,\quad c=U(0,v)=\frac{v^2}{2}
\label{eq:alphap}
\end{equation}
are the times taken by the flow $\phi_\U(t;0,\pm v)$, with $v>0$, to reach
$\Sigma^\mp$.
Hence, when $\varepsilon=0$, the impact map takes the form
\begin{equation*}
P_0(\tw)=\left( v,\phi_\X(\alpha( v);x,y),s+\alpha(
v) \right),
\end{equation*}
where
\begin{equation}
\alpha(v)=\alpha^+(\vert v\vert)+\alpha^-(-\vert v\vert)
\label{eq:period_unperturbed_po}
\end{equation}
is the period of the orbit of system~(\ref{eq:field2}) with $c=U(0,v)$, and
$\phi_\X(t;x,y)=\phi^\pm_\X(t;x,y)$ if $(x,y)\in \Sigma^\pm\cup S^\pm$.
\subsection{Impact sequence}
\label{sec:impact_sequence}
Let $\tw \in \tO_{P_\varepsilon}$ and $\varepsilon \ge0$ small enough.
Proceeding as in~\cite{Hog89,GraHogSea12}, we define the direct sequence of
impacts $\tw_i$ associated with the section $\tS$ as
\begin{equation}
\tw_i=
\left\{
\begin{aligned}
&P^+_{\varepsilon}(\tw_{i-1})&&\text{if
}\tw_{i-1}\in\tO_{P_\varepsilon^+}\\
&P^-_{\varepsilon}(\tw_{i-1})&&\text{if
}\tw_{i-1}\in\tO_{P_\varepsilon^-},
\end{aligned}\right.\label{eq:impact_sequence_direct}
\end{equation}
with $i\ge0$ and $\tw_{0}=\tw$.
We also define the inverse sequence of impacts, if they exist, as
\begin{equation}
\tw_{i}=
\left\{
\begin{aligned}
&(P^+_{\varepsilon})^{-1}(\tw_{i+1})&&\text{if
}\tw_{i+1}\in P^+_\varepsilon(\tO_{P_\varepsilon^+})\\
&(P^-_{\varepsilon})^{-1}(\tw_{i+1})&&\text{if
}\tw_{i+1}\in P^-_\varepsilon(\tO_{P_\varepsilon^-}),
\end{aligned}\right.\label{eq:impact_sequence_inverse}
\end{equation}
with $i<0$.
In general, this is a finite sequence, and is defined up to the $n$th iterate
such that
\begin{align*}
\tw_{n}\notin \tO_{P_\varepsilon^+}\cup \tO_{P_\varepsilon^-},\,n>0\\
\tw_{n}\notin
P_\varepsilon^-\left(\tO_{P_\varepsilon^-}\right)\cup
P_\varepsilon^+\left(\tO_{P_\varepsilon^+}\right),\,n<0
\end{align*}
That is, we consider all the impacts with the switching surface $u=0$ of the
trajectory associated with system~(\ref{eq:general_field_scattering}) with
initial condition $\tz$ that are previous to the first impact with $x=0$, both
forwards and backwards in time. When this occurs, then it is possible to extend
the sequence by properly concatenating the flow.

In general, one can  extend the definition of the impact sequence to arbitrary
points in $\tz\in S^\pm\times S^\pm\times\T$, no necessarily  located at $\tS$.
This can be done by flowing $\tz$ by $\tp$ both forwards and backwards in time
until the switching manifold $\tS$ is reached at the points
$\tz_{1}\in\tO_{P_\varepsilon}$ and $\tz_2\in\tO_{P_\varepsilon}$, respectively.
Then, one just considers the direct and inverse impact sequence of associated
with the points $\tz_1$ and $\tz_2$, respectively.

Note that the impact sequence can be used to obtain explicit expressions for the
flows (see~\cite{Gra12} for details).

\section{Invariant manifolds and their persistence}\label{sec:invariant_objects_persistence}
\subsection{Unperturbed case in the extended phase space}\label{sec:invariant_objects_unperturbed_flow}
We consider invariant objects of system~(\ref{eq:general_field_scattering}) when
$\varepsilon=0$.  The cross products of the hyperbolic critical points $Q^\pm$
and the periodic orbits $\Lambda_c$ give rise to two families of invariant
$2$-dimensional   tori $\tTT_c^\pm$ of the form
\begin{align}
\tTT_c^\pm&=\Lambda_c\times Q^\pm\times\T=\nonumber\\
&\left\{
(u,v,x,y,s)
\,\vert\,U(u,v)=c,\,(x,y)=Q^\pm\!,s\in\T
\right\}
\label{eq:2d_tori}
\end{align}
with $0<c\le\bar{c}$. These tori are only continuous manifolds, because of the
singularity of the Hamiltonian $U$ at $u=0$ (see
figure~\ref{fig:norm_hyper_mani_extended_ps}).
\begin{figure}
\begin{center}
\includegraphics[width=1\textwidth]
{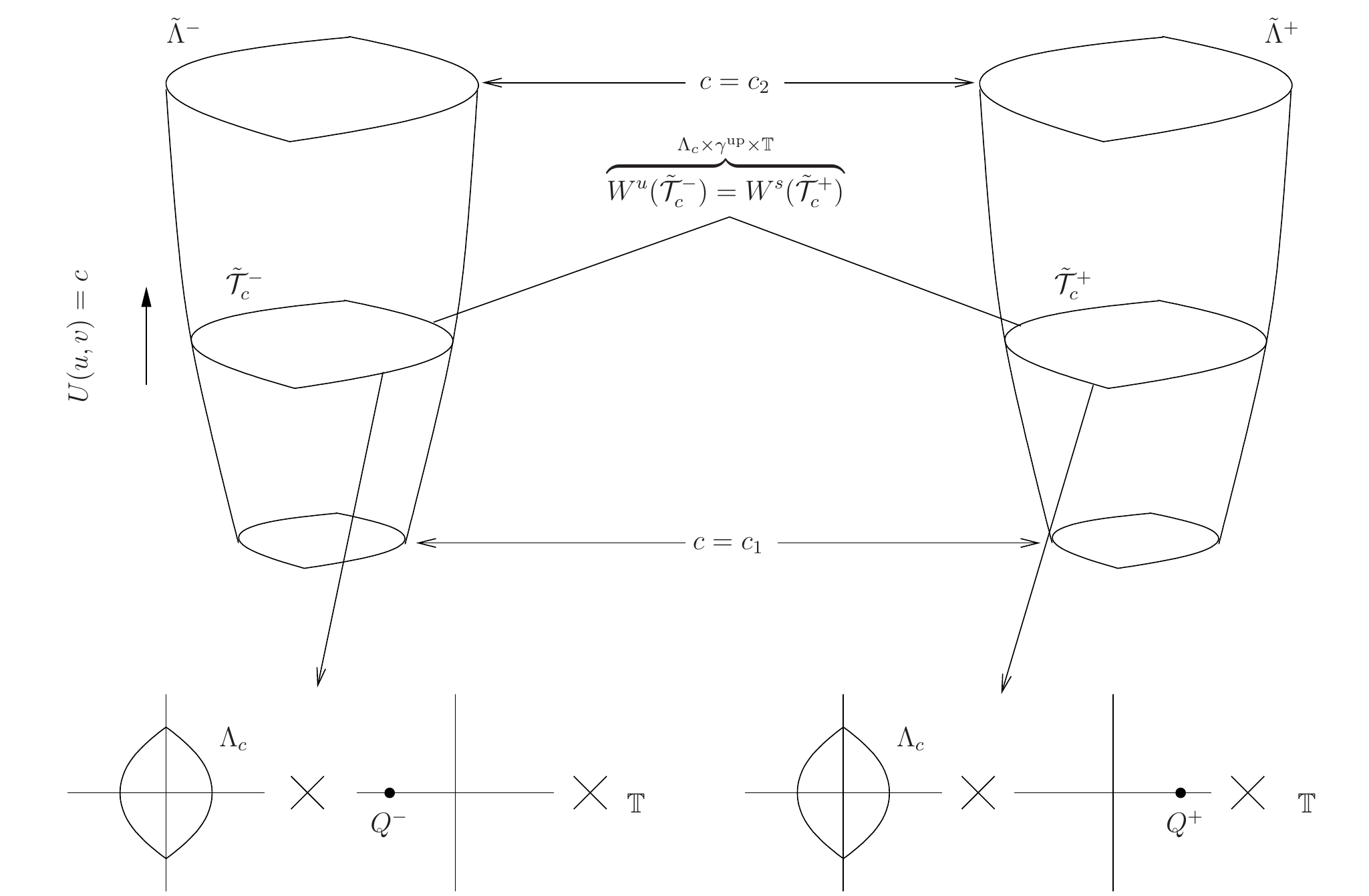}
\end{center}
\caption{Scheme of the manifolds $\tL^\pm$, the tori $\tTT_c^\pm$ and their
invariant manifolds.}
\label{fig:norm_hyper_mani_extended_ps}
\end{figure}
We parameterize $\tTT_c^\pm$ by
\begin{equation}
\tTT_c^\pm=\left\{ (\phi_\U(\theta \alpha(v);0,v),Q^\pm, s )
,\,\theta\in\Tu,\,v\in\RR,\,U(0,v)=c,\,s\in\T \right\},
\label{eq:tori_parameterized}
\end{equation}
where $\alpha(v)$ is given in~(\ref{eq:period_unperturbed_po}),
$\Tu=\RR\backslash\mathbb{Z}$ is the usual circle and $\phi_\U$ is the flow
associated with system~(\ref{eq:field2}). Then the flow $\tp$ restricted to
these tori becomes
\begin{equation*}
\begin{aligned}
\tp(t;&\phi_\U(\theta\alpha(v);0,v),Q^\pm,s;0)\\
&=\left(\phi_\U\left(
\left(\theta+\frac{t}{\alpha(v)}\right)\alpha(v);0,v\right),Q^\pm,s+t\right),\,\forall
t\in\RR,
\end{aligned}
\end{equation*}
and hence $\tTT_c^\pm$ is invariant. 
For each of these invariant tori there exist $3$-dimensional continuous manifolds
\begin{align*}
W^s&(\tTT_c^+)=W^u(\tTT_c^-)\nonumber\\
&=\Lambda_c\times W^s(Q^-)\times\T=\Lambda_c\times W^u(Q^+)\times\T\nonumber\\
&=\left\{
(\phi_\U(\theta\alpha(v);0,v),\su(\xi),s),\vert\,U(0,v)=c,\,\theta\in\Tu,\,\xi\in\RR,\,s\in\T
\right\},
\end{align*}
where $\su(\xi)$, given
in~(\ref{eq:unpert_het_orbit})-(\ref{eq:unpert_het_orbit_parameterization}),
parameterizes the upper heteroclinic connection $\gamma^{\text{up}}$ of system
$\X$ (see figure~\ref{fig:norm_hyper_mani_extended_ps}).  The flow $\tp$
restricted to these manifolds can be written as
\begin{equation*}
\begin{aligned}
\tp(t;&\phi_\U(\theta\alpha(v);0,v),\su(\xi),s;0)\\
&=\left(\phi_\U\left( \left(\theta+\frac{t}{\alpha(v)}\right)
\alpha(v);0,v\right),\su(\xi+t),s+t \right),\,\forall t\in\RR,
\end{aligned}
\end{equation*}
and hence they are invariant. 
Moreover, for any 
$\tz=\left( \phi_\U(\theta\alpha(v);0,v),\su(\xi),
s \right)\in
W^s(\tTT_c^+)=W^u(\tTT_c^-)$, 
there exists two points
\begin{equation*}
\tz^\pm=\left( \phi_\U(\theta\alpha(v);0,v),Q^\pm,s \right)\in\tTT_c^\pm
\end{equation*}
such that
\begin{equation*}
\lim_{t\to\pm\infty}\left\vert
\tp(t;\tz;0)-\tp(t;\tz^\pm;0)\right\vert=\lim_{t\to\pm\infty}\left(
0,0,\su(\xi+t)-Q^\pm,0\right)=0.
\end{equation*}
In addition, as the points $Q^\pm$ are hyperbolic for the flows $\phi^\pm_\X$,
then there exist positive constants $K^\pm$ such that
\begin{equation}
\left\vert \phi(t;\tz;0)-\phi(t;\tz^\pm;0) \right\vert< K^\pm e^{-\lambda^\pm \vert
t\vert},\, t \to\pm\infty,
\label{eq:hyperbolicity_flow}
\end{equation}
where $\pm\lambda^+$ and $\pm\lambda^-$ are the eigenvalues of $D\X^+(Q^+)$ and
$D\X^-(Q^-)$, respectively. \\
Although $W^{u,s}(\tTT_c^\pm)$  are just continuous manifolds, they are the
stable and unstable manifolds of $\tTT_c^\pm$. As they coincide,
$\tg_c:=W^s(\tTT_c^+)=W^u(\tTT_c^-)$ will be a $3$-dimensional heteroclinic
manifold between the tori $\tTT_c^-$ and $\tTT_c^+$.
The lower heteroclinic connection mentioned in condition
C.2 leads to similar heteroclinic manifolds between the tori $\tTT_c^+$
and $\tTT_c^-$.

Following~\cite{DelLlaSea00}, considering all the tori $\tTT^+_c$ and $\tTT^-_c$
together we end up with two $3$-dimensional continuous manifolds
\begin{align}
\tL^\pm&=\bigcup_{c\in[c_1,c_2]}\tTT_c^\pm=\bigcup
_{c\in[c_1,c_2]}\Lambda_c\times Q^\pm\times\T\nonumber\\
&=\left\{ \left( 
u,v,x,y,s
\right),\,c_1\le
U(u,v)\le c_2,\,(x,y)=Q^\pm,\, s\in\T\right\}\nonumber\\
&=\left\{ \left( \phi_\U(\theta\alpha(v);0,v),Q^\pm, s \right),\,\theta\in\Tu,\,c_1\le
U(0,v)\le c_2,\, s\in\T \right\}
\label{eq:normally_hyper_manifolds}
\end{align}
with $0<c_1<c_2<\bar{c}$, shown  schematically in
figure~\ref{fig:norm_hyper_mani_extended_ps}. These manifolds have
$4$-dimensional stable and unstable continuous manifolds given by
\begin{align}
W^{s}&(\tL^+)=W^{u}(\tL^-)=\bigcup_{c\in[c_1,c_2]}W^s(\tTT^+_c)=\bigcup_{c\in[c_1,c_2]}W^u(\tTT^-_c)\nonumber\\
&=\left\{
\left( \phi_\U(\theta\alpha(v);0,v),\su(\xi), s
\right),\,\theta\in\Tu,\,c_1\le U(0,v)\le c_2,\,\xi\in\RR,\,s\in\T
\right\}.
\label{eq:unperturbed_unstable_manifold_po}
\end{align}
As they coincide, $\tg:=W^s(\tL^+)=W^u(\tL^-)$ will be a $4$-dimensional heteroclinic continuous
manifold between the manifolds $\tL^\pm$.

It will be convenient to write the manifolds $\tL^\pm$ in terms of a
\emph{reference manifold} $N$ (see~\cite{DelLlaSea08}) as follows. Let
\begin{equation}
N=\left\{ (\theta,v,s)\in \Tu\times[v_1,v_2]\times\T\right\}
\label{eq:reference_manifold_flow}
\end{equation}
where $c_i=U(0,v_i)$, and consider two homeomorphisms
\begin{equation}
\begin{array}{cccc}
F_0^\pm:&N&\longrightarrow&\tL^\pm\\
&(\theta,v,s)&\longmapsto&(\phi_\U(\theta\alpha(v);0,v),Q^\pm,s).
\end{array}
\label{eq:unperturbed_reference_map}
\end{equation}
Hence the continuous manifolds $\tL^\pm$ are given by $\tL^\pm=F_0^\pm(N)$.
This will later allow us to identify points on the perturbed manifolds
$\tL_\varepsilon^\pm$ in terms of the same coordinates $(\theta,v,s)$ if
$\varepsilon>0$ is small enough.
Note that $F_0^\pm$ are in fact diffeomorphisms as long as
$\theta\in\left( 0,\frac{\alpha^+(v)}{\alpha(v)} \right)\cup\left(
\frac{\alpha^+(v)}{\alpha(v)},1 \right)$ as $\phi_\U(\theta\alpha(v);0,v)$ hits the
switching manifold given by $u=0$ for $\theta=0$,
$\theta=\frac{\alpha^+(v)}{\alpha(v)}$ and
$\theta=1$.

\subsection{Invariant manifolds for the unperturbed impact map}
The fact that the manifolds $\tL^\pm$ are only continuous manifolds will prevent
us from applying classical perturbation theory for hyperbolic manifolds
\cite{Fen72,Fen74,Fen77,HirPugShu77,DelLlaSea00} to study their
persistence for $\varepsilon>0$. In the smooth case, the usual tool to prove persistence following a
non-autonomous periodic perturbation is the stroboscopic Poincar\'e map, which
integrates the system during a certain time $T$, the period of the perturbation.
However, in our case, such a map becomes unwieldy because, for a given time, the
number of occasions that the switching manifold can be crossed is unknown and can even be arbitrarily large. 

Instead, we will consider the Poincar\'e impact map defined
in~\S\ref{sec:impact_map_scattering}, which is a smooth map as regular as the
flows $\tp^{\pm\pm}$ restricted to their respective domains. We first describe
the invariant objects introduced above for the impact map restricted to $\tS^+$
when $\varepsilon=0$. As mentioned in~\S\ref{sec:impact_map_scattering}, we will
identify the switching manifold $\tS^+$ with the set $[v_1,v_2]\times \RR^2\times\T$ and omit the repetition of
the coordinate $u=0$ for points in $\tS$. 
We then consider the unperturbed impact map
\begin{equation*}
P_0:\tO_{P_0}\cap\left\{ 
(v,x,y,s)\in\RR^3\times\T,\,v>0 \right\}\longrightarrow \RR^3\times\T.
\end{equation*}
Taking into account that 
\begin{equation}
P_0(v,Q^\pm,s)=(v,Q^\pm,s+\alpha(v)),
\label{eq:unperturbed_impact_map}
\end{equation}
and letting $U\left(0,v \right)=c$, the invariant tori $\tTT_c^\pm$ give rise to 
smooth invariant curves
\begin{align}
\tC_c^\pm&=\left\{ v \right\}\times Q^\pm\times \T\nonumber\\
&=\left\{ 
(v,x,y,s)
\in\RR^3\times\T,\,U(0,v)=c,\, (x,y)=Q^\pm\right\}
\label{eq:invariant_cylinders}
\end{align}
with $0<c\le\bar{c}$. For those values of $c$ such that $m\alpha(v)=nT$, for some natural numbers $n$ and
$m$, the curves $\tC_c^\pm$ are filled by periodic points. The rest are formed
by points whose trajectories are dense in $\tC_c^\pm$.

For each of these curves there exist $2$-dimensional (locally smooth) continuous manifolds
\begin{align*}
W^{s}(\tC_c^+)&=W^u(\tC_c^-)\\
&=\left\{ \left( v,\su(\xi),s \right),\,U(0,v)=c,\,\xi\in\RR,\,s\in\T \right\}
\end{align*}
which are invariant under $P_0$:
\begin{equation*}
P_0(v,\su(\xi),s)=\left( v,\su(\xi+\alpha(v)),s+\alpha(v)) \right)\in
W^u(\tC_c^-)=W^s(\tC_c^+).
\end{equation*}
Moreover, due to the hyperbolicity of the points $Q^\pm$
(see~(\ref{eq:hyperbolicity_flow})), for any $\tw=(v,\su(\xi),s)\in
W^u(\tC_c^-)=W^s(\tC_c^+)$, there exist $\tw^\pm=(v,Q^\pm,s)\in \tC_c^\pm$ such
that
\begin{equation}
\vert P^n_0(\tw)-P_0^n(\tw^\pm)\vert =\left\vert \left(
0,\su\left(\xi+n\alpha\left(v\right)\right)-Q^\pm,0
\right)\right\vert< \hat{K}^\pm (\hlambda^{\pm})^{\vert
n\vert},\,n\to\pm\infty,
\label{eq:hyperbolicity_map}
\end{equation}
where $\hat{K}^\pm=K^\pm e^{-\lambda^\pm\xi}$,
$0<\hlambda^\pm=e^{-\lambda^\pm\alpha(v)}<1$ and $\lambda^\pm,K^\pm$ are defined
in~(\ref{eq:hyperbolicity_flow}).

Proceeding similarly as with the flow, we now consider the union over $c$ of all
the curves $\tC^\pm_c$ which become the smooth cylinders
\begin{equation}
\tG^\pm=\bigcup_{c\in[c_1,c_2]}\tC_c^\pm=\left\{
(v,Q^\pm,s)\,\vert\,v_1<v<v_2,\,s\in\T \right\}\\
\label{eq:normally_hype_man_impact_map}
\end{equation}
with $0<c_i\le\bar{c}$ and $c_i=U(0,v_i)$, $i=1,2$, which are invariant under
$P_0$. Note that the manifolds $\tG^\pm$ correspond to the intersection
\begin{equation*}
\tL^\pm\cap\tS^+=\left\{ 0 \right\}\times\tG^\pm.
\end{equation*}

Taking into account that $\tG^\pm$ are compact manifolds with boundaries given
by $v=v_1$ and $v=v_2$, there exist constants $\bmu >1$, $0<\blambda^\pm<1$,
$\blambda^\pm<\frac{1}{\bmu}$ (in fact $\bmu$ can be taken as close to one as
desired) such that, for all $\tw\in \tG^\pm$
\begin{equation}
\begin{array}{lcll}
 w\in E_{\tw}^s &\Longleftrightarrow & \vert DP_0^n(\tw) w\vert \le
K^\pm (\blambda^\pm)^n \vert w\vert,&\, n\ge0\\
 w\in E_{\tw}^u &\Longleftrightarrow & \vert DP_0^n(\tw) w\vert \le
K^\pm (\blambda^\pm)^{-n}\vert w\vert,&\,n\le 0\\
 w\in T_{\tw}\tG^\pm &\Longleftrightarrow& \vert DP^n_0(\tw) w\vert
\le
K^\pm (\bmu)^{\vert n\vert}\vert w\vert,&\,n\in\mathbb{Z}
\end{array}
\label{eq:normally_hyperbolicity_map}
\end{equation}
where $E_{\tw}^s, E_{\tw}^u$ and $T_{\tw}\tG^\pm$ are the stable, unstable and tangent bundles of $\tG^\pm$ respectively.
Assuming that $\alpha(v)$ is an increasing function of $v$, we can take
\begin{equation}
\blambda^\pm=e^{-\alpha(v_1)\lambda^\pm}.
\label{eq:blambda}
\end{equation}
Hence, $\tG^\pm$ are $C^\infty$ (as regular as the flows) normally hyperbolic manifolds for the unperturbed impact map $P_0$, with stable and unstable
invariant manifolds
\begin{align*}
W^u(\tG^-)&=W^s(\tG^+)=\bigcup_{c\in[c_1,c_2]}W^u(\tC_c^-)=\bigcup_{c\in[c_1,c_2]}W^s(\tC_c^+)\\
&=\left\{ (v,\su(\xi),s),\,v_1\le v\le v_2,\,\xi\in\RR,\,s\in\T
\right\}.
\end{align*}

\subsection{Perturbed case}\label{sec:invariant_objects_perturbed}
Let us now consider the persistence of the invariant manifolds introduced in the
previous section when $\varepsilon>0$ is small. We first focus on the normally
hyperbolic manifolds, $\tG^\pm$, for the map $P_0$.  As mentioned in
Remark~\ref{rem:smoothness_of_impact_map}, the impact map $P_{\varepsilon}$ is,
in $\tO_{P_\varepsilon}$, as regular as the flows $\tp^{\pm\pm}$ restricted to
$S^\pm\times S^\pm\times\T$. Thus, the persistence of the normally hyperbolic
manifolds $\tG^\pm$ for $\varepsilon>0$ comes from the theory of normally
hyperbolic manifolds \cite{HirPug70,Fen72,Fen74,HirPugShu77,DelLlaSea08}. This
guarantees the existence of normally hyperbolic invariant manifolds
$\tG_\varepsilon^\pm$ and $\mathcal {C} ^r$ diffeomorphisms (with $r$ as big as we want)
\begin{align*}
G_\varepsilon^\pm&:\RR\times \T\longrightarrow \RR^3\times\T,\\
g_\varepsilon^\pm&:\RR\times \T\longrightarrow \RR^2
\end{align*}
such that the points at the manifolds $\tG_\varepsilon^+$ are parameterized by
$\tw^\pm=G_\varepsilon^\pm(v,s)\in\tG_\varepsilon^\pm.$ These parameterizations
are not unique. We can make them unique by taking $G_\varepsilon^\pm$ to
be the identity in the $v$ and $s$ coordinates, 
\begin{equation}
\Pi_{v,s}\left(G_\varepsilon^\pm\right)=Id,
\label{eq:parameterization_identity}
\end{equation}
that is,
\begin{equation*}
G_\varepsilon^\pm(v,s)=(v,g_\varepsilon^\pm(v,s),s).
\end{equation*}
When $\varepsilon=0$, these maps coincide with the parametrization defined
in~\eqref{eq:normally_hype_man_impact_map}. Therefore $g_0^\pm(v,s)=Q^\pm$ and
$\tG_\varepsilon^+$ is $\varepsilon$-close to $\tG^+$. In particular,
$\tG_\varepsilon^\pm\subset \tO_\varepsilon$.

The manifolds $\tG_{\varepsilon}^\pm$ have $\mathcal {C} ^r$ local stable and
unstable manifolds
\begin{equation*}
W^{s,u}( \tG_\varepsilon^\pm )
\end{equation*}
$\varepsilon$-close to $W^{s,u}( \tG^\pm )
$, satisfying:
\begin{itemize}
\item
For every $\tw^{s}\in W^{s}(\tG_\varepsilon^+)$
there exists $\tw^+\in \tG_\varepsilon^+$ such that:
\begin{equation}
\vert P^n_\varepsilon(\tw^{s})-P^n_\varepsilon(\tw ^+)\vert <
K^+\left(\blambda  ^+ +O(\varepsilon)\right)^n\to 0,
\,n\to+\infty,
\label{eq:perturbed_hyperbolicity_map2}
\end{equation}
\item
For every $\tw^{u}\in W^{u}(\tG_\varepsilon^+)$
there exists $\tw^+\in \tG_\varepsilon^+$ such that:
\begin{equation}
\vert P^n_\varepsilon(\tw^{u})-P^n_\varepsilon(\tw ^+)\vert <
K^+\left(\blambda  ^+ +O(\varepsilon)\right)^{-n}\to 0,
\,n\to-\infty,
\label{eq:perturbed_hyperbolicity_map2u}
\end{equation}
\end{itemize}
and $0<\blambda^+<1$ is  the constant given
in~(\ref{eq:normally_hyperbolicity_map}).
Analogous properties hold for the manifold $\tG_\varepsilon^-$.

\begin{remark}
In general, the theorem of persistence of normally hyperbolic manifolds only
gives local invariance for the perturbed manifold. Nevertheless,
following~\cite{KunKupYou97}, one can use the change of variables given
in~\cite{Lev91} to obtain the impact map in symplectic coordinates.
Therefore, one can apply the twist theorem to the perturbed impact map.  This
gives that those curves $\tC_c\subset\tG_0^\pm$ with
 $\alpha(\sqrt{2c})/T$ far away from rational numbers persist as
invariant curves.
These provide invariant boundaries for the perturbed manifolds
$\tG_\varepsilon^\pm$, and hence these are compact invariant manifolds.
\label{rem:compact_tG}
\end{remark}

We now consider the existence of manifolds equivalent to $\tG^\pm_\varepsilon$
for the flow $\tp$. More precisely, we are interested in obtaining the perturbed
version of the  manifolds $\tL^\pm$ in terms of the reference manifold $N$ given
in~(\ref{eq:reference_manifold_flow}).\\

\begin{prop}\label{prop:persistence_Lambda}
Let $\tL^\pm$ and $W^{s,u}(\tL^\pm)$ be the manifolds described in
\S\ref{sec:invariant_objects_unperturbed_flow} invariant for the unperturbed
system~(\ref{eq:general_field_scattering}). Then, there exist continuous maps
\begin{equation*}
F_\varepsilon^\pm:N\longrightarrow \RR^4\times\T,
\end{equation*}
where $F_0^\pm$ are given in \eqref{eq:unperturbed_reference_map},
that are Lipschitz in $\varepsilon$, such that the $C^0$ manifolds
\begin{equation}
\tL_\varepsilon^\pm=F_\varepsilon^\pm(N)
\label{eq:perturbed_Lambdas}
\end{equation}
are invariant under $\tp$ and $\varepsilon$-close to $\tL^\pm$.
Moreover, there exist $C^0$ manifolds $W^{s,u}(\tL_\varepsilon^{\pm})$,
$\varepsilon$-close to $W^{s,u}(\tL^{+})$, satisfying:
\begin{itemize}
\item
for any
$\tz^{s}=(z^{s},s^{s})\in
W^{s}(\tL_\varepsilon^{+})$ there exists
$\tz^{+}=(z^{+},s^{+})\in\tL_\varepsilon^{+}$ 
such that
\begin{equation}
\vert \tp(t;\tz^{s};\varepsilon)-\tp(t;\tz^{+};\varepsilon)\vert
< \bar{K}^{+}e^{-(\lambda^{+}+O(\varepsilon))\vert t\vert},\,t\to+\infty,
\label{eq:perturbed_hyperbolicity_flow}
\end{equation}
\item
for any
$\tz^{u}=(z^{u},s^{u})\in
W^{u}(\tL_\varepsilon^{+})$ there exists
$\tz^{+}=(z^{+},s^{+})\in\tL_\varepsilon^{+}$ 
such that
\begin{equation}
\vert \tp(t;\tz^{u};\varepsilon)-\tp(t;\tz^{+};\varepsilon)\vert
< \bar{K}^{+}e^{-(\lambda^{+}+O(\varepsilon))\vert t\vert},\,t\to-\infty,
\label{eq:perturbed_hyperbolicity_flowu}
\end{equation}
\item
$s^{s}=s^{u}=s^{+}$ 
\end{itemize} 
where $\bar{K}^{+}>0$, and $\lambda^{+}>0$ is given
in~(\ref{eq:hyperbolicity_flow}). 
Analogous properties hold for the manifold $\tL_\varepsilon^-$.
\end{prop}
\begin{proof}
The maps $F_\varepsilon^\pm$ are obtained by flowing the manifolds
$\tG_\varepsilon^\pm$. Let $(\theta,v,s)\in N$ and consider
\begin{align*}
\tw^\pm&=G^\pm_\varepsilon(v,s)=(v,g_\varepsilon^\pm(v,s),s)\in\tG_\varepsilon^\pm\\
\tw^\pm_1&=(\omega_1^\pm,s_1^\pm)=P_\varepsilon(\tw^\pm),
\end{align*}
and
\begin{equation*}
\tz^\pm=(0,\tw^\pm).
\end{equation*}
Then we define
\begin{align}
F_\varepsilon^\pm(\theta,v,s)&=
\tp\left(\left(s_1^\pm-s\right)\theta;(0,\tw^\pm);\varepsilon\right)\nonumber\\
&=\tp\left( \left( s_1^\pm-s
\right)\theta;0,v,g_\varepsilon^\pm(v,s),s;\varepsilon \right),\,\theta\in[0,1],
\label{eq:perturbed_reference_map}
\end{align}
which are smooth maps as long as the flow does not hit $u=0$, which occurs at
\begin{equation}
\theta= 0,\,\theta=
\frac{\Pi_s(P_\varepsilon^+(\tw^\pm))-s}{s_1^\pm-s},\,\theta=
1.
\label{eq:discontinuity_points}
\end{equation}
These provide the $3$-dimensional continuous manifolds given
in~(\ref{eq:perturbed_Lambdas}) which are invariant by the flow $\tp$. This can
be seen using the fact that
\begin{equation}
F_\varepsilon^\pm(1,v,s)=\tp(s_1^\pm-s;\tz^\pm;\varepsilon)=\left( 0,P_\varepsilon(\tw^\pm)
\right)\in\left\{ 0
\right\}\times\tG_\varepsilon^\pm,
\label{eq:Fproperty}
\end{equation}
(see~\cite{Gra12} \S 6.3.6 for details).
Note that, when $\varepsilon=0$, $\tz^\pm=(0,\tw^\pm)=(0,G_0^\pm(v,s))=\left(0, v,Q^\pm,s
\right)\in\tG_0^\pm$ and $(s_1^\pm-s)=\alpha(v)$.
Therefore,
\begin{equation*}
F^\pm_0(\theta,v,s)=\tp\left( \alpha(v)\theta;\tz^\pm;0\right)=
\left(\phi_\U(\theta\alpha(v);0,v),Q^\pm,s  \right)
\end{equation*}
which coincide with the parameterizations  $F_0^\pm$ defined
in~(\ref{eq:unperturbed_reference_map}).

Let us now consider the stable and unstable manifolds
$W^{s,u}(\tL_\varepsilon^\pm)$. We illustrate the method for
$W^s(\tL_\varepsilon^+)$.  Let $(\theta,v,s)\in N$ and consider
\begin{equation*}
\tw^s=(\omega^s,s^s)\in W^s(\tw^+)
\end{equation*}
with
\begin{equation*}
\tw^+=G_\varepsilon^+(v,s)=(v,g^+_\varepsilon(v,s),s)\in\tG_\varepsilon^+,
\end{equation*}
which satisfy (\ref{eq:perturbed_hyperbolicity_map2}).  Defining
\begin{equation*}
\tau^+=\left( \Pi_s\left( P_\varepsilon(\tw^+) \right)-s\right)\theta
\end{equation*}
we consider the point
\begin{equation*}
\tz^+=\tp(\tau^+;(0,\tw^+);\varepsilon)=F_\varepsilon^+(\theta,v,s)\in\tL_\varepsilon^+.
\end{equation*}
We now define
\begin{equation}
\tau'=s-s^s
\label{eq:delay_stable_manifold}
\end{equation}
and the point
\begin{equation}
\tz^s=\tp(\tau^++\tau';(0,\tw^s);\varepsilon).
\label{eq:point_stable_manifold}
\end{equation}
We now show that it belongs to the stable fibre of the point $\tz^+$.  Let
$\tw^+_i$ and $\tw^s_i$ be the impact sequences associated with the points
$\tw^+$ and $\tw^s$, respectively.  Using the (smooth) intermediate map
$P^+_\varepsilon$ defined in~(\ref{eq:impact_mapp}) we have that
\begin{align*}
\tw^+_{2i+1}&=P_\varepsilon^+(\tw^+_{2i}),\quad\tw_{2i}^+=P^i_\varepsilon(\tw^+)\\
\tw^s_{2i+1}&=P_\varepsilon^+(\tw^s_{2i}),\quad\tw^s_{2i}=P^i_\varepsilon(\tw^s),
\end{align*}
and hence, by~(\ref{eq:perturbed_hyperbolicity_map2}),
\begin{equation*}
\left| \tw^s_i-\tw^+_i\right|<K^+ (\blambda^+)^i,\, i\to \infty.
\end{equation*}
The constant $K^+$ may differ from the one used
in~(\ref{eq:perturbed_hyperbolicity_map2}). To simplify the notation we take the
maximum of both and use the same symbol.  Consequently, the sequences $s^s_i$ and
$s^+_i$ satisfy
\begin{equation*}
\vert s_i^s-s_i^+\vert < K^+(\blambda^+)^i,\,i\to\infty,
\end{equation*}
for $K^+>0$ again properly redefined. In other words, there exist two sequences
of times, $t_i^{s}=s_i^s-s^s$ and $t_i^+=s_i^+-s$, where the impacts occur, such
that
\begin{equation}
\left\vert
\tp(t_i^s;(0,\tw^s);\varepsilon)-\tp(t_i^+;(0,\tw^+);\varepsilon)\right\vert
< K^{+}(\blambda^+)^i,\,i\to\infty
\label{eq:hyperbolicity_aux}
\end{equation}
and
\begin{equation*}
\vert t_i^s-t_i^+-(s^s-s)\vert < K^+(\hlambda^+)^i,\,i\to\infty.
\end{equation*}
The fact that the perturbed manifold $\tG_\varepsilon^+$ is compact (see
remark~\ref{rem:compact_tG}) ensures that the sequences $t_{i+1}^s-t_{i}^s$ and
$t_{i+1}^+-t_{i}^+$ are bounded, both from above and below,
\begin{equation}
\begin{aligned}
\alpha^{\pm}(v)+O(\varepsilon)<t^s_{i+1}-t^s_i<\alpha^\pm(v)+O(\varepsilon)\\
\alpha^\pm(v)+O(\varepsilon)<t^+_{i+1}-t^+_i<\alpha^\pm(v)+O(\varepsilon).
\end{aligned}
\label{eq:bounded_diff}
\end{equation}
Hence, using the lower bound, if $t$ is large enough, we can always find $i$
such that
\begin{equation}
b_i:=\min(t_i^+-\tau^+,t_i^s-\tau^+-\tau')<t<\max(t_{i+1}^+-\tau^+,t_{i+1}^s-\tau^+-\tau'):=a_{i+1}
\label{eq:t_assumption}
\end{equation}
and write the flows
\begin{align}
\tp(t;\tz^+;\varepsilon)&=\tp(t-t_i^++\tau^+;(0,\tw^+_i);\varepsilon)\label{eq:phi_zp}\\
\tp(t;\tz^s;\varepsilon)&=\tp(t-t_i^s+\tau^++\tau';(0,\tw^s_i);\varepsilon)\label{eq:phi_zs}.
\end{align}
Let us now assume that
\begin{equation*}
a_i<t<b_{i+1}.
\end{equation*}
For $t\in(a_i,b_{i+1})$, both flows~(\ref{eq:phi_zp}) and~(\ref{eq:phi_zs}) are
located in the same domain $S^\pm\times S^+\times\T$ and hence, the the function
\begin{equation*}
\mathfrak{u}(t)=\vert \tp(t;\tz^+;\varepsilon)-\tp(t;\tz^s;\varepsilon)\vert
\end{equation*}
is a smooth function in all its variables because so are the flows $\tp^{\pm+}$.
Note that no impacts occur in the interval $(a_i,b_{i+1})$.  

Let $\mathfrak{K}>0$ be the largest Lipschitz constant of the  two vector fields; then,
for $t\in (a_i,b_{i+1})$ we have
\begin{equation*}
\mathfrak{u}(t)\le K^+(\blambda^+)^i+\int_{a_i}^t\mathfrak{K} u(t)dt.
\end{equation*}
Applying Gronwall's Lemma, we obtain
\begin{equation*}
\mathfrak{u}(t)\le K^+(\blambda^+)^ie^{\mathfrak{K}(b_{i+1}-a_i)}.
\end{equation*}
Using~(\ref{eq:bounded_diff}), the difference $b_{i+1}-a_i$ is bounded by
$\max(\alpha^+(v),\alpha^-(-v))+O(\varepsilon)$, and hence
\begin{equation*}
a_i=i\left(\max(\alpha^+(v),\alpha^-(v))+O(\varepsilon)\right).
\end{equation*}
Then, recalling the definition of $\blambda^+$ given in~(\ref{eq:blambda}) and
assumption~(\ref{eq:t_assumption}), there exists a positive constant $K^+$
(suitably redefined) such that
\begin{equation*}
\vert \phi(t;\tz^{s};\varepsilon)-\phi(t;\tz^{+};\varepsilon)\vert
< K^{+}e^{-\lambda^{+}\vert t\vert},\,t\to\infty,
\end{equation*}
which is what we wanted to prove. If $t\in(b_i,a_i)$ (equivalently for
$t\in(b_{i+1},a_{i+1})$), the flows are not located at the same domains
$S^\pm\times S^\pm\times\T$ and we can not use the Lipschitz constants of the
fields to bound $\mathfrak{u}(t)$. However, as the difference between the fields
defined in $S^\pm\times S^\pm\times\T$ is bounded, one can find some constant
$K_1$ such that $\mathfrak{u}(t)<(a_i-b_i)K_1$.  Since
\begin{equation*}
\left( t_i^+-\tau^+ \right)-\left( t_i^s-\tau^+-\tau'
\right)<K^+\left( \hlambda^+ \right)^i\longrightarrow 0,
\end{equation*}
we also see that $\mathfrak{u}(t)\to 0$ exponentially.

In order to see that $W^{s,u}(\tL_\varepsilon^\pm)$ are $\varepsilon$-close to
$W^{s,u}(\tL^\pm)$ we recall that so are the manifolds
$\tG_{\varepsilon}^\pm$ and $\tG^\pm$ and $W^{s,u}(\tG_\varepsilon^\pm)$ and 
$W^{s,u}(\tG^\pm)$. This implies that $\tau'$ given
in~(\ref{eq:delay_stable_manifold}) is of order $\varepsilon$. Now, since, for
$\varepsilon=0$,
\begin{align*}
\tw^s&=(v,\su(\xi),s)\\
\tw^+&=(v,Q^+,s)\\
\tz^s&=(\phi_\U(\tau^+;0,v),\su(\xi+\tau^+),s+\tau^+)\\
\tz^+&=(\phi_\U(\tau^+;0,v),Q^+,s+\tau^+)
\end{align*}
and the flow $\tp(t;\tz;\varepsilon)$ is $\varepsilon$-close to $\tp(t;\tz;0)$,
we find that, for $\varepsilon>0$,
\begin{equation}\label{eq:zsz+0}
\begin{array}{lcr}
\tz^s&=&(\phi_\U(\tau^+;0,v),\su(\xi+\tau^+),s+\tau^+)+O(\varepsilon)\\
\tz^+&=&(\phi_\U(\tau^+;0,v),Q^+,s+\tau^+)+O(\varepsilon).
\end{array}
\end{equation}
Hence, the manifolds $W^{u,s}(\tL_\varepsilon^\pm)$ are $\varepsilon$-close to
the unperturbed manifolds $W^{u,s}(\tL^\pm)$.
\end{proof}

\begin{remark}\label{rem:sigma_not_isochrone}
The section $\tS$ does not necessary have to
be an isochrone. By introducing in~(\ref{eq:point_stable_manifold}) the delay
$\tau'$ given in~(\ref{eq:delay_stable_manifold}) we obtain points $\tz^s$ that
belong to the isochrone of $\tz^+$.
\end{remark}

\begin{remark}
The direct impact
sequences of points $\tw^+$ and $\tw^s$ are
infinite. Since
$\tw^+\in\tG_\varepsilon^+$ and hence, as $\tG_\varepsilon^+$ is
invariant and $\varepsilon$-close to $\tG_0^+$, it is contained in
$\Sigma^+\times S^+\times\T$, so the flow $\tp(t;(0,\tw^+);\varepsilon)$ never crosses
the switching manifold associated with $x=0$.\\
Similarly, if $\tw^s$ is chosen to be in $\Sigma^+\times S^+\times\T$,
then the flow $\tp(t;(0,\tw^s);\varepsilon)$ approaches $\tL_\varepsilon^+$ for
$t>0$, and thus never crosses the switching manifold $x=0$.  It may happen
 that $W^s(\tL_\varepsilon^+)$ crosses $x=0$ more than once backwards in
time. In this case, $\tw^s$ has to be chosen in the piece of
$W^s(\tG_\varepsilon^+)$ ``closest to'' $\tG_\varepsilon^+$.
\end{remark}

Finally, properties~\eqref{eq:perturbed_hyperbolicity_flow} and \eqref{eq:perturbed_hyperbolicity_flowu} allow us to refer
to $W^{s,u}(\tL_\varepsilon^\pm)$ as stable and unstable manifolds of the
invariant manifolds $\tL_\varepsilon^\pm$.

\section{Scattering map}\label{sec:scattering_map}
The \emph{scattering map}~\cite{DelLlaSea00}, also called the \emph{outer map},
is an essential tool in the study of Arnol'd diffusion.  The diffusion mechanism
in our system when $\varepsilon>0$ consists of trajectories that follow
heteroclinic chains between the manifolds $\tL_\varepsilon^\pm$ such that energy
may increase at each heteroclinic link. The main novelty in the mechanism we
present here, is that we have a scattering map between two different normally
hyperbolic $C^0$ manifolds.  Therefore, the scattering map consists on
identifying points at the invariant manifolds $\tL_\varepsilon^\pm$ via
heteroclinic connections as follows. Let $\tz^\pm \in\tL^\pm_\varepsilon$ and
assume that there exists $\tz^*\in W^u(\tL^-_\varepsilon)\cap
W^s(\tL^+_\varepsilon)$ such that
\begin{equation*}
\lim_{t\to \pm}\left| \tp(t;\tz^\pm;\varepsilon)-\tp(t;\tz^*;\varepsilon)
\right|\to 0.
\end{equation*}
Then, the scattering map becomes
\begin{equation*}
\begin{array}{cccc}
\Su_\varepsilon:&\tL_\varepsilon^-&\longrightarrow &\tL_\varepsilon^+\\
&\tz^-
&\longmapsto&\tz^+
\end{array}.
\end{equation*}
The heteroclinic manifold $W^u(\tL^-_\varepsilon)\cap W^s(\tL^+_\varepsilon)$ is
generated by the upper heteroclinic connection
$\su$~\eqref{eq:unpert_het_orbit_parameterization} of the unperturbed system.
Sufficient conditions for its existence will be studied
in~\S\ref{sec:transverse_intersection} by adapting the Melnikov procedure
described in~\cite{DelLlaSea06} to the piecewise-smooth nature of our problem.

Similarly, one can obtain sufficient conditions for the existence of the
heteroclinic manifold $W^s(\tL^-_\varepsilon)\cap W^u(\tL^+_\varepsilon)$, which
is born from the lower heteroclinic connection of the unperturbed system, $\sd$ by 
considering the scattering map
\begin{equation*}
\Sd_\varepsilon:\,\tL^+_\varepsilon\longrightarrow
\tL^-_\varepsilon.
\end{equation*}
Note that for $\varepsilon=0$ the composition of these maps becomes the identity.

A first order study of these maps (\S\ref{sec:first_order_scattering}) allows us
to identify those trajectories which will land on higher energy levels of the
target manifold. Hence we can construct proper heteroclinic chains with
increasing energy levels. As in~\cite{DelLlaSea06}, the concatenation of these
chains is done by the so-called \emph{inner map}, which is obtained by studying
the dynamics inside the manifolds.  We believe that the combination of the
dynamics of the two scattering maps $\Su_\varepsilon$,  $\Sd_\varepsilon$ and
the inner dynamics in $\tL_\varepsilon ^+$,  $\tL_\varepsilon ^-$ gives
more possibilities for diffusion than the mechanism in~\cite{DelLlaSea06}.
However, in this paper we restrict ourselves to the study of the scattering
maps, and we leave the study of the  dynamics inside the manifolds  for future
work.

\subsection{Transverse intersection of the stable and unstable manifolds}\label{sec:transverse_intersection}
We now study sufficient conditions for the intersection of the stable and
unstable manifolds of $\tL_\varepsilon^+$ and $\tL_\varepsilon^-$ when
$\varepsilon>0$.  The following result, equivalent to Proposition~$9.1$
in~\cite{DelLlaSea06}, provides sufficient conditions for both manifolds to
intersect ``transversally" in a 3-dimensional
heteroclinic manifold which can be parameterized by the coordinates
$(\theta,v,s)\in\T\times[v_1,v_2]\times\Tu$.  Recalling that these
manifolds are only piecewise-smooth continuous and Lipschitz at the  points given
in~\eqref{eq:discontinuity_points}, the tangent space is not defined everywhere and
hence, the notion of transversality has to be adapted. In fact, what is
important for us is that the intersection is locally unique. The continuity of the
system gives us that, in those points where differentiability is lost,
the intersection is the unique lateral limit of unique intersections.
This will provide robustness of intersections under small perturbations.
\begin{prop}\label{prop:heteroclinic_intersection}
Let~(\ref{eq:unperturbed_unstable_manifold_po}) be a parameterization of the
unperturbed heteroclinic manifold $W^u(\tL^-)=W^s(\tL^+)$, and assume that
there exists an open set $J\subset N$ such that, for all $(\theta_0,v_0,s_0)\in
J$, the function
\begin{equation}
\zeta\longmapsto M(\zeta,\theta_0,v_0,s_0),
\label{eq:melnikov_function}
\end{equation}
with
\begin{equation}
M(\zeta,\theta_0,v_0,s_0):=\int_{-\infty}^\infty\left\{ X,h \right\}
\left(\phi_{\U}\left(\theta_0\alpha(v_0)+\zeta+t;0,v_0\right),\su(t),s_0+\zeta+t \right)dt,
\label{eq:Melnikov_def}
\end{equation}
has a simple zero at $\zeta=\bt(\theta_0,v_0,s_0)$. Then, the manifolds
$W^s(\tL_\varepsilon^+)$ and $W^u(\tL_\varepsilon^-)$ intersect transversally.
Moreover, there exist open sets $\mathcal {J}^-\subset \tL_\varepsilon^-$ and $\mathcal {J}^+\subset \tL_\varepsilon^+$
such that  for any point $\tz^-_\varepsilon \in \mathcal {J}^-$,   and
$\tz^+_\varepsilon\in\mathcal {J}^+$ there exists a locally unique point $\tz^*(\theta_0,v_0,s_0;\varepsilon)\in
W^s(\tL_\varepsilon^+)\cap W^u(\tL_\varepsilon^-)$ such that
\begin{equation*}
\lim_{t\to\pm\infty}\tp(t;\tz^*;\varepsilon)-\tp(t;\tz^\pm;\varepsilon)=0.
\end{equation*}
\end{prop}

\begin{proof}
We first study the intersection of the $W^s(\tL_\varepsilon^+)$ and
$W^u(\tL_\varepsilon^-)$ with the section given by $x=0$,
$\RR^2\times\Sigma\times\T$. We consider a point at the
intersection between the unperturbed heteroclinic connection and this section.
We write this point in terms of the parameters in the reference manifold
$(\theta_0,v_0,s_0)\in N$ as
\begin{equation*}
\begin{aligned}
\tz_0(&\theta_0,v_0,s_0):=\left(
\phi_{\U}(\theta_0\alpha(v_0);0,v_0),\su(0),s_0\right)\\
&=\left(
\phi_{\U}(\theta_0\alpha(v_0);0,v_0),0,y_h,s_0\right)
\in \left\{ x=0 \right\}\cap W^u(\tL^-)=\left\{ x=0 \right\}\cap W^s( \tL^+),
\end{aligned}
\end{equation*}
where $(0,y_h)=\su(0)$ and $\phi_{\U}(t;0,v_0)$ is the solution of system~(\ref{eq:field2})
such that $\phi_{\U}(0;0,v_0)=(0,v_0)$. We now introduce a fourth parameter,
$\zeta\in\RR$, in the parameterization of $\tz_0$ as follows
\begin{align}
\tz_0( \theta_0+\frac{\zeta}{\alpha(v_0)},v_0,s_0+\zeta)&:=\left(
\phi_{\U}\left( \theta_0\alpha(v_0)+\zeta;0,v_0\right),0,y_h,s_0+\zeta\right)\label{eq:z0}\\
&\in \left\{ x=0 \right\}\cap W^u(\tL^-)=\left\{ x=0 \right\}\cap W^s(
\tL^+)\nonumber.
\end{align}
Let us consider the line 
\begin{equation*}
\tilde{N}=\left\{\tz_0+l
(0,0,0,1,0),\,l\in\RR  \right\}\subset
\RR^2\times\Sigma\times\T.
\end{equation*}
As the perturbed manifolds are $\varepsilon$-close to the unperturbed ones, if
$\varepsilon>0$ is small enough this
line intersects transversally the manifolds $W^{s/u}(\tL_\varepsilon^{\pmm})$ at
two points
\begin{equation*}
\tz^{u/s}=\tz_0+\left( 0,0,0,O\left( \varepsilon
\right),0 \right)=W^{s/u}(\tL_{\varepsilon}^{\pmm})\cap \tilde{N},
\end{equation*}
which we write as
\begin{align}
\tz^{s/u}&\left(\zeta, \theta_0,v_0,s_0,\varepsilon
\right)\nonumber\\&
=\left(
\phi_\U \left(\theta_0\alpha(v_0)+\zeta;0,v_0\right),0,y^{s/u},s_0+\zeta
\right).
\label{eq:zus}
\end{align}
Note that, if $\tz_0\in\Sigma\times\Sigma\times\T\subset\tS$ and therefore the unperturbed
manifolds are not differentiable at $\tz_0$, the uniqueness of these points is
also guaranteed because $\tilde{N}\subset \Sigma\times\Sigma\times\T$.\\
We measure the distance between these points
using the unperturbed Hamiltonian:
\begin{equation}
\Delta(\theta_0+\frac{\zeta}{\alpha(v_0)},v_0,s_0+\zeta,\varepsilon)=H_0\left( \tz^u
\right)-H_0\left( \tz^s
\right)=X(\tz^u)-X(\tz^s),
\label{eq:distzus}
\end{equation}
where $X(\tz^{s/u})$ is a shorthand for $X\left( \Pi_x
\left(\tz^{s/u}\right),\Pi_y\left(\tz^{s/u}\right)
\right)=X\left(0,y^{s/u}\right)$.  When this distance is zero
\begin{equation}
\Delta(\theta_0+\frac{\zeta}{\alpha(v_0)},v_0,s_0+\zeta,\varepsilon)=0,
\label{eq:distance_equation}
\end{equation}
and we solve this equation for $\zeta$.

To find an expression for $\Delta$ we proceed as usual in Melnikov methods.\\
From Proposition~\ref{prop:persistence_Lambda} we know that there exist points
$\tz^\pm \in \tL_\varepsilon^\pm$ satisfying 
\begin{equation*}
\vert \phi(t;\tz^{s};\varepsilon)-\phi(t;\tz^{+};\varepsilon)\vert
< K^{+}e^{-(\lambda^{+}+O(\varepsilon))\vert t\vert},\,t\to+\infty,
\end{equation*}
\begin{equation}
\vert \phi(t;\tz^{u};\varepsilon)-\phi(t;\tz^{+};\varepsilon)\vert
< K^{+}e^{-(\lambda^{+}+O(\varepsilon))\vert t\vert},\,t\to-\infty.
\label{eq:hyperbolicity_property_again}
\end{equation}
We add and subtract $X(\tz^+)$ and $X(\tz^-)$ to~(\ref{eq:distzus}) and write
$\Delta$ as
\begin{align}
\Delta&(\theta_0+\frac{\zeta}{\alpha(v_0)},v_0,s_0+\zeta,\varepsilon)
=\nonumber\\
&\Delta^-(\theta_0+\frac{\zeta}{\alpha(v_0)},v_0,s_0+\zeta,\varepsilon)
-\Delta^+(\theta_0+\frac{\zeta}{\alpha(v_0)},v_0,s_0+\zeta,\varepsilon)
\label{eq:distzus2b}
\end{align}
where
\begin{align}
\Delta^-(\theta_0+\frac{\zeta}{\alpha(v_0)},v_0,s_0+\zeta,\varepsilon)&=X^-(\tz^u)-X^-(\tz^-)+X^+(\tz^-)\label{eq:Deltamb}\\
\Delta^+(\theta_0+\frac{\zeta}{\alpha(v_0)},v_0,s_0+\zeta,\varepsilon)&=X^+(\tz^s)-X^+(\tz^+)+X^-(\tz^+)\label{eq:Deltapb}.
\end{align}
We now obtain expressions for $\Delta^\pm$. We illustrate the procedure for
$\Delta^-$.
Flowing backwards the points $\tz^u\in W^u(\tz^-)$ and $\tz^-\in
\tL_\varepsilon^-$ until the switching manifold $\tS$ is reached we obtain
points for which their backwards impact sequences are defined for all the
iterates. This is because, backwards in time, their trajectories never reach the
other switching manifold given by $x=0$. This provides a sequence of times for
which the flows $\tp(t;\tz^u;\varepsilon)$ and $\tp(t;\tz^-;\varepsilon)$ hit
the switching manifold $\tS$ for $t<0$. This permits us to apply the fundamental
theorem of calculus in these time intervals and write
\begin{align*}
\Delta^-(\theta_0,v_0,s_0+\zeta,\varepsilon)&\\
&=\underbrace{\varepsilon\int_t^0\left\{ X^-,h \right\}\left( \tp(r;\tz^u;\varepsilon)
\right)dr+X^-\left(\tp(t;\tz^u;\varepsilon)\right)}_{X^-(\tz^u)}\\
&\underbrace{-\varepsilon\int_t^0\left\{ X^-,h \right\}\left( \tp(r;\tz^-;\varepsilon)
\right)dr-X^-\left( \tp\left( t;\tz^-;\varepsilon \right)
\right)}_{-X^-(\tz^-)}\\
&+X^-(\tz^-).
\end{align*}
We now merge these two integrals and use the hyperbolicity
property~(\ref{eq:hyperbolicity_property_again}) to ensure convergence when $t\to-\infty$. 
The same property ensures that $\left|X^-\left( \tp\left(
t;\tz^u;\varepsilon \right) \right)-X^-\left( \tp\left( t;\tz^-;\varepsilon
\right) \right)\right|\to 0$ when $t\to -\infty$. Hence, we can write $\Delta^-$
as
\begin{align*}
\Delta^-(\theta_0,v_0,&s_0+\zeta,\varepsilon)\\
&=\varepsilon\int_{-\infty}^0\left( \left\{ X^-,h \right\}\left(
\tp(r;\tz^u;\varepsilon) \right)-\left\{ X^-,h \right\}\left( \tp\left(
r;\tz^-;\varepsilon
\right) \right) \right)dr\\
&+X\left( \tz^- \right).
\end{align*}
We now expand this in powers of $\varepsilon$. On the one hand, as $Q^-$ is a critical
point of the system associated with the Hamiltonian $X^-$, since
\begin{equation*}
\left(\Pi_x\left( \tz^- \right),\Pi_y\left( \tz^- \right)\right)=Q^-+O(\varepsilon),
\end{equation*}
we have that
\begin{equation*}
\left\{ X^-,h \right\}\left( \tp\left( t;\tz^-;\varepsilon \right)
\right)=O(\varepsilon).
\end{equation*}
On the other hand, and for the same reason, we have that 
\begin{equation*}
X^-(\tz^-)=X^-(Q^-)+O(\varepsilon^2).
\end{equation*}
Hence, using $\tz^u=\tz_0+O(\varepsilon)$ and
Proposition~\ref{prop:persistence_Lambda}, we can write $\Delta^-$ as
\begin{align*}
\Delta^-(\theta_0,v_0,&s_0+\zeta,\varepsilon)\\
&=\varepsilon\int_{-\infty}^0\left\{ X^-,h \right\}\left( \tp\left( r;\tz_0;0
\right) \right)dr\\
&+X^-(Q^-)+O(\varepsilon^2),
\end{align*}
where $\tz_0$ is given in~(\ref{eq:z0}).\\
Proceeding similarly with $\Delta^+$ and $X^+$ and using~\eqref{eq:z0} and that
$X^-(Q^-)=X^+(Q^+)=\bar{d}$, we finally get
\begin{align*}
\Delta(\theta_0&+\frac{\zeta}{\alpha(v_0)},v_0,s_0+\zeta,\varepsilon)=\\
&\varepsilon\int_{-\infty}^\infty\left\{ X,h \right\}\left(
\phi_{\U}\left(\theta_0\alpha(v_0)+\zeta+t;0,v_0\right),\su(t),s_0+\zeta +t\right)dt+O\left( \varepsilon^2
\right)\\
&:=\varepsilon M(\zeta,\theta_0,v_0,s_0)+O(\varepsilon^2).
\end{align*}

Each of these integrals is made up of a sum of integrals given by the impact
sequence associated with $u=0$ of the point $\tz_0$ and whose integrands are
smooth functions. Hence, the function
\begin{equation}
\zeta\longmapsto M(\zeta,\theta_0,v_0,s_0)
\label{eq:melnikov_function2}
\end{equation}
is a smooth function as regular as the flows $\tp^{\pm\pm}$ associated with
system~(\ref{eq:general_field_scattering}) restricted to their respective
domains. This is more apparent when performing the change of variables
$r=\theta_0\alpha(v_0)+\zeta+t$ leading to
\begin{align*}
M(&\zeta,\theta_0,v_0,s_0)\\
&=\int_{-\infty}^\infty\left\{ X,h \right\}\left(
\phi_{\U}\left(r;0,v_0\right),\su(r-\theta_0\alpha(v_0)-\zeta),s_0+r-\theta_0\alpha(v_0)\right)dr.
\end{align*}
Taking $(\theta_0,v_0,s_0)\in J$ given in
Proposition~\ref{prop:heteroclinic_intersection}, let $\bt(\theta_0,v_0,s_0)$ be
a simple zero of~(\ref{eq:melnikov_function2}).  Then, applying the implicit
function theorem to the equation
\begin{equation*}
\frac{\Delta(\theta_0+\frac{\zeta}{\alpha(v_0)},v_0,s_0+\zeta)}{\varepsilon}
=M\left( \zeta,\theta_0,v_0,s_0 \right)+O(\varepsilon)=0
\end{equation*}
at the point $(\zeta,\theta_0,v_0,s_0,\varepsilon)=(\bt,\theta_0,v_0,s_0,0)$, 
if $\varepsilon>0$ is small enough, there exists
\begin{equation}
\zeta^*(\theta_0,v_0,s_0;\varepsilon)=\bt+O(\varepsilon)
\label{eq:gamma_star}
\end{equation}
which solves~(\ref{eq:distance_equation}).
Thus, for every $(\theta_0,v_0,s_0)\in J$, there exists a locally unique point at the
section $\RR^2\times \Sigma\times\T$ belonging to the heteroclinic manifold
between the manifolds
$\tL_\varepsilon^\pm$,
\begin{equation}
\begin{aligned}
\tz_0^*(\theta_0+\frac{\zeta^*}{\alpha(v_0)},v_0,s_0+\zeta^*;\varepsilon)&
=\tz^u(\zeta^*,\theta_0,v_0,s_0;\varepsilon)\\
&=\tz^s(\zeta^*,\theta_0,v_0,s_0;\varepsilon)\\
&\in W^{u}( \tL_\varepsilon^- )\ti W^s(
\tL_{\varepsilon}^+)\cap\Sigma\times\RR^2\times\T,
\end{aligned}
\label{eq:z0star}
\end{equation}
which is of the form
\begin{equation*}
\tz_0^*(\theta_0,v_0,s_0;\varepsilon)=
(\phi_\U(\theta_0\alpha(v_0)+\zeta^*;0,v_0),0,y_h^*,s_0+\zeta^*).
\end{equation*}

Finally we consider the point
\begin{equation}
\tz^*(\theta_0,v_0,s_0;\varepsilon)=\tp(-\zeta^*;\tz^{*}_0;\varepsilon),
\label{eq:zstar}
\end{equation}
which belongs to $W^u(\Lambda_\varepsilon^-)\cap W^s(\Lambda_\varepsilon^+)$ but
is not in $\tS$.  Moreover, since $\zeta^*=\bt+O(\varepsilon)$, as given
in~(\ref{eq:gamma_star}), $\tz^*$ is of the form
\begin{equation*}
\tz^*(\theta_0,v_0,s_0;\varepsilon)=\left(\phi_\U(\theta_0\alpha(v_0);0,v_0),\sigma(-\bt)),s_0\right)+O(\varepsilon),
\end{equation*}
where $\left( \phi_\U(\theta_0\alpha(v);0,v),\sigma(\xi),s \right)$ is the
parameterization of the unperturbed heteroclinic connection introduced
in~(\ref{eq:unperturbed_unstable_manifold_po}). As $\tz^*$ depends on
$(\theta_0,v_0,s_0)\in N$, this permits us to consider two points
\begin{equation}
\tz^\pm(\theta_0,v_0,s_0;\varepsilon)=F_\varepsilon^\pm(\theta_0^\pm,v_0^\pm,s_0^\pm)
=F_0^\pm(\theta_0,v_0,s_0)+O(\varepsilon)
\in\tL_\varepsilon^\pm,
\label{eq:points_in_Lambda}
\end{equation}
such that
\begin{equation*}
\lim_{t\to\pm\infty}\tp(t;\tz^*;\varepsilon)-\tp(t;\tz^\pm;\varepsilon)=0,
\end{equation*}
where $F_\varepsilon^\pm$ are the parameterizations of $\tL_\varepsilon^\pm$
defined in~(\ref{eq:perturbed_Lambdas}) and $(\theta_0^\pm,v_0^\pm,s_0^\pm)\in N$,
with $N$ the reference manifold given in~(\ref{eq:reference_manifold_flow}).
\end{proof}

\subsection{First order properties of the scattering map}\label{sec:first_order_scattering}
Let $\bt=\bt(\theta,v,s)$ be a simple zero of the function~(\ref{eq:melnikov_function}) for any
$(\theta,v,s)\in J\subset N$. Then, for any $(\theta,v,s)\in J$  we can define the
scattering map
\begin{equation}
\begin{array}{cccc}
\Su_\varepsilon:&\mathcal{J}^-\subset \tL_\varepsilon^-&\longrightarrow &\mathcal{J}^+\subset \tL_\varepsilon^+\\
&\tz^-(\theta,v,s;\varepsilon)&\longmapsto&\tz^+(\theta,v,s;\varepsilon)
\end{array}
\label{eq:scattering_map}
\end{equation}
which identifies the points in~(\ref{eq:points_in_Lambda}) connected by the
orbit of the heteroclinic point $\tz^*(\theta,v,s;\varepsilon)\in W^u(\tz^-)\cap
W^s(\tz^+)$, which is of the form
\begin{equation*}
\tz^*(\theta,v,s;\varepsilon)=(\phi_\U(\theta\alpha(v);0,v),\su(-\bt),s)
+O(\varepsilon).
\end{equation*}

From equation~(\ref{eq:points_in_Lambda}), the points $\tz^\pm$ are of the form
\begin{align*}
\tz^\pm(\theta,v,s;\varepsilon)&=F_\varepsilon^\pm(\theta^\pm,v^\pm,s^\pm)\\
&=F_0^\pm(\theta,v,s)+O(\varepsilon)\\
&=(\phi_\U(\theta\alpha(v);0,v),Q^\pm,s)+O(\varepsilon).
\end{align*}

\begin{remark}\label{rem:intersection_other_manifolds}
All the computations for the scattering map $\Sd$ associated with the ``lower''
heteroclinic connection of system $\X$, that is, for the  heteroclinic manifold
close to $W^s(\tL_0^-)=W^u(\tL_0^+)$, are completely analogous.  In particular,
when $\varepsilon=0$, the compositions of both scattering maps become the
identity,
\begin{equation*}
\Su_0\circ\Sd_0=Id,
\end{equation*}
and therefore, there is no possibility of increasing the energy in this case.
\end{remark}

We now want to derive properties of the image of the scattering map given in
equation~(\ref{eq:scattering_map}).  More precisely, we want to measure the
difference of the energy levels of the points $\tz^\pm$. To this end, as is
usual in Melnikov-like theory, we use the Hamiltonian $U$ which, since
$U(0,v)=\frac{v^2}{2}$, is equivalent to measuring the distance in the
coordinate $v$. Hence we consider
\begin{equation}
\Delta U=U(\tz^+)-U(\tz^-),
\label{eq:dist_zpm1}
\end{equation}
where $U(\tz)$ is a shorthand for $U(\Pi_u(\tz),\Pi_v(\tz))$ and the points
$\tz^\pm$ are given in Proposition~\ref{prop:heteroclinic_intersection}.  Note
that this difference is $0$ for $\varepsilon=0$, and therefore $\Delta U
=O(\varepsilon)$.  The following Proposition provides an expression for the
first order term in $\varepsilon$ of $\Delta U$.
\begin{prop}\label{prop:first_order_scattering}
Let $(\theta,v,s)\in J\subset N$ given in
Proposition~\ref{prop:heteroclinic_intersection}, and let
$\bt=\bt(\theta,v,s)$ be a simple zero of the function
\begin{equation*}
\zeta\longrightarrow M(\zeta,\theta,v,s),
\end{equation*}
where $M$ is defined in~(\ref{eq:Melnikov_def}). Let also
$\tz^\pm(\theta,v,s;\varepsilon)\in\tL^\pm_\varepsilon$ be the points given
in~(\ref{eq:points_in_Lambda}), and hence satisfying
\begin{equation*}
\tz^+=\Su_\varepsilon(\tz^-).
\end{equation*}
Then, 
\begin{equation}
\begin{aligned}
U(\tz^+)-U(\tz^-)&=\varepsilon \int_{-\infty}^{0}\Big( \left\{ U,h \right\}\left(
(\phi_\U\left( \theta\alpha(v)+t;0,v \right),\su(t-\bt),s+t)
\right)\\
&-\left\{ U,h \right\}\left(  \phi_\U(\theta\alpha(v)+t;0,v),Q^-\!,s +t
\right) \Big)dt\\
&+\varepsilon \int^{+\infty}_{0}\Big( \left\{ U,h \right\}\left(
(\phi_\U\left( \theta\alpha(v)+t;0,v \right),\su(t-\bt),s+t)
\right)\\
&-\left\{ U,h \right\}\left(  \phi_\U(\theta\alpha(v)+t;0,v),Q^+\!,s +t
\right) \Big)dt\\
&+O(\varepsilon^{1+\rho_2}),
\end{aligned}
\label{eq:melnikov-like_U}
\end{equation}
for some $\rho_2>0$.\\
Moreover,
\begin{align}
<U(\tp(t&;\tz^+;\varepsilon))>-<U\left( \tp\left( t;\tz^-;\varepsilon \right)
\right)>\nonumber\\
&:=\lim_{T\to\infty}\frac{1}{T}\int_0^T\left( U\left( \tp(t;\tz^+;\varepsilon \right)
\right)-U\left( \tp\left( -t;\tz^-;\varepsilon \right) \right)dt\nonumber\\
&=\varepsilon\left(\lim_{T\to\infty}\frac{1}{T}\int_0^T\int_{-t}^t\left\{ U,h
\right\}\Big( \phi_\U\left(\theta\alpha(v)+r;0,v\right),\su\left( r-\bt
\right),s+r \Big)drdt\right)\label{eq:average_formula_heteroclinic}\\
&\quad+O(\varepsilon^{1+\rho_2})\nonumber.
\end{align}
\end{prop}
In order to prove this result, we will use the following Lemma, whose proof
is given after the proof of Proposition~\ref{prop:first_order_scattering}.
\begin{lem}
Let
\begin{align*}
\tz^\pm&=F^\pm_\varepsilon(\theta,v,s)\in\tL_\varepsilon^\pm\\
\tz^\pm_0&=F^\pm_0(\theta,v,s)\in\tL_0^\pm.
\end{align*}
Given $c>0$, there exists $\rho>0$ independent of $\varepsilon$ such that,
if $\varepsilon>0$ is small enough, then
\begin{equation*}
\left| \tp(t;\tz^\pm;\varepsilon)-\tp(t;\tz^\pm_0;0)\right|=O(\varepsilon^\rho)
\end{equation*}
for $0\le t\le c\ln\frac{1}{\varepsilon}$.
\label{lem:per-unpert_flows}
\end{lem}

\begin{proof}{\emph{Of Proposition~\ref{prop:first_order_scattering}}}\\
Let $(\theta,v,s)\in J$ and let $\bt$ be a simple zero of the 
Melnikov function~(\ref{eq:melnikov_function}), (\ref{eq:Melnikov_def}).  Let
also $\zeta^*(\theta,v,s,\varepsilon)$ be the solution
of~(\ref{eq:distance_equation}) given by the implicit function theorem near
$\bt$, and $\tz^*(\theta,v,s;\varepsilon)$ the heteroclinic point defined
in~(\ref{eq:zstar}).\\
Let us write $\Delta U$ as
\begin{equation}
\Delta U=\Delta U_++\Delta U_-,
\label{eq:dist_zpm2}
\end{equation}
where
\begin{align*}
\Delta U_+&=U(\tz^+)-U(\tz^*)\\
\Delta U_-&=U(\tz^*)-U(\tz^-).
\end{align*}
and  $\tz ^\pm$ are the points given in Proposition \ref{prop:persistence_Lambda}
satisfying \eqref{eq:perturbed_hyperbolicity_flow} and \eqref{eq:perturbed_hyperbolicity_flowu}.
We first derive an expression for the difference $\Delta U_+$;
an analogous one can be obtained for $\Delta U_-$.\\

Proceeding similarly as in the proof of Proposition~\ref{prop:heteroclinic_intersection} we
apply the fundamental theorem of calculus in the time intervals given by the
direct impact sequences of the points obtained by flowing $\tz^+$ and $\tz^*$,
forwards in time, until the switching manifold $\tS$ is reached. This provides
expressions for $U(\tz^+)$ and $U\left( \tz^* \right)$ which allow us to write
$\Delta U_+$ as
\begin{equation}
\begin{aligned}
\Delta U_+&=\varepsilon\int_0^t\left( \left\{ U,h \right\}\left(
\tp(r;\tz^*;\varepsilon)
\right)-\left\{ U,h \right\}\left( \tp\left( r;\tz^+;\varepsilon \right) \right)
\right)dr\\
&+U\left( \tp(t;\tz^+;\varepsilon) \right)-U\left( \tp\left( t;\tz^*;\varepsilon
\right) \right).
\end{aligned}
\label{eq:DeltaU1b}
\end{equation}
As $\tz^*\in W^s\left( \tz^+ \right)$ and $U$ is continuous,
formula~\eqref{eq:hyperbolicity_property_again} implies that the second line
in~(\ref{eq:DeltaU1b}) tends to zero as $t\to\infty$. As $\Delta U_+$ is
independent of $t$ even if the impacts sequences of $\tz^+$ and $\tz^*$ are
different, the integral in~(\ref{eq:DeltaU1b}) converges when $t\to\infty$. This
can also be seen by arguing as in the proof of
Proposition~\ref{prop:persistence_Lambda} and exponentially bounding the
integrand.  Hence, we obtain that
\begin{equation}
\Delta U_+=\varepsilon\int_0^\infty\left\{ U,h \right\}\left( \tp\left(
r;\tz^*;\varepsilon
\right) \right)-\left\{ U,h \right\}\left( \tp\left( r;\tz^+;\varepsilon \right)
\right)dr.
\label{eq:DeltaU2b}
\end{equation}
We now want to expand integral~(\ref{eq:DeltaU2b}) in powers of $\varepsilon$.
Unlike in the proof of Proposition~\ref{prop:heteroclinic_intersection}, when
using the Hamiltonian $U$ instead of $X$, the first order term in $\varepsilon$
of the Poisson bracket $\left\{ U,h \right\}$ restricted to the manifold
$\tL_\varepsilon^+$ does not vanish. For finite fixed times, the difference
between the perturbed and unperturbed flows restricted to $\tL_\varepsilon^+$ is
of order $O(\varepsilon)$. However, as the integral is performed from $0$ to
$\infty$, one has to proceed carefully.\\
Using Lemma~\ref{lem:per-unpert_flows}, we can split the integral
in~(\ref{eq:DeltaU2b}) to obtain
\begin{equation}
\begin{aligned}
\Delta U_+&=\varepsilon \int^{c\ln\frac{1}{\varepsilon}}_0
\left(\left\{ U,h \right\}
\left( \tp\left( r;\tz^*;\varepsilon\right)
\right)-
\left\{ U,h \right\}
\left( \tp\left( r;\tz^+;\varepsilon \right)
\right)\right)dr\\
&+\varepsilon \int_{c\ln(\frac{1}{\varepsilon})}^{\infty}
\left(\left\{ U,h \right\}
\left( \tp\left( r;\tz^*;\varepsilon \right)
\right)-
\left\{ U,h \right\}
\left( \tp\left( r;\tz^+;\varepsilon \right)
\right)\right)dr\\
&= \varepsilon \int^{c\ln\frac{1}{\varepsilon}}_0
\left(\left\{ U,h \right\}
\left( \tp\left( r;\tz^*_0;0\right)
\right)-
\left\{ U,h \right\}
\left( \tp\left( r;\tz^+_0;0 \right)
\right)\right)dr+O(\varepsilon^{\rho+1}\ln\frac{1}{\varepsilon})\\
&+\varepsilon \int_{c\ln\frac{1}{\varepsilon}}^{\infty}
\left(\left\{ U,h \right\}
\left( \tp\left( r;\tz^*;\varepsilon \right)
\right)-
\left\{ U,h \right\}
\left( \tp\left( r;\tz^+;\varepsilon \right)
\right)\right)dr,
\end{aligned}
\label{eq:integral5bb}
\end{equation}
where $\tz_0^+$ and $\tz_0^*$ are $\tz^+$ and $\tz^*$ for $\varepsilon=0$,
respectively and are given in \eqref{eq:zsz+0}.\\
We now consider the last integral in~(\ref{eq:integral5bb}).  As mentioned
above, arguing as in the proof of Proposition~\ref{prop:persistence_Lambda} with the
sequences $a_i$ and $b_i$ defined in~(\ref{eq:t_assumption}), we can
exponentially bound the integrand of the last integral in~(\ref{eq:integral5bb})
and write
\begin{align*}
&\int_{c\ln\frac{1}{\varepsilon}}^{\infty}
\left(\left\{ U,h \right\}
\left( \tp\left( r;\tz^*;\varepsilon \right)
\right)-
\left\{ U,h \right\}
\left( \tp\left( r;\tz^+;\varepsilon \right)
\right)\right)dr\\
&<\int_{c\ln\frac{1}{\varepsilon}}^{\infty}
K^+e^{-\lambda^+t}dt=\frac{K^+}{\lambda^+}\varepsilon^{\bar{\rho}},
\end{align*}
with $\bar{\rho}=c\lambda^+>0$. This allows us to write~(\ref{eq:DeltaU2b}) as
\begin{align*}
\Delta U_+&=\varepsilon \int^{\infty}_0
\left(\left\{ U,h \right\}
\left( \tp\left( r;\tz^*;\varepsilon \right)
\right)-
\left\{ U,h \right\}
\left( \tp\left( r;\tz^+;\varepsilon \right)
\right)\right)dr\\
&=\varepsilon \int^{c\ln\frac{1}{\varepsilon}}_0
\left(\left\{ U,h \right\}
\left( \tp\left( r;\tz^*_0;0\right)
\right)-
\left\{ U,h \right\}
\left( \tp\left( r;\tz^+_0;0 \right)
\right)\right)dr\\
&+O(\varepsilon^{\rho+1}\ln\frac{1}{\varepsilon})+O(\varepsilon^{\bar{\rho}+1}).
\end{align*}
By reverting the last argument, we can complete this integral from
$c\ln\frac{1}{\varepsilon}$ to $\infty$ to finally obtain
\begin{align*}
\Delta U_+&= \varepsilon\int^{\infty}_0 \Big(\left\{ U,h \right\}
\left( \tp\left( r;\tz_0^*;0 \right) \right)\\
&\qquad- \left\{ U,h \right\} \left(
\tp\left( r;\tz^+_0;0 \Big)
\right)\right)dr+O(\varepsilon^{1+\rho_1})\\
&=\varepsilon \int^{+\infty}_{0}\Big( \left\{ U,h \right\}\left(
(\phi_\U\left( \theta\alpha(v)+t;0,v \right),\su(t-\bt),s+t)
\right)\\
&-\left\{ U,h \right\}\left(  \phi_\U(\theta\alpha(v)+t;0,v),Q^+\!,s +t
\right) \Big)dt+O(\varepsilon^{1+\rho_1}).
\end{align*}
for some $\rho_1>0$.
Finally, proceeding similarly for $\Delta^-$ we obtain
expression~(\ref{eq:melnikov-like_U}) for some $\rho_2>0$.

We now obtain the formula given in~(\ref{eq:average_formula_heteroclinic}).
Proceeding similarly as before we apply the fundamental theorem of calculus and
obtain
\begin{align*}
<U(\tp(t&;\tz^+;\varepsilon))>-<U\left( \tp\left( t;\tz^-;\varepsilon \right)
\right)>\\
&=U\left( \tz^+ \right)-U\left( \tz^-
\right)\\
&\quad+\varepsilon\Bigg(\lim_{T\to\infty}\frac{1}{T}\int_0^T\Bigg(\int_0^t\left\{ U,h
\right\}\left( \tp\left( r;\tz^+;\varepsilon \right)  \right) dr\\
&\quad+\int_{-t}^0\left\{ U,h \right\}\left(\tp\left( r;\tz^-;\varepsilon
\right) 
\right)dr\Bigg)dt\Bigg).
\end{align*}
Using~\eqref{eq:DeltaU2b} and the equivalent one for $\Delta U_-$, we get
\begin{align*}
U(\tz^+)-U(\tz^-)&=\varepsilon\int_0^\infty\left\{ U,h \right\}\left( \tp\left(
r;\tz^*;\varepsilon
\right) \right)-\left\{ U,h \right\}\left( \tp\left( r;\tz^+;\varepsilon \right)
\right)dr\\
&+\varepsilon\int_{-\infty}^0\left\{ U,h \right\}\left( \tp\left(
r;\tz^*;\varepsilon
\right) \right)-\left\{ U,h \right\}\left( \tp\left( r;\tz^-;\varepsilon \right)
\right)dr.
\end{align*}

We now compute
\begin{align*}
<U(\tp(t&;\tz^+;\varepsilon))>-<U\left( \tp\left( t;\tz^-;\varepsilon \right)
\right)>-
\varepsilon\lim_{T\to\infty}\frac{1}{T}\int_0^T\int_{-t}^t\left\{ U,h
\right\}\left( \tp\left( r;\tz^*;\varepsilon \right) \right)drdt.
\end{align*}
Using formula for $U(\tz^+)-U(\tz^-)$ obtained before, this difference becomes
\begin{align*}
&\varepsilon\lim_{T\to\infty}\frac{1}{T}\int_0^T\int_t^\infty \left\{ U,h \right\}\left(
\tp\left( r;\tz^*;\varepsilon \right) \right)-\left\{ U,h \right\}\left(
\tp\left( r;\tz^+;\varepsilon \right)
\right)drdt\\
&-\varepsilon \lim_{T\to\infty}\frac{1}{T}\int_0^T\int_{-\infty}^{-t}\left\{
U,h \right\}\left( \tp\left( r;\tz^*;\varepsilon \right) \right)-\left\{ U,h
\right\}\left( \tp\left( r;\tz^-;\varepsilon \right)
\right)drdt=0,
\end{align*}
due to the asymptotic properties~\eqref{eq:hyperbolicity_property_again}.
This gives,
\begin{align*}
<U(\tp(t&;\tz^+;\varepsilon))>-<U\left( \tp\left( t;\tz^-;\varepsilon \right)
\right)>\\
&=\varepsilon\lim_{T\to\infty}\int_0^T\int_{-t}^t\left\{ U,h \right\}\left(
\tp\left( r;\tz^*;\varepsilon \right)
\right)drdt,
\end{align*}
which expanding in powers of $\varepsilon$ gives the desired formula
in~\eqref{eq:average_formula_heteroclinic}.
\end{proof}

\begin{proof}{\emph{Of Lemma~\ref{lem:per-unpert_flows}}}\\
We proceed with the points $\tz^+$ and $\tz^+_0$; analogous arguments hold for
$\tz^-$ and $\tz^-_0$.
Let us first flow the points $\tz^+$ and $\tz^+_0$ backwards in time until their
trajectories reach
the section $\tS^+$. This provides two points, $(0,\tw^+)$ and $(0,\tw^+_0)$,
respectively, such that
\begin{align*}
\tw^+&=\left( \omega^+,s^+ \right)=G_\varepsilon^+(v,s)\in\tG_\varepsilon^+\\
\tw^+_0&=\left( \omega^+_0,s^+_0 \right)=G_0^+(v,s)\in\tG_0^+.
\end{align*}
We now proceed by considering the trajectories of these last points.
Let 
\begin{align*}
\tz^{+}_{n,\varepsilon}&=(0,\tw^{+}_{n,\varepsilon})=\impactwzz{+}{n,\varepsilon}\\
\tz^{+}_{n,0}&=(0,\tw^{+}_{n,0})=\impactwzz{+}{n,0}
\end{align*}
be the impact sequences associated with $\tw^+$ and $\tw^+_0$, respectively.  We
first write
\begin{align*}
\Delta(t)&:=\left| \tp(t;(0,\tw^+);\varepsilon)-\tp(t;(0,\tw^+_0);0)\right|\\
&=\left|\tp(t-s_{n,\varepsilon}^{+}+s^+;\tz^{+}_{n,\varepsilon};\varepsilon)-
\tp(t-s_{n,0}^{+}+s^+_0;\tz^{+}_{n,0};0)\right|.
\end{align*}
Proceeding as in the proof of Proposition~\ref{prop:first_order_scattering}, we
define
\begin{align*}
a_n&=\min\left( s_{n,\varepsilon}^{+}-s^+,s^{+}_{n,0}-s^+_0 \right)\\
b_n&=\max\left(  s_{n,\varepsilon}^{+}-s^+,s^{+}_{n,0}-s^+_0\right).
\end{align*}
such that $t\in[a_n,b_{n+1}]$.\\
Let $\F_\varepsilon$ be the piecewise-smooth vector field associated with the
perturbed system~(\ref{eq:general_field_scattering}) and $\F_0$ the one for
$\varepsilon=0$; applying the fundamental theorem of calculus we get
\begin{align*}
\Delta(t)&\le \left| \tz^{+}_{n,\varepsilon}-\tz^{+}_{n,0} \right|\\
&+\int_{a_n}^{b_{n}}\left|\F_\varepsilon\left(\tp\left(t-s^{+}_{n,\varepsilon};
\tz^{+}_{n,\varepsilon};\varepsilon\right)\right)
-\F_0\left(\tp\left(t-s^{+}_{n,0}+s_0^+;\tz^{+}_{n,0};0\right)\right)  \right|dt\\
&+\int_{b_n}^{a_{n+1}}\left|\F_\varepsilon\left(\tp\left(t-s^{+}_{n,\varepsilon};
\tz^{+}_{n,\varepsilon};\varepsilon\right)\right)
-\F_0\left(\tp\left(t-s^{+}_{n,0}+s_0^+;\tz^{+}_{n,0};0\right)\right)  \right|dt\\
&+\int^{b_{n+1}}_{a_{n+1}}\left|\F_\varepsilon
\left(\tp\left(t-s^{+}_{n,\varepsilon};
\tz^{+}_{n,\varepsilon};\varepsilon\right)\right)
-\F_0\left(\tp\left(t-s^{+}_{n,0}+s_0^+;\tz^{+}_{n,0};0\right)\right)  \right|dt,
\end{align*}
For the first and third integrals, as both flows do not belong to the same domain
$S^\pm\times S^\pm\times\T$, it is not the case that $\F_\varepsilon\to
\F_0$ as $\varepsilon\to0$. However, their difference is bounded by some
constant $K_1>0$.\\
For the middle integral, both flows belong to the same domain and 
 $\F_\varepsilon$ and $\F_0$ are $\varepsilon$-close. Hence, there exists
a constant $K>0$ such that
\begin{align*}
\Delta(t)&\le \left| \tz_{n,\varepsilon}^+-\tz_{n,0}^+ \right|\\
&+K_1 (b_n-a_n)\\
&+\int_{b_n}^{a_{n+1}}K \left| \tp\left( t;\tz_{n,\varepsilon}^+;\varepsilon \right)
-\tp\left( t;\tz_{n,0}^+;0 \right) \right|dt\\
&+K_1(a_{n+1}-b_{n+1}).
\end{align*}
Let us bound $\left| \tz^{+}_{n,\varepsilon}-\tz^{+}_{n,0} \right|$.  We first
write
\begin{align*}
\left| \tz^{+}_{n,\varepsilon}-\tz^{+}_{n,0} \right|&=\left|
\tw^+_{n,\varepsilon}-\tw^{+}_{n,0} \right|\\
&=\left| P^\pm_\varepsilon(\tw^{+}_{n-1,\varepsilon})-P_0^\pm(\tw^{+}_{n-1,0}) \right|\\
&=\Big|
P^\pm_\varepsilon(\tw^{+}_{n-1,\varepsilon})-P_0^\pm(\tw^{+}_{n-1,\varepsilon})\\
&+P_0^\pm(\tw^{+}_{n-1,\varepsilon})
-P_0^\pm(\tw^{+}_{n-1,0}) \Big|,
\end{align*}
where we apply $P^+_\varepsilon$ or $P^-_\varepsilon$ and $P^+_0$ or $P^-_0$
depending on the sign of $\Pi_v(\tw^{+}_{n-1,\varepsilon})$ and
$\Pi_v(\tw^{+}_{n-1,0})$,
respectively.\\
Since $P_\varepsilon^\pm$ and $P^\pm_0$ are $\varepsilon$-close and 
$P^\pm_0$ are Lipschitz maps, there exist positive constants $c$,
$K_{P_0}$ and $c_0$ such that, for $n=1$,
\begin{align*}
\left| \tz_{1,\varepsilon}^+-\tz_{1,0}^+ \right|=\Big|
&P^\pm_\varepsilon(\tw^{+}_{0,\varepsilon})-P_0^\pm(\tw^{+}_{0,\varepsilon})+P_0^\pm(\tw^{+}_{0,\varepsilon})
-P_0^\pm(\tw^{+}_{0,0}) \Big|\\
&\le c\varepsilon+K_{P_0}\left| \tw^{+}_{0,\varepsilon}-\tw^{+}_{0,0} \right|=c\varepsilon+K_{P_0}c_0\varepsilon.
\end{align*}
By induction and assuming the general case $K_{P_0}>1$, we obtain
\begin{align*}
\left| \tz^{+}_{n,\varepsilon}-\tz^{+}_{n,0} \right|&=c\varepsilon
+K_{P_0}\left|\tw^{+}_{n-1,\varepsilon}-\tw^{+}_{n-1,0} \right|\\
&\le c\varepsilon+K_{P_0}\left( c\varepsilon+K_{P_0}\left|
\tw^{+}_{n-2,\varepsilon}-\tw^{+}_{n-2,0} \right| \right)\\
&\le c\varepsilon \sum_{i=0}^{n-1}\left( K_{P_0} \right)^i+\left( K_{P_0}
\right)^nc_0\varepsilon\\
&=c\varepsilon \frac{1-\left( K_{P_0} \right)^{n-1}}{1-K_{P_0}}+\left(
K_{P_0}
\right)^nc_0\varepsilon\\
&\le M\left( K_{P_0} \right)^n\varepsilon,
\end{align*}
for some $M>0$.\\
Arguing similarly for the differences $b_n-a_n$ and $b_{n+1}-a_{n+1}$ we get
that
\begin{equation*}
\Delta(t)\le M\left( K_{P_0} \right)^n\varepsilon
+\int_{b_n}^{a_{n+1}}K \Delta(r) dr,
\end{equation*}
with $M$ properly redefined.\\
We now apply the Gronwall's inequality to this expression. Noting that
\begin{equation*}
a_{n+1}-b_n=K_2+nK_3\varepsilon,
\end{equation*}
with $K_2=\max(\alpha^+(v),\alpha^-(-v))$ and $K_3>0$, this finally gives us
\begin{align*}
\Delta(t)& \le  M\left( K_{P_0} \right)^n\varepsilon
e^{K_2+nK_3\varepsilon}\\
&=M\varepsilon e^{K_2+n(K_3\varepsilon+\ln K_{P_0})}\\
&<M\varepsilon e^{K_2+n\ln K_4},
\end{align*}
for some positive $K_4$.\\
Taking $n=c_2\ln\frac{1}{\varepsilon}$ and making $c_2$ small enough such that
$c_2\ln K_4<1$, which is independent from $\varepsilon$, we finally have
\begin{align*}
\left| \tp(t;(0,\tw^+);\varepsilon)-\tp(t;(0,\tw^+_0);0)\right|
&\le M e^{K_2}
\varepsilon \left( \frac{1}{\varepsilon} \right)^{c_2\ln K_4}\\
&\le Me^{K_2} \varepsilon^\rho,
\end{align*}
for some $\rho>0$, which is what we wanted to prove.
\end{proof}

\section{Example: two linked rocking blocks}\label{sec:example}
In this section we apply some of our results to a mechanical example consisting
of two rocking blocks coupled by means of a spring (see
figure~\ref{fig:two_blocks}). The single block model was first introduced
in~\cite{Hou63}; further details of its dynamics can be found
in~\cite{YimChoPen80,SpaKoh84,Hog89,GraHogSea12}.

\begin{figure}
\begin{center}
\includegraphics[width=1\textwidth]
{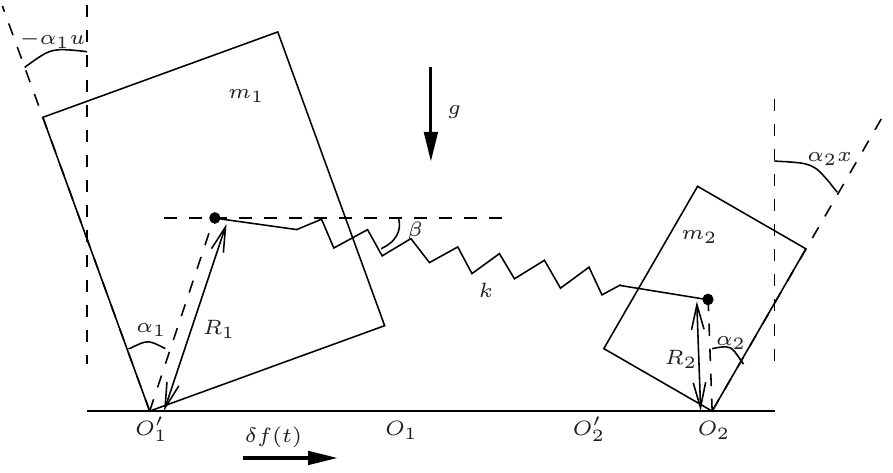}
\end{center}
\caption{Two rocking blocks linked by a spring.}
\label{fig:two_blocks}
\end{figure}
Both blocks are rigid, of mass $m_i$ and with semi--diagonal of length $R_i$. They are connected by a light spring, with spring constant $k$. The base is assumed to be sufficiently flat so that block $i$ rotates only about points $O_i, O^{'}_i$. On impact with the rigid base, neither block loses energy. Let $\alpha_i$ be the angle formed by the lateral side and the diagonal of each
block. We then take as state variables $u$ and $x$ such that $\alpha_1 u$ and
$\alpha_2 x$ are the angles formed by the vertical and the lateral side of
each block. When there is rotation, $u$ is positive (negative) for rotation about $O_1$ ($O^{'}_1$) and $x$ is positive (negative) for rotation about $O_2$ ($O^{'}_2$). The spring makes an angle $\beta$ with the horizontal. As shown in~\cite{Hog89}, when both blocks are slender ($\alpha_i<<1$), the dynamics of each is modeled by the piecewise Hamiltonian
systems
\begin{equation*}
U(u,v)=\left\{
\begin{aligned}
\frac{v^2}{2}-\frac{u^2}{2}+u,&\text{ if }u\ge0\\
\frac{v^2}{2}-\frac{u^2}{2}-u,&\text{ if }u<0
\end{aligned}\right.
\end{equation*}
and
\begin{equation*}
X(x,y)=\left\{
\begin{aligned}
\frac{y^2}{2}-\frac{x^2}{2}+x,&\text{ if }x\ge0\\
\frac{y^2}{2}-\frac{x^2}{2}-x,&\text{ if }x<0.
\end{aligned}\right.
\end{equation*}
Each system has two critical points at $Q^\pm=(\pm1,0)$, and there are two
heteroclinic connections $\gamma^\text{up/down}$ between them, given by the
energy level $U(u,v)=\frac{1}{2}$ and $X(x,y)=\frac{1}{2}$,
\begin{equation*}
\gamma^\text{up/down}=\left\{ (x,y)=\sigma^{\text{up/down}}(\xi),\xi\in\RR \right\},
\end{equation*}
where
\begin{equation*}
\sigma^{\text{up/down}}(\xi)=\left\{
\begin{aligned}
&(\pmm(1-e^{-\xi}),\pmm e^{-\xi})&&\text{if }\xi\ge0\\
&(\pmm(e^{\xi}-1),\pmm e^{\xi})&&\text{if }\xi<0.
\end{aligned}
\right\}
\end{equation*}
These heteroclinic connections surround a region filled with a continuum of
period orbits, which are given by $U(u,v)=c$ and $X(x,y)=c$, with
$0<c<\frac{1}{2}$, and
\begin{equation*}
\phi_{\U}(\tau;0,v)=\left\{
\begin{aligned}
&\bigg(\frac{v-1}{2}e^\tau-\frac{v+1}{2}e^{-\tau}+1,\\
&\quad\frac{v-1}{2}e^\tau+\frac{v+1}{2}e^{-\tau}\bigg)&&
\text{if }0\le\tau\le\alpha^+(v)\\
&\bigg(-\frac{v-1}{2}e^{\tau-\alpha^+(v)}+\frac{v+1}{2}e^{-\tau+\alpha^+(v)}-1,\\
&\quad-\frac{v-1}{2}e^{\tau-\alpha^+(v)}-
\frac{v+1}{2}e^{-\tau+\alpha^+(v)}\bigg)&&
\text{if }\tau\alpha^+(v)\le\tau\le\alpha(v),
\end{aligned}
\right.
\end{equation*}
with $\alpha^+(v)$ and $\alpha(v)$ given by
\begin{align*}
\alpha^+(v)&=2\int_0^{1-\sqrt{1-v^2}}\frac{1}{\sqrt{v^2+u^2-2u}}du=\\
&=2\ln\left(\frac{1+u}{1-u}\right)\\
\alpha(v)&=2\alpha^+(v),
\end{align*}
(similarly for the Hamiltonian $X$). Hence conditions C.1--C.4 of
\S\ref{sec:system_description_scattering} are satisfied.

We now assume that both blocks are identical ($\alpha_1=\alpha_2$, $R_1=R_2, m_1=m_2$).
This allows us to assume that the angle formed by the spring and the horizontal
is small, and hence to linearize the coupling around $\beta=0$.  When the blocks are subject to an external  small $T$-periodic forcing given by $\delta f(t)$, the
(linearized) equations that govern the system in the extended phase space are
\begin{equation}
\begin{aligned}
\dot{u}=&v\\
\dot{v}=&u-\sgn(u)\\
&+k(x-u)-\delta f(s)\\
\dot{x}=&y\\
\dot{y}=&x-\sgn(x)\\
&+k(u-x)-\delta f(s)\\
\dot{s}=&1,
\end{aligned}
\label{eq:linked_blocks}
\end{equation}
Introducing the perturbation parameter $\varepsilon$ through the
reparameterization
\begin{equation*}
\delta=\te \varepsilon,\quad k=\tk\varepsilon,
\end{equation*}
with $\te$ and $\tk$ both positive constants, and taking $f(s)=\cos(\omega s)$
(\cite{Hog89}), these equations can be written in
terms of a piecewise-smooth Hamiltonian of the form
\begin{equation}
H_{\varepsilon}(u,v,x,y,s)=U(u,v)+X(x,y)+\varepsilon h(u,x,s)
\label{eq:hamiltonian_perturbed_linked_blocks2}
\end{equation}
where $h$ is the Hamiltonian perturbation
\begin{equation}
h(u,x,s)=
\te( u+x) \cos(\omega s)+\tk\big(\frac{u^2}{2}+ \frac{x^2}{2}-ux\big).
\label{eq:hamiltonian_perturbation2}
\end{equation}
The objects given by the critical points and heteroclinic connections of the
Hamiltonian $X$, on one hand, and the periodic orbits of the Hamiltonian $U$, on
the  other one, give rise to the manifolds
\begin{equation*}
\tL^\pm=\left\{ (\phi_\U(\theta\alpha(v);0,v),\pm1,0,s)\in
\RR^4\times\T,\,\sqrt{2c_1}\le v\le \sqrt{2c_2},\,0\le \theta\le 1
\right\},
\end{equation*}
$0<c_1,c_2<\frac{1}{2}$, that are invariant for the coupled system when
$\varepsilon=0$ and have $4$-dimensional heteroclinic manifolds $\tg=
W^s(\tL^+)=W^u(\tL^-)$ and
$\tilde{\gamma}^{\text{down}}=W^u(\tL^+)=W^s(\tL^-)$.

As stated in Proposition~\ref{prop:persistence_Lambda}, the invariant manifolds
$\tL^\pm$ persist when $\varepsilon>0$ is small enough.
Moreover, as shown in Proposition~\ref{prop:heteroclinic_intersection},
the Melnikov function~\eqref{eq:Melnikov_def} provides the first order term in
$\varepsilon$ of the distance of between the unstable and stable manifolds of
$\tL_\varepsilon^-$ and $\tL_\varepsilon^+$, respectively. For
system~\eqref{eq:hamiltonian_perturbed_linked_blocks2} this
becomes
\begin{equation}
M(\zeta,\theta,v,s):=\int_{-\infty}^\infty\left(
-y(t)(\te\cos(\omega (s+\zeta))+\tk(x(t)-u(t)))
\right)dt,
\label{eq:melnikov_scattering_example}
\end{equation}
where $(x(t),y(t))=\su(t)$ and
$u(t)=\Pi_u(\phi_\U(\theta\alpha(v)+t+\zeta;0,v))$.\\
More precisely, $M(\zeta,\theta,v,s)$ computes the first order distance between
the points $\tz^u$ and $\tz^s$, given, respectively, by intersection between
$W^u(\tL_\varepsilon^-)$ and $W^s(\tL_\varepsilon^+)$ with the line
\begin{equation*}
\tilde{N}=\left\{\tz_0+l
(0,0,0,1,0),\,l\in\RR  \right\}\subset
\RR^2\times\Sigma\times\T,
\end{equation*}
where $\tz_0$ belongs to the intersection of $W^u(\tL_0^-)=W^s(\tL_0^+)$ with
$\left\{ x=0 \right\}$ for $\varepsilon=0$, and is parametrized by
\begin{equation*}
\tz_0=\tz_0( \zeta,\theta,v,s):=\left(
\phi_{\U}\left( \theta\alpha(v)+\zeta;0,v\right),0,1,s+\zeta\right).
\end{equation*}
\begin{figure}
\begin{center}
\includegraphics[width=0.4\textwidth,angle=-90]{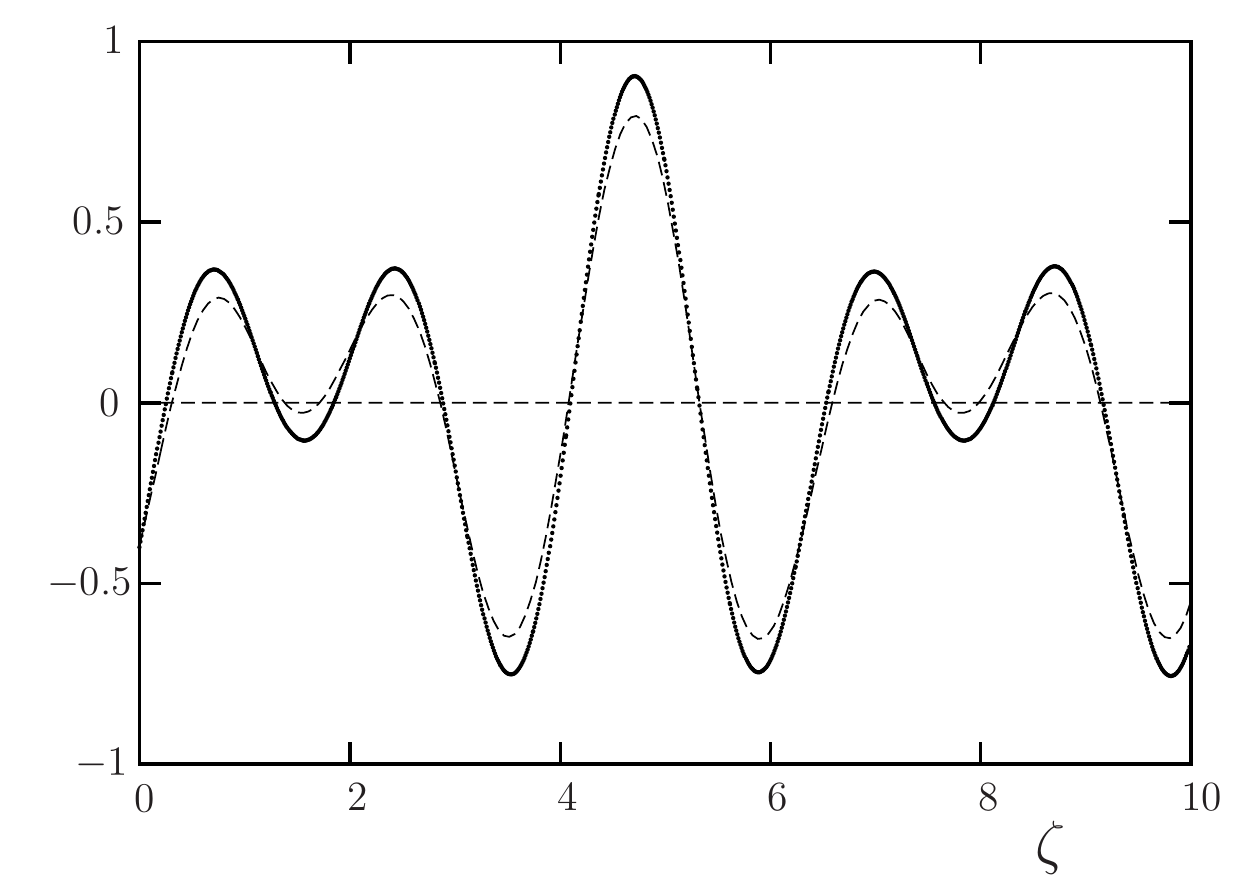}
\end{center}
\caption{Melnikov function and real distance (dashed) for $v=0.48$, $\theta=s=0$
and $\varepsilon=0.01$. The real distance has been magnified by a factor of
$\frac{1}{\varepsilon}$ to be compared with the Melnikov function.}
\label{fig:melnikov_v0d48_tau0}
\end{figure}
In figure~\ref{fig:melnikov_v0d48_tau0} we provide both the Melnikov function and
the real distance between $\tz^u$ and $\tz^s$ for $\varepsilon=0.01$, $v=0.48$,
$\theta=s=0$ when varying $\zeta$. This real distance is computed as follows.
Having fixed $(\theta,v,s)$ and $\varepsilon$,  for every $\zeta$ we numerically
find the $y$ coordinates ($y^u$ and $y^s$) of $\tz^u$ and $\tz^s$ and subtract
them. To compute $y^s$, we take an $\varepsilon$-neighbourhood of $y=1$ (where
the unperturbed manifold intersects $x=0$) which we assume contains $y^s$, and use a Bolzano-like method. 
We consider a set of initial conditions
\begin{equation*}
\left( \phi_\U(\theta\alpha(v)+\zeta;0,v),0,y_i,s+\zeta \right),
\end{equation*}
with $y_i\in\left( 1-O(\varepsilon),1+O(\varepsilon) \right)$, and integrate the
flow forwards in time for each of them. As the stable manifold
$W^s(\tL_\varepsilon^*)$ is $4$-dimensional, it separates the space into two
pieces and hence, if $\varepsilon>0$ is small enough,
trajectories either escape to infinity and or return to the section $x=0$.
This gives us $y_i$ and $y_{i+1}$ where $y_i$ is the largest value such that the
trajectory returns to $x=0$ and $y_{i+1}$ is the smallest value such that its
trajectory escapes to infinity. Hence $y^s\in(y_i,y_{i+1})$ and we proceed again
with this smaller interval. This is repeated until some desired tolerance is achieved.

The integration of the flow was done using an adaptative high order Runge Kutta
method (RKF78) with multiple precision libraries. The number of initial
conditions taken along the current interval at each iteration was $50$, and
their trajectories were launched in parallel. This allowed us to compute $y^s$
with a tolerance of $10^{-27}$ (length of the last interval) within a reasonable
time. We proceeded similarly for $y^u$, integrating backwards in time, also in
parallel. As can be seen in figure~\ref{fig:melnikov_v0d48_tau0}, the real
distance agrees very well with the value given by $M(\zeta,\theta,v,s)$
multiplied by $\varepsilon$.

Note that both the integration of the system and the computation of the
Melnikov function have been done numerically. We have used neither
the linearity nor the symmetry of the system, apart from the explicit
expressions for $\alpha(v)$, $\phi_\U$ and $\su(t)$, which could easily have been
computed numerically. Thus, the same techniques could easily be applied
to the full nonlinear equations.

As shown in Proposition~\ref{prop:heteroclinic_intersection}, of special
interest are the zeros of the Melnikov function, which lead to zeros of the real
distance between $\tz^u$ and $\tz^s$ and, hence, to heteroclinic connections. In
other words, for each simple zero $\bt$ there exists
$\zeta^*=\bt+O(\varepsilon)$ such that $\tz^s=\tz^u:=\tz^*$ and points $\tz^\pm$
satisfying\footnote{For convenience, we have slightly changed the notation with
respect to section~\ref{sec:scattering_map}. Points $\tz^*$ and $\tz^\pm$ here
correspond to the ones in Proposition~\ref{prop:heteroclinic_intersection}
flowed a time $\zeta^*$ by $\tp$.}
\begin{equation*}
\lim_{t\to\pm\infty}\left\vert\tp(t;\tz^*;\varepsilon)-\tp(t;\tz^\pm;\varepsilon)\right\vert=0.
\end{equation*}
These are of the form
\begin{align*}
\tz^\pm&=\left( \phi_\U(\theta \alpha(v)+\bt;0,v),\pm 1,0,s+\zeta
\right)+O(\varepsilon)\\
\tz^*&=\left( \phi_\U(\theta\alpha(v)+\bt;0,v),0,1,s+\zeta \right)+O(\varepsilon).
\end{align*}

The points $\tz^\pm$ may be located at different energy levels on the manifolds
$\tL_\varepsilon^\pm$. Their first order difference is provided in terms of the
unperturbed flows by (\ref{eq:melnikov-like_U}) in
Proposition~\ref{prop:first_order_scattering}. In addition, (\ref{eq:average_formula_heteroclinic}) of
Proposition~\ref{prop:first_order_scattering} provides an expression for
the first order difference between the average energy of the trajectories
$\tp(\pm t;\tz^\pm;\varepsilon)$ for $t\to\infty$.

If we compute expression~\eqref{eq:average_formula_heteroclinic} for the third and
fourth positive (in $\zeta$) zeros of the Melnikov function we obtain
\begin{equation}
<U(\tp(t;\tz^+;\varepsilon))>-<U\left( \tp\left( t;\tz^-;\varepsilon \right)
\right)>\simeq 0.4
\label{eq:average_third_zero}
\end{equation}
for the third zero, and
\begin{equation}
<U(\tp(t;\tz^+;\varepsilon))>-<U\left( \tp\left( t;\tz^-;\varepsilon \right)
\right)>\simeq -0.3
\label{eq:average_fourth_zero}
\end{equation}
for the fourth one. Note that a positive difference implies an increase of the energy
of the system while a negative one a decrease. Note the high dependence of
this difference on the choice of the zero.

We now compute numerically the third and fourth zeros of the real distance in
order to compute their associated heteroclinic connections and illustrate this
behaviour. This is done by using a Bolzano method starting in a
$\varepsilon$-neighbourhood of each zero. For each value of $\zeta$, we
calculate $\tz^u$ and $\tz^s$ as explained before and calculate their difference.
From the third step of the Bolzano method we use the previous computations to
obtain a prediction for the next interval in $y$ where to look for $y^s$
(similarly for $y^u$), which improves the method significantly. This is done
until the real zero is computed with a precision of $10^{-26}$. We find
\begin{align*}
\tz^*&=(
-0.11379311572593961969337806,\\
&0.12554935975439240524029269,\\
&0,1.11150143902429741752435119,\\
&1.71158269885731891700238123)
\end{align*}
for the third zero of the Melnikov function and
\begin{align*}
\tz^*&=(0.09636673455802005569868835,\\
&-0.21668659029422144991945461,\\
&0, 1.12033664434168488471504850,\\
&2.85947780778602337824850186
)
\end{align*}
for the fourth one. Their trajectories are shown in
figures~\ref{fig:trajectory_third_zero} and~\ref{fig:trajectory_fourth_zero}.
The initial condition $\tz^*$ belongs to the section $x=0$ and is used to
integrate the flow forwards and backwards. Note that, due to numerical errors,
the trajectory escapes after spiraling around the manifolds $\tL_\varepsilon^*$
and $\tL_\varepsilon^-$.
\begin{figure}
\begin{center}
\begin{picture}(1, 0.4)
\put(0,0.35){
\subfigure[]
{\includegraphics[angle=-90,width=.5\textwidth]
{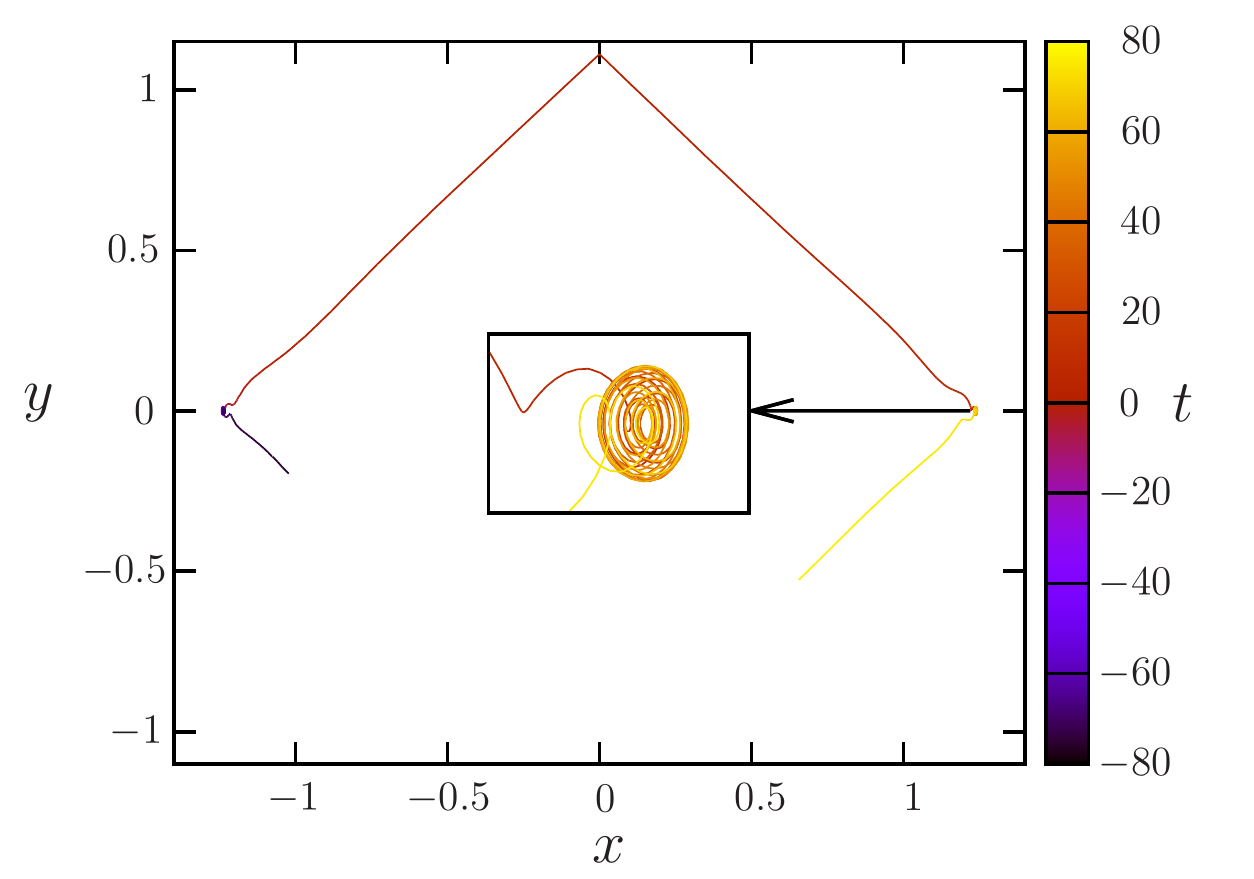}
}}
\put(0.5,0.35){
\subfigure[]
{\includegraphics[angle=-90,width=.5\textwidth]
	{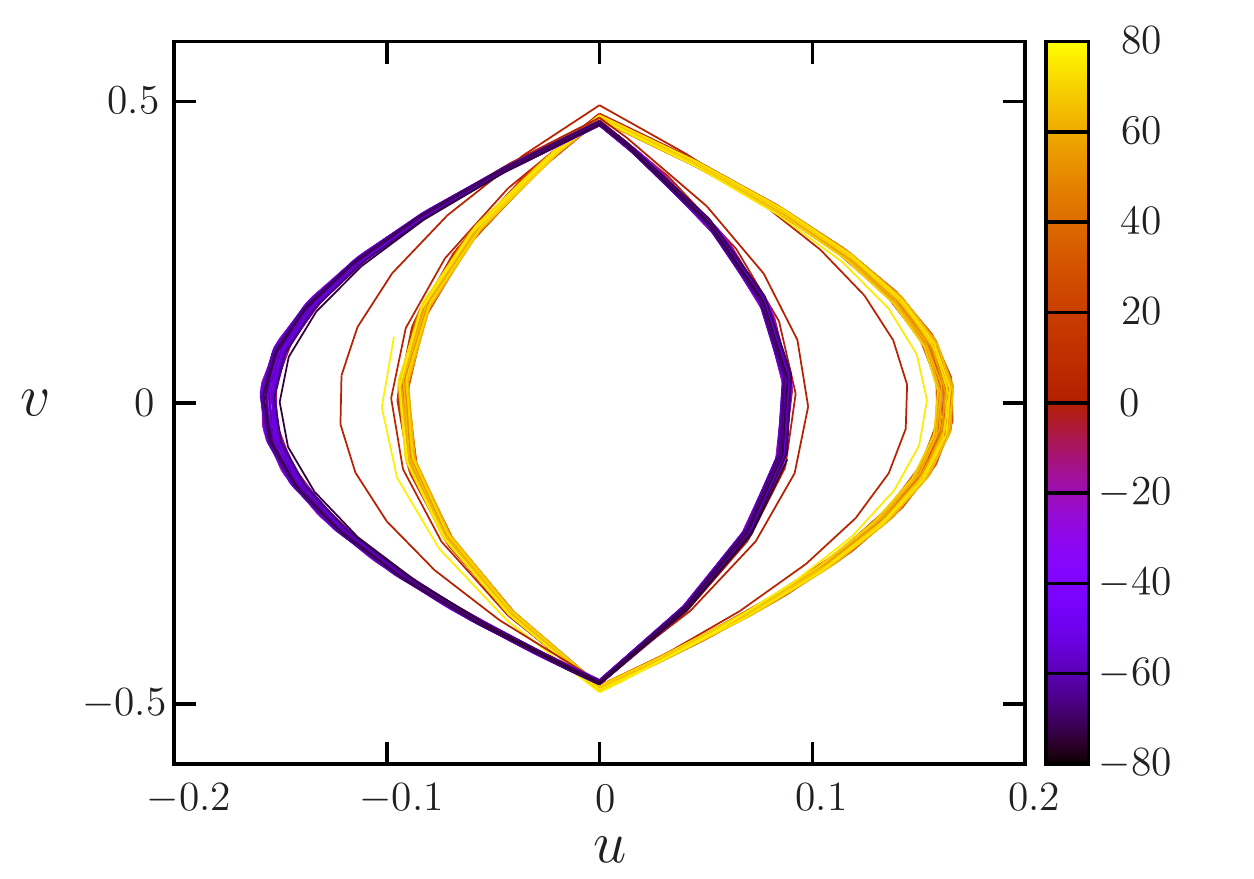}
}}
\end{picture}
\end{center}
\caption{Trajectory close the heteroclinic point $\tz^*$ obtained for the third
zero of the Melnikov function in figure~\ref{fig:melnikov_v0d48_tau0}.
Projections onto (a) the $(u,v)$ plane and (b) the $(x,y)$ plane. The colour bar denotes time.}
\label{fig:trajectory_third_zero}
\end{figure}

\begin{figure}
\begin{center}
\begin{picture}(1, 0.4)
\put(0,0.35){
\subfigure[]
{\includegraphics[angle=-90,width=.5\textwidth]
{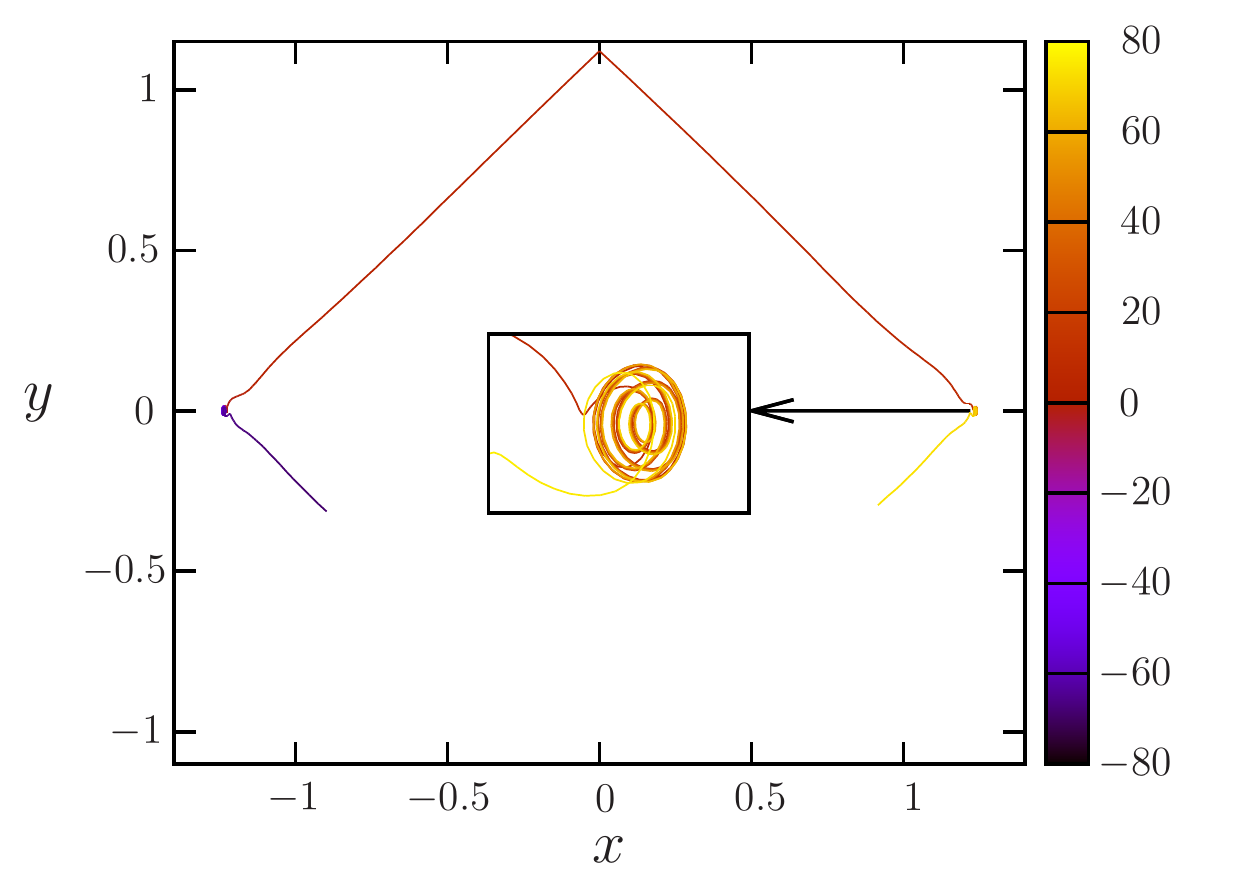}
}}
\put(0.5,0.35){
\subfigure[]
{\includegraphics[angle=-90,width=.5\textwidth]
{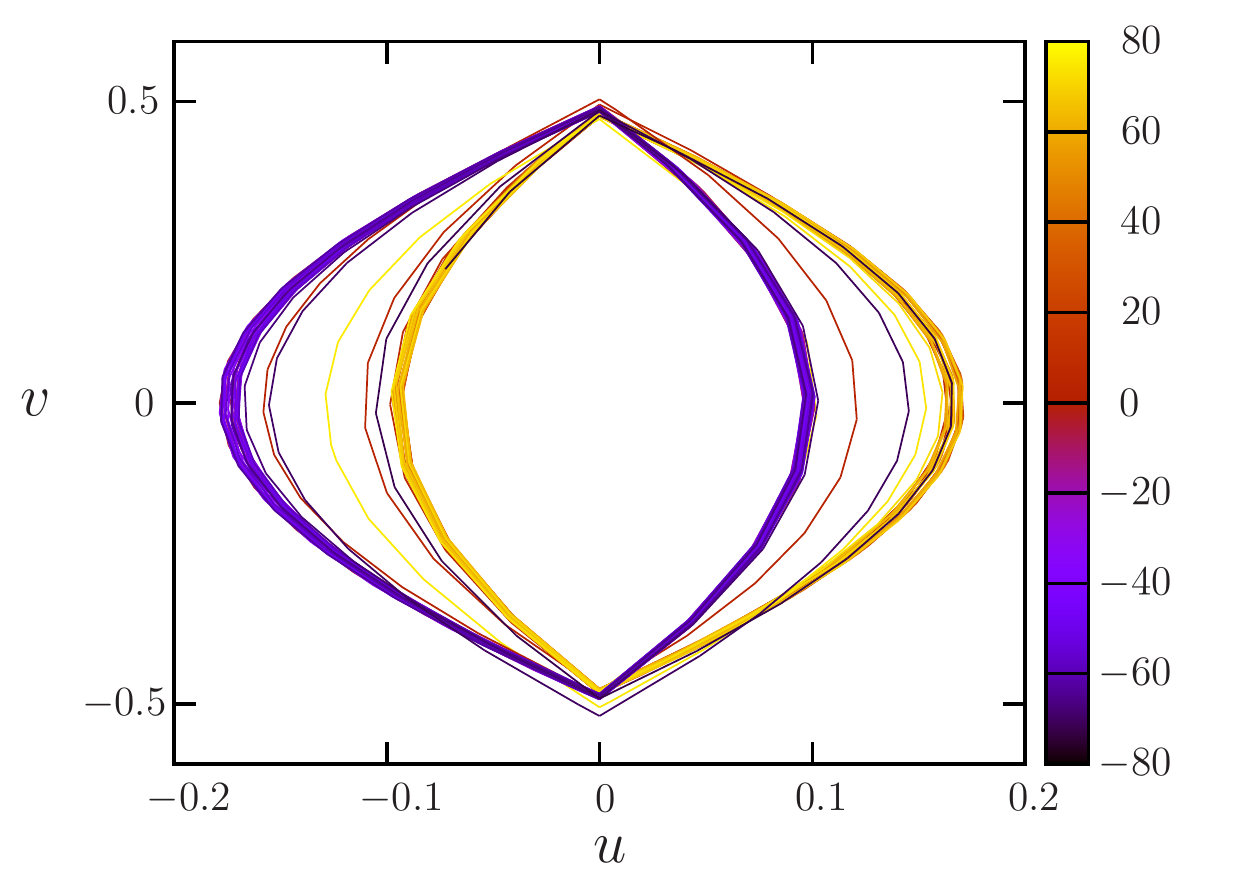}
}}
\end{picture}
\end{center}
\caption{Same as in figure~\ref{fig:trajectory_third_zero} for the fourth zero of
the Melnikov function in figure~\ref{fig:melnikov_v0d48_tau0}.}
\label{fig:trajectory_fourth_zero}
\end{figure}

In order to validate~\eqref{eq:average_third_zero}
and~\eqref{eq:average_fourth_zero}, we show in figure~\ref{fig:Uphi} the
Hamiltonian $U$ evaluated along the trajectories. Note that the
transition from $\tL_\varepsilon^-$ to $\tL_\varepsilon^+$ is very fast and
the trajectories spend most of the time close to the invariant manifolds until
they escape, both forwards and backwards in time.
In the same figure, we show the average functions
\begin{equation}
\frac{1}{t}\int_0^tU\left( \tp\left( r;\tz^*;\varepsilon \right) \right)dr.
\label{eq:average_function}
\end{equation}
The difference between the limiting values of the averages is shown magnified in
figure~\ref{fig:averages} for $t\to\infty$ and $t\to-\infty$. There is good agreement with the values
given in~\eqref{eq:average_third_zero} and~\eqref{eq:average_fourth_zero},
multiplied by $\varepsilon$.
\begin{figure}
\begin{center}
\begin{picture}(1, 0.4)
\put(0,0.35){
\subfigure[]
{\includegraphics[angle=-90,width=.5\textwidth]
{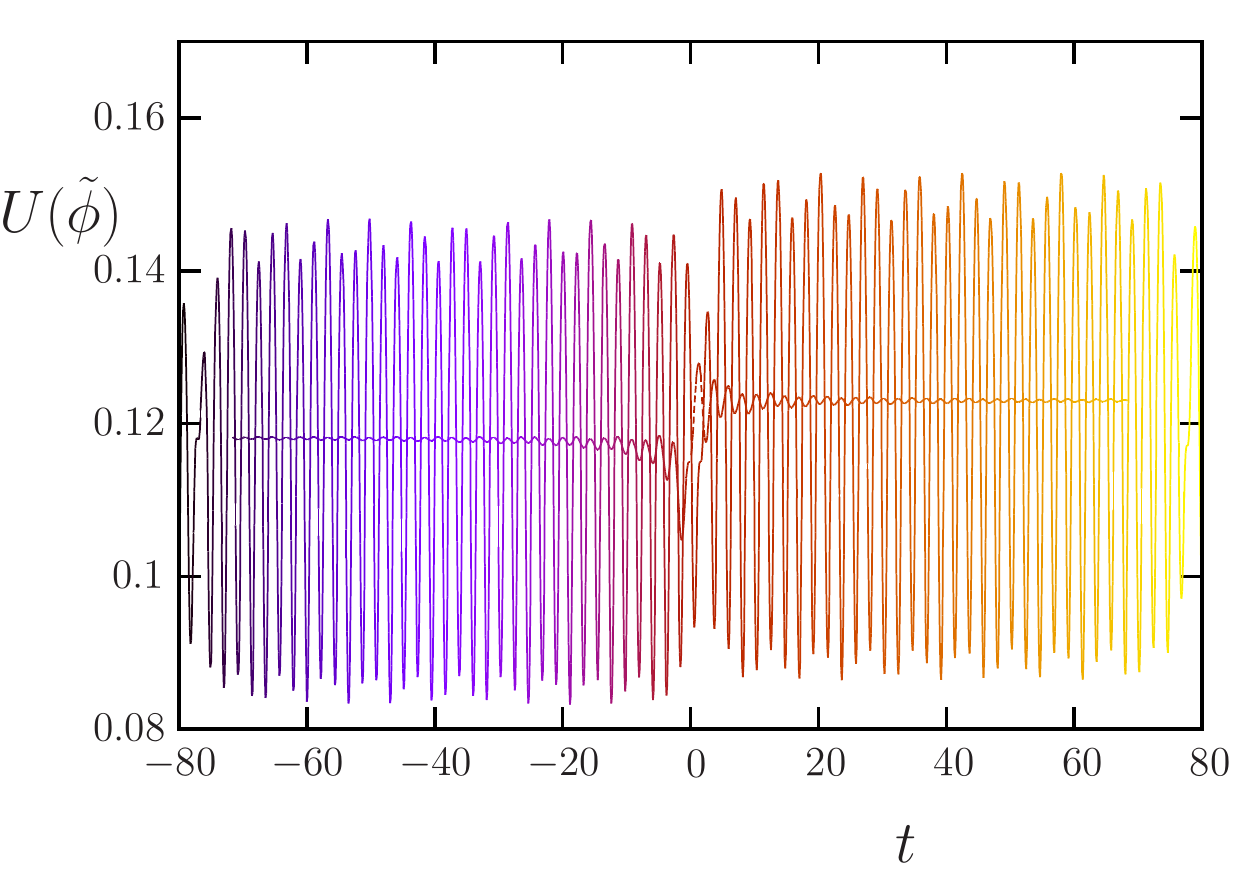}}
}
\put(0.5,0.35){
\subfigure[]
{\includegraphics[angle=-90,width=.5\textwidth]
{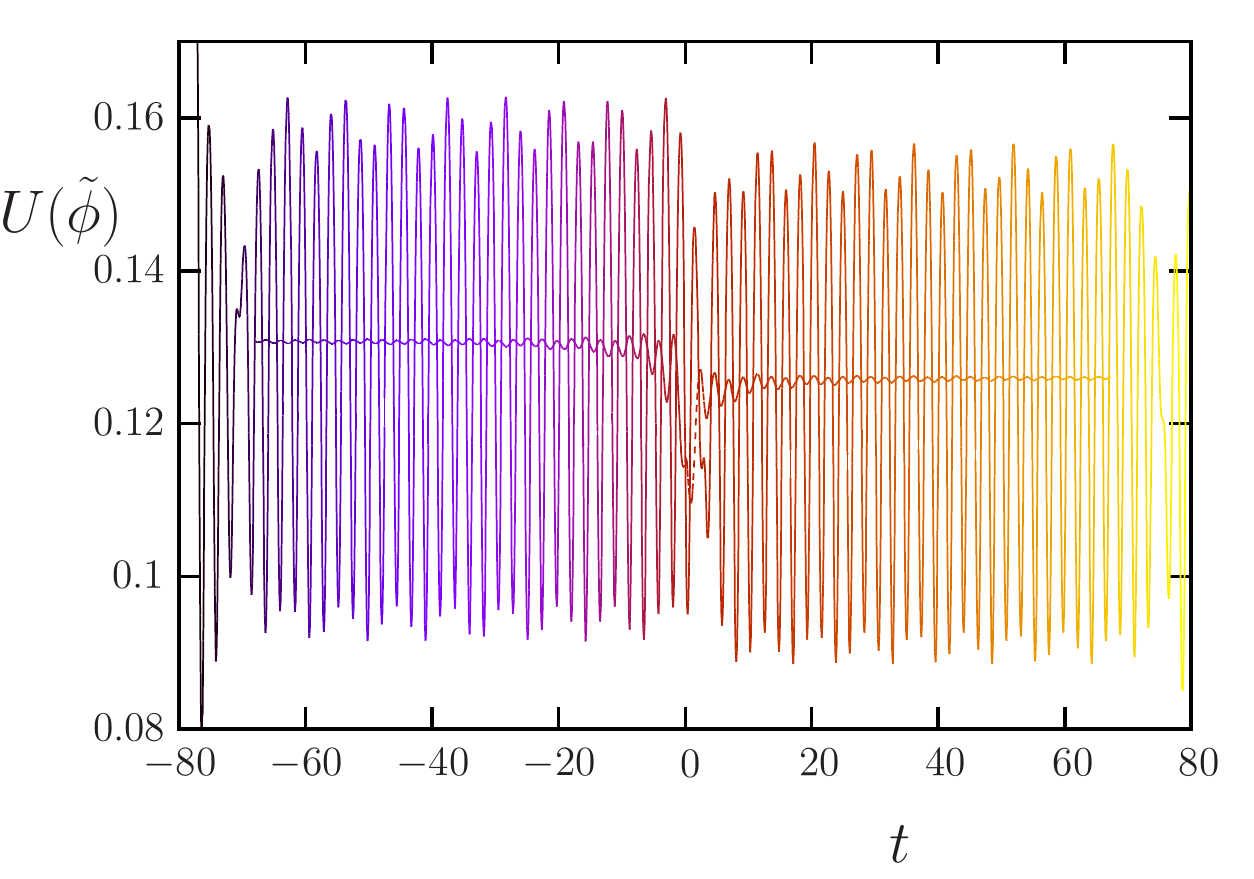}}
}
\end{picture}
\end{center}
\caption{Hamiltonian $U$ evaluated along the trajectory
$\tp(t;\tz^*;\varepsilon)$ for (a) the third and (b) the fourth zeros of the Melnikov
function in figure~\ref{fig:melnikov_v0d48_tau0}. The average
function~\eqref{eq:average_function} is also shown.}
\label{fig:Uphi}
\end{figure}

\begin{figure}
\begin{center}
\begin{picture}(1, 0.4)
\put(0,0.35){
\subfigure[]
{\includegraphics[angle=-90,width=.5\textwidth]
{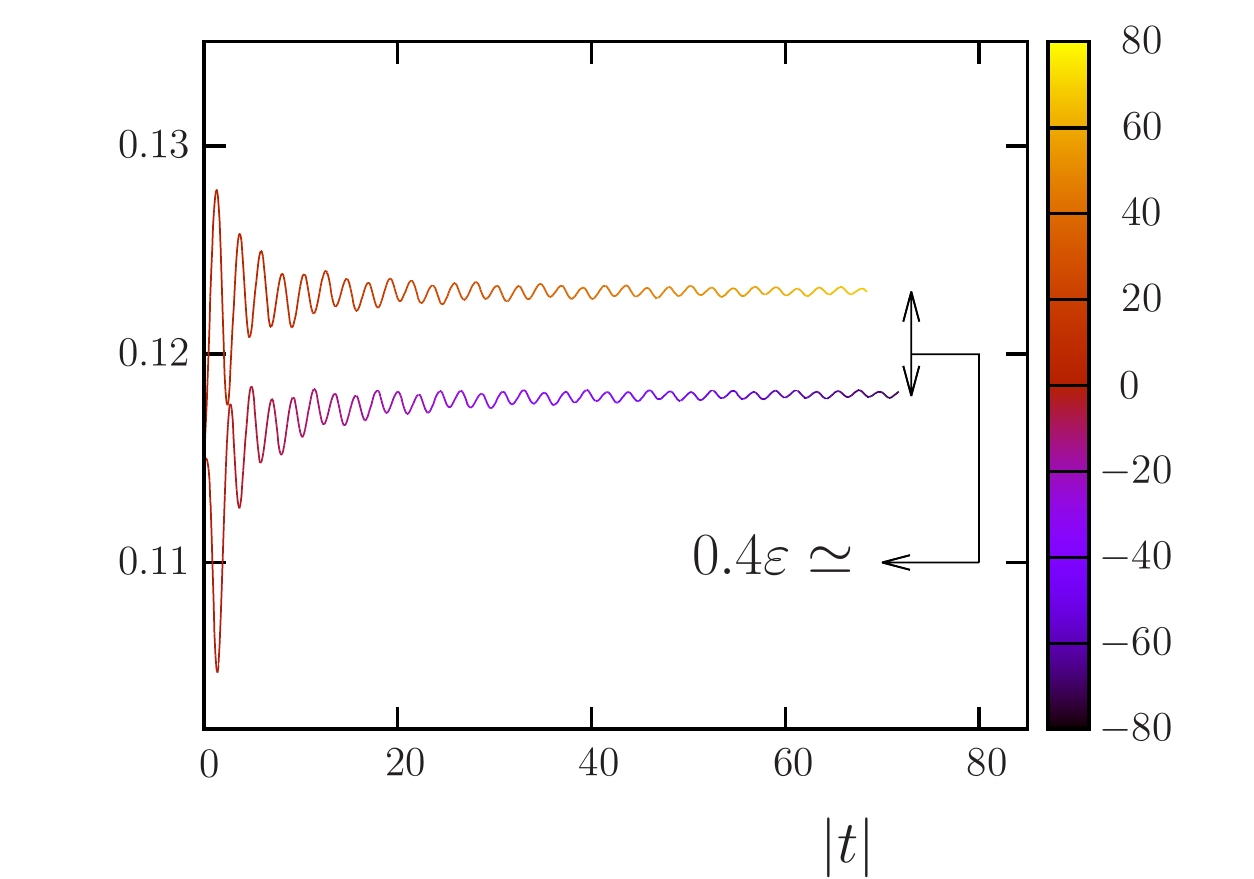}}
}
\put(0.5,0.35){
\subfigure[]
{\includegraphics[angle=-90,width=.5\textwidth]
{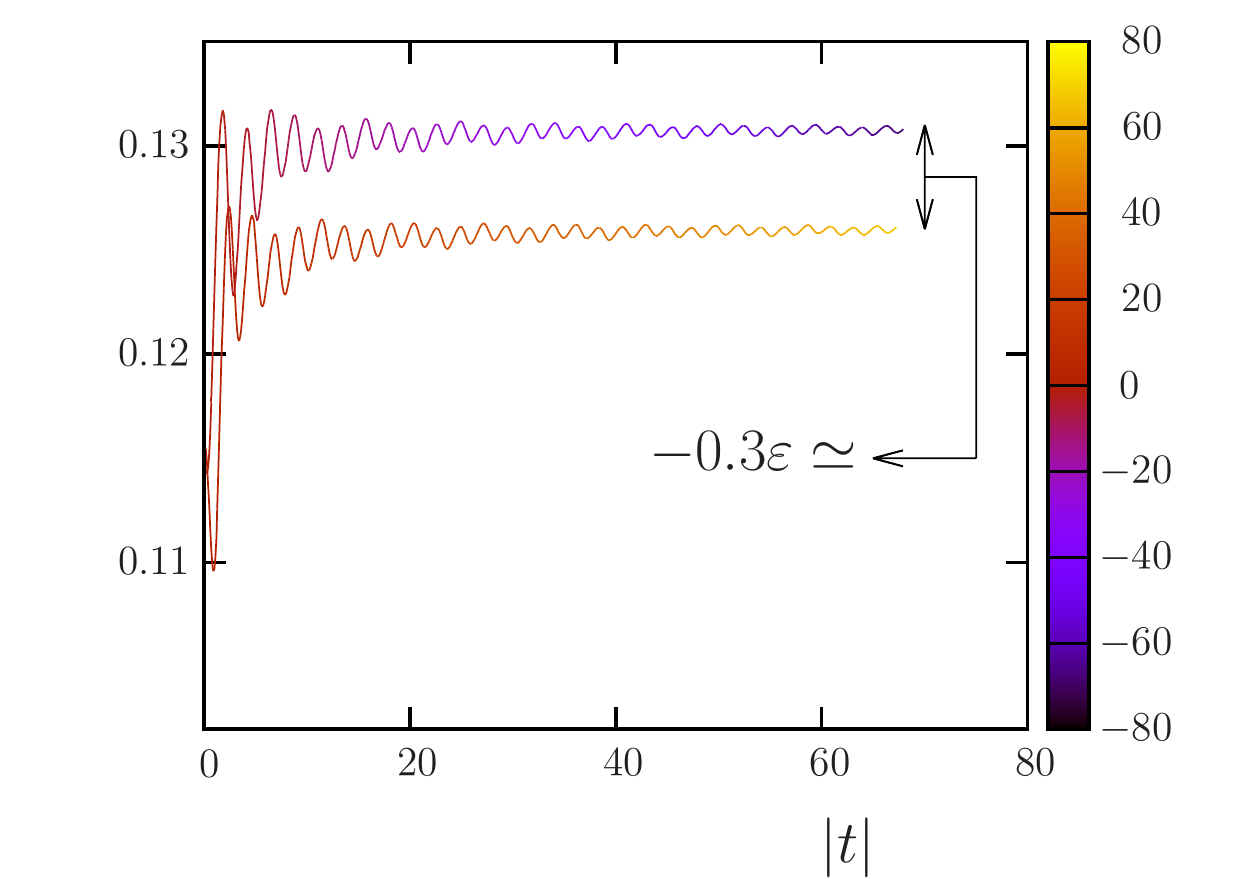}}
}
\end{picture}
\end{center}
\caption{Average function~\eqref{eq:average_function} for the trajectories shown
in figures~\ref{fig:trajectory_third_zero} and~\ref{fig:trajectory_fourth_zero}
(magnification of figure~\ref{fig:Uphi}). The colour bar denotes
$t$. Note that the horizontal axis denotes $|t|$, for better
comparison of the limiting values.}
\label{fig:averages}
\end{figure}
We now study the effect of varying $\theta$ whilst keeping $v$ and
$s$ constant. In figure~\ref{fig:averages_theta} we show the values (dotted)
of~\eqref{eq:average_formula_heteroclinic} for different values of
$\theta$, for the third and fourth zeros of the Melnikov function. In the same figure we
show the result of computing the heteroclinic point $\tz^*$ and proceeding as
before to compute the difference between the limiting averages of the asymptotic
dynamics. The agreement is good.
\begin{figure}
\begin{center}
\includegraphics[angle=-90,width=0.6\textwidth]
{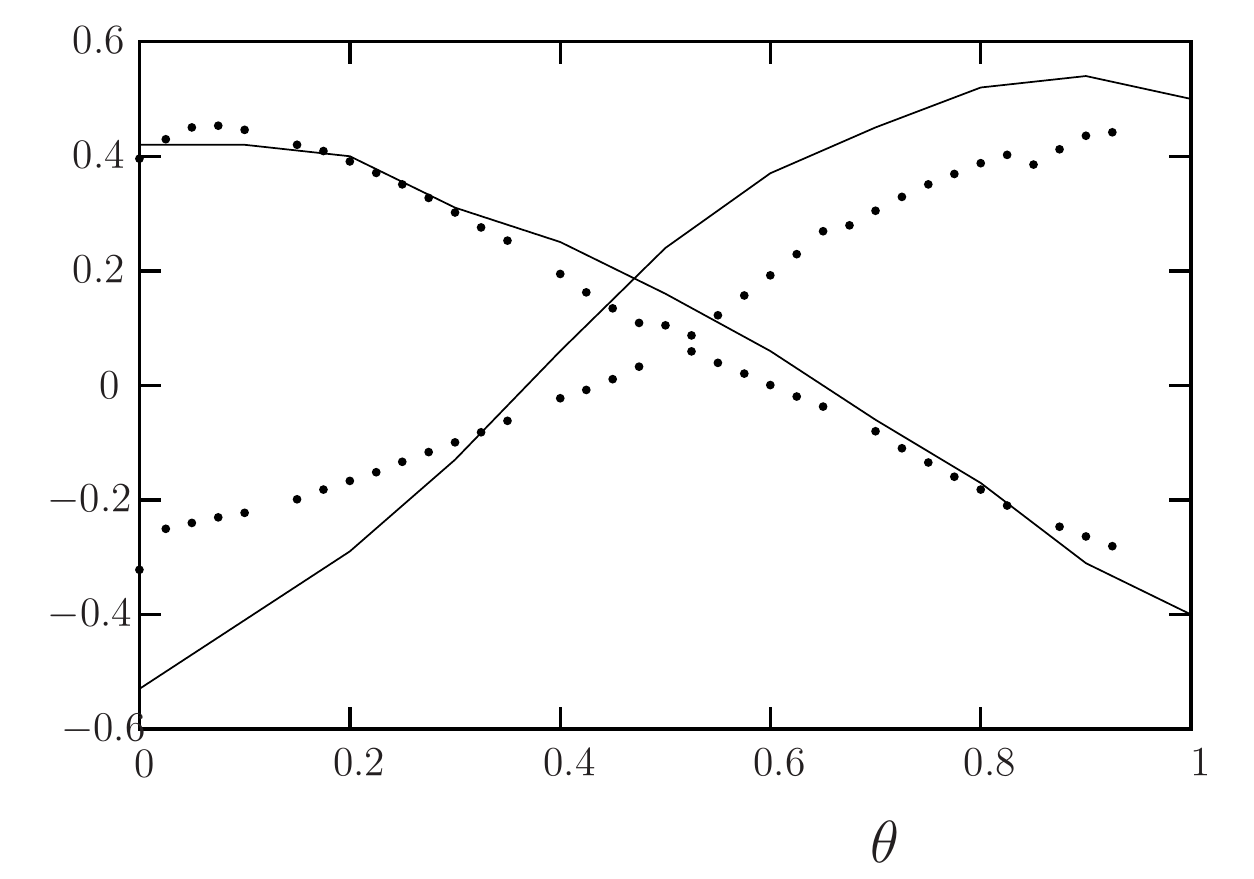}
\end{center}
\caption{Values (dotted) of~\eqref{eq:average_formula_heteroclinic} as a
function of $\theta$ for $v=0.48$ and $s=0$ for the third and fourth positive
zeros of the Melnikov function compared with the difference computed as in
figure~\ref{fig:averages} for $\varepsilon=0.01$ (lines).}
\label{fig:averages_theta}
\end{figure}

Finally, we study the first order difference given
in~\eqref{eq:average_formula_heteroclinic} when varying $v$ and $\theta$, whilst
keeping $s$ constant, for different zeros of the Melnikov function.  For each
$(\theta,v)$ we compute the Melnikov function, and for each zero we compute
expression~(\ref{eq:average_formula_heteroclinic}). The resulting values are
shown on the left of figures~\ref{fig:first_zeros}-\ref{fig:third_zeros} for the
first three positive zeros of the Melnikov function, which are shown on the
right of these figures. Note that, in figure~\ref{fig:first_zeros} (right),
there is a discontinuity curve (in black) corresponding to a relabelling of
zeros.  Positive values in the left-hand figures lead to an increase of energy
in one iteration of the scattering map $\Su$, while negative ones lead to a
decrease.  For the same coordinates $(\theta,v)$, different zeros of the
Melnikov function have different behaviours. When combining this with the same
study for the map $\Sd$ associated with the lower heteroclinic connection, these
results can be used to find suitable candidate trajectories exhibiting
diffusion.

\begin{figure}
\begin{center}
\unitlength=0.5\textwidth
\begin{picture}(0.9, 0.6)
\put(-0.6,0.6){
\includegraphics[angle=-90,width=.5\textwidth]
{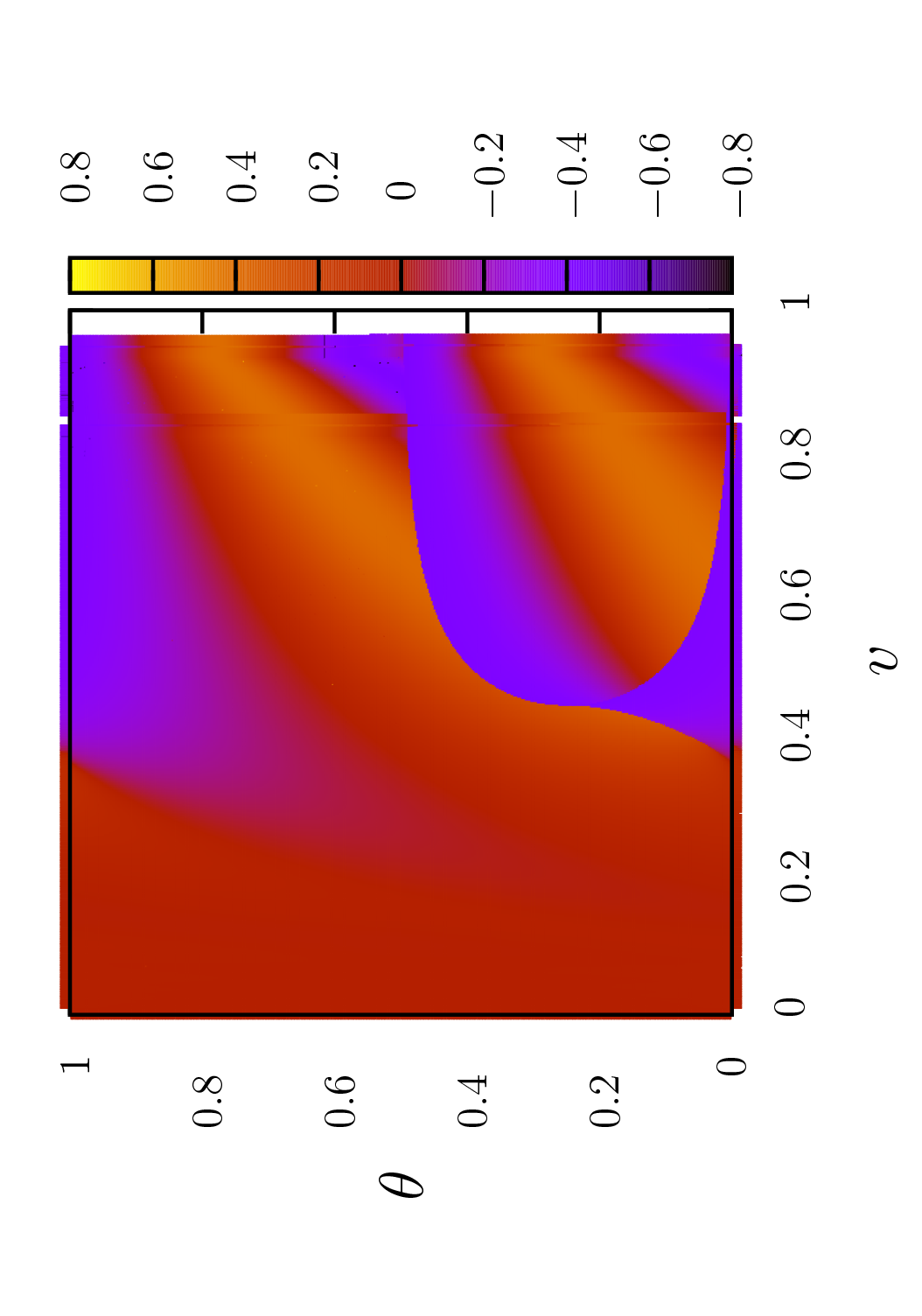}
}
\put(0.45,0.6){
\includegraphics[angle=-90,width=.5\textwidth]
{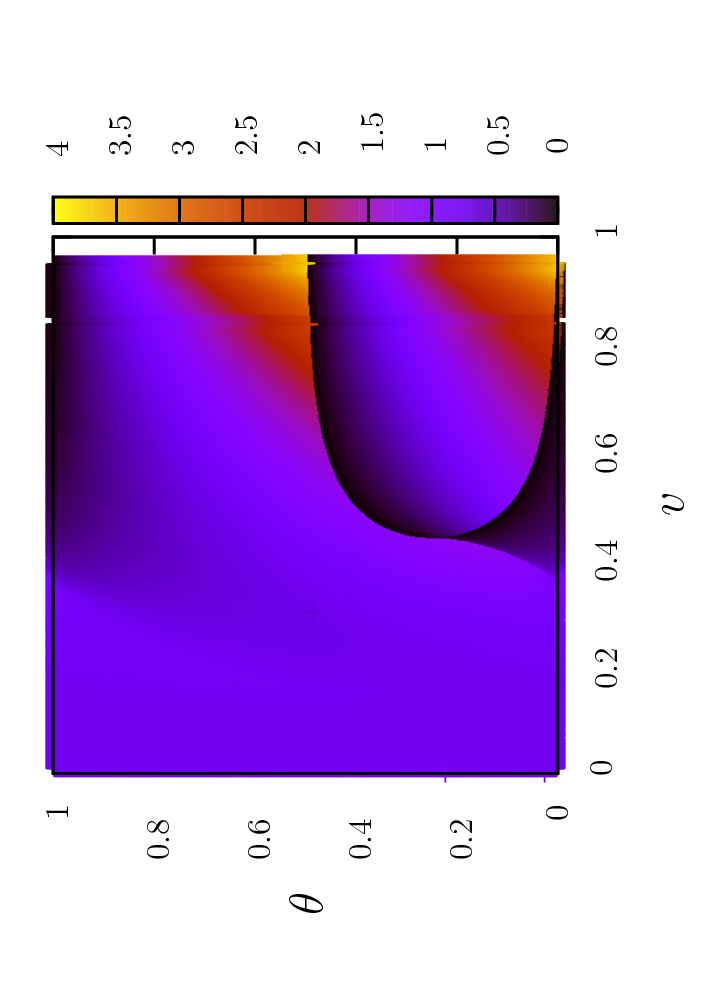}
}
\end{picture}
\end{center}
\caption{(left) First order difference between the average energy of the
trajectories $\tp(\pm t;\tz^\pm;\varepsilon)$ when $t\to\infty$ for the first
positive zero of the Melnikov function. (right) First positive zeros of the
Melnikov function.}
\label{fig:first_zeros}
\end{figure}

\begin{figure}
\begin{center}
\unitlength=0.5\textwidth
\begin{picture}(0.9, 0.6)
\put(-0.6,0.6){
\includegraphics[angle=-90,width=.5\textwidth]
{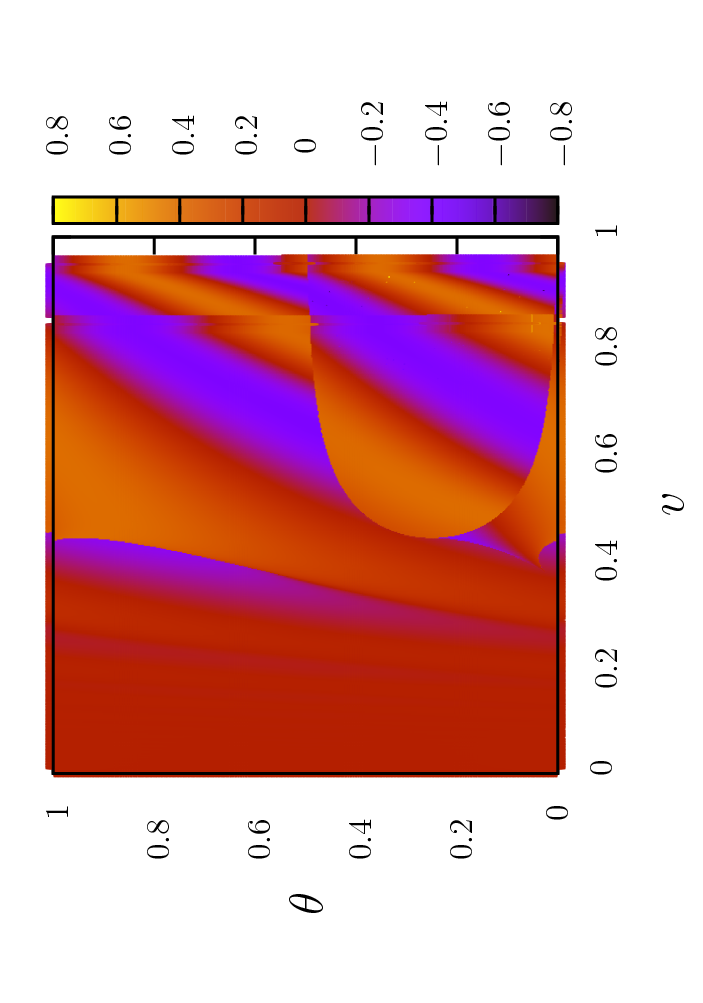}
}
\put(0.45,0.6){
\includegraphics[angle=-90,width=.5\textwidth]
{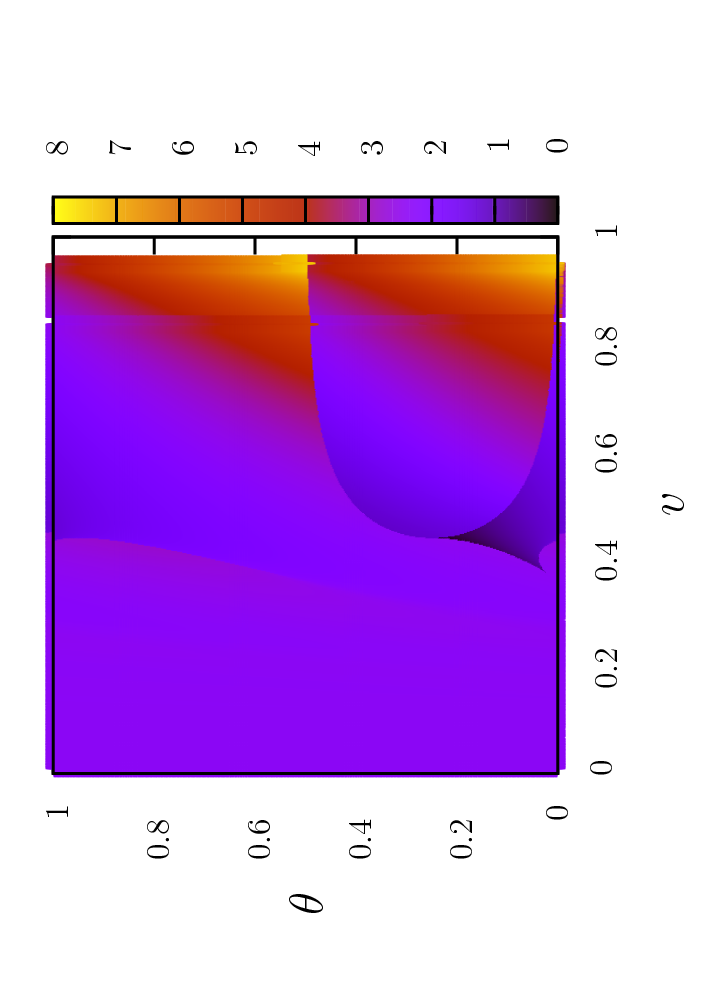}
}
\end{picture}
\end{center}
\caption{Same as figure~\ref{fig:first_zeros} for the second zero of the Melnikov
function.}
\label{fig:second_zeros}
\end{figure}

\begin{figure}
\begin{center}
\unitlength=0.5\textwidth
\begin{picture}(0.9, 0.6)
\put(-0.6,0.6){
\includegraphics[angle=-90,width=.5\textwidth]
{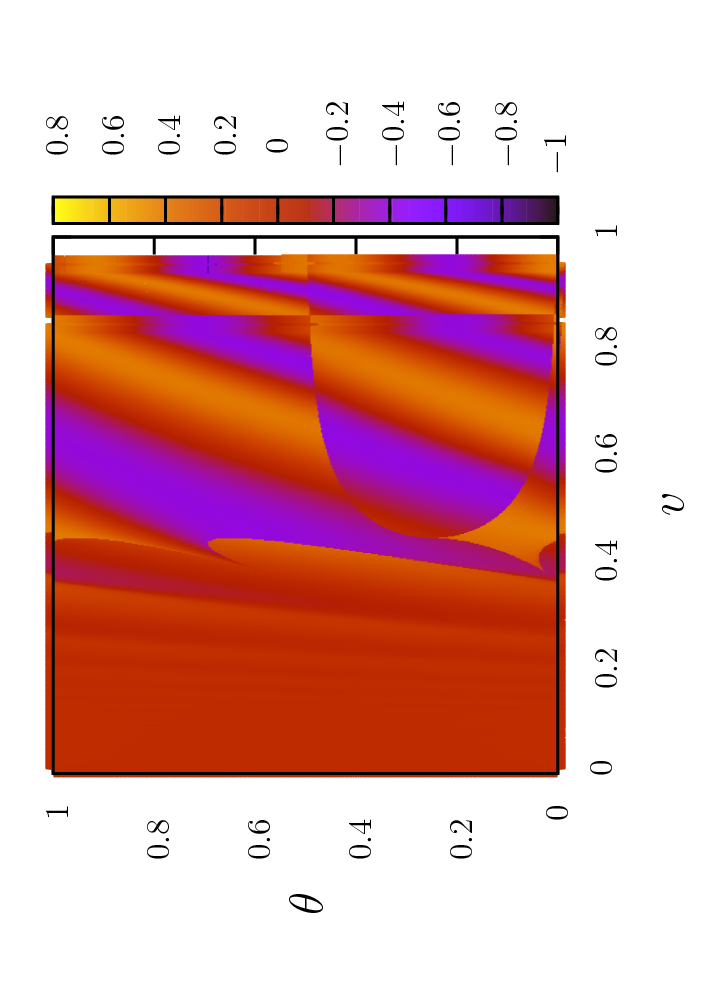}
}
\put(0.45,0.6){
\includegraphics[angle=-90,width=.5\textwidth]
{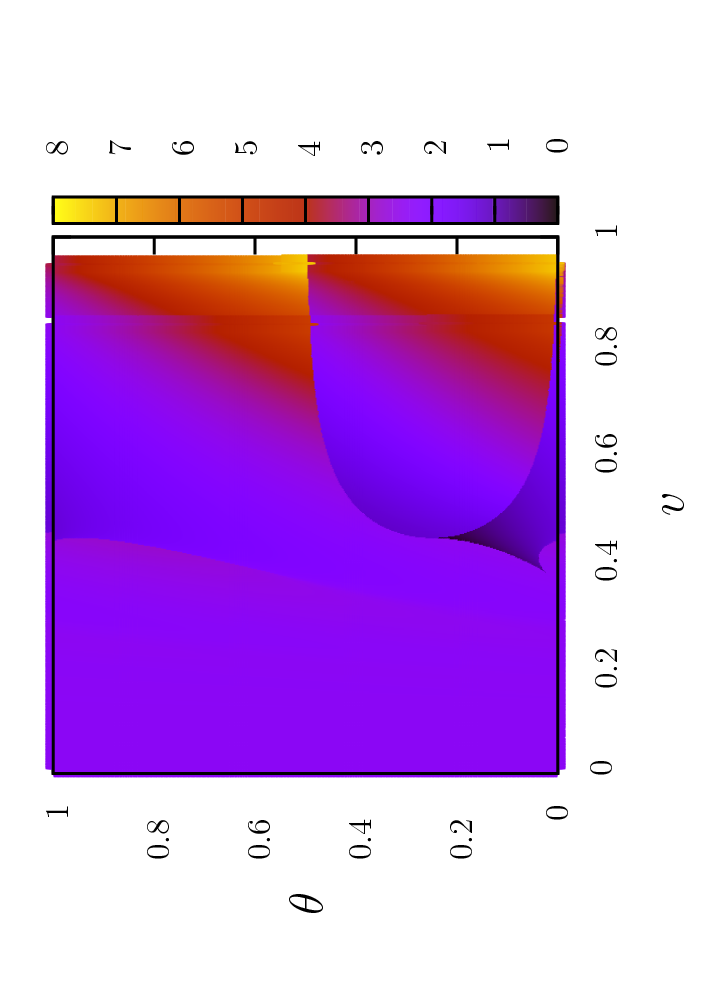}
}
\end{picture}
\end{center}
\caption{Same as figure~\ref{fig:first_zeros} for the third zero of the Melnikov
function.}
\label{fig:third_zeros}
\end{figure}

\section{Conclusions}\label{sec:conclusions}
We have considered a non-autonomous dynamical system formed by coupling
two piecewise-smooth systems in $\RR^2$ through a non-autonomous periodic
perturbation, leading to a two and a half degrees of
freedom piecewise-smooth Hamiltonian system with two switching manifolds.\\
We have studied the dynamics around one of the heteroclinic orbits of one of the piecewise-smooth systems, which is captured by $3$-dimensional invariant manifolds with stable and unstable
manifolds. In the unperturbed case, these stable and unstable manifolds
coincide, leading to the existence of two $4$-dimensional heteroclinic
manifolds connecting the two invariant manifolds. These heteroclinic manifolds are
foliated by heteroclinic connections between $C^0$ tori located at the same levels of
energy in both invariant manifolds.\\
By means of the {\em impact map} we have proved the persistence of these
objects under perturbation. In addition, we have provided
sufficient conditions for the existence of transversal heteroclinic intersections
through the existence of simple zeros of Melnikov-like functions, thereby extending 
some of the results given in~\cite{DelLlaSea06}.\\
These heteroclinic manifolds allow us to define the {\em scattering map},
which links asymptotic dynamics in the invariant manifolds through heteroclinic
connections. First order properties of this map provide sufficient conditions for the 
asymptotic dynamics to be located in different energy levels in
the perturbed invariant manifolds. Hence this is an essential tool for the 
construction of a heteroclinic skeleton which, when followed, can lead to the
existence of Arnol'd diffusion: trajectories that, on large time scales,
destabilize the system by further accumulating energy.\\
Finally we have validated all the theoretical results in this paper with
detailed numerical computations of a mechanical system with impacts, formed by
the linkage of two rocking blocks with a spring.  Future work should include the
study of the concatenation of the scattering map in order to construct diffusion
trajectories.


\bibliographystyle{model-num-names}
\def\zh{Zh}\def\yu{Yu}\def\ya{Ya}

\end{document}